\documentclass[12pt,a4paper,fleqn]{article}
\usepackage{a4wide,amsfonts,amsmath,latexsym,amssymb,euscript,graphicx,units,mathrsfs}

\usepackage[english]{babel}
\usepackage{graphicx}
\usepackage{color}
\usepackage{amssymb}
\usepackage{amssymb}
\usepackage[T1]{fontenc}
\usepackage{latexsym}
\usepackage{xypic}
\usepackage{eufrak}
\usepackage{euscript}
\usepackage{amsfonts,amsmath}
\usepackage{verbatim}
\usepackage{fancyhdr}
\usepackage{mathrsfs}
\usepackage{units}

\newtheorem{prop}{Proposition}[section]

\newtheorem{corollary}[prop]{Corollary}
\newtheorem{lemme}[prop]{Lemma}
\newtheorem{lemma}[prop]{Lemma}

\newtheorem{remark}[prop]{Remark}

\newtheorem{theorem}[prop]{Theorem}

\renewcommand{\geq}{\geqslant}
\def\leq{\leqslant}

\newcommand{\N}{\mathbb{N}}
\newcommand{\Z}{\mathbb{Z}}

\newcommand{\R}{\mathbb{R}}

\def\cal{\mathcal}
\def\1{{\mathbf{1}}}

\def\1{{\mathbf{1}}}
\def\0.5{{\frac{1}{2}}}

\newenvironment{proof}[1]{\begin{trivlist}\item {\it
\bf Proof.}\quad} {\qed\end{trivlist}}

\newcommand{\qed}{\nopagebreak\hspace*{\fill}
{\vrule width6pt height6ptdepth0pt}\par}

\begin{document}
\thispagestyle{empty}

\begin{center}
{\bf\Large Asymptotic behavior of  weighted power variations of  fractional Brownian motion in Brownian time }
\end{center}
\bigskip
\begin{center}
{\bf Raghid Zeineddine\footnote{Research Training Group 2131, Fakult\"at Mathematik, Technische Universit\"at Dortmund, Germany. \\Email: raghid.z@hotmail.com}}
\end{center}

\bigskip
\begin{abstract}
We study the asymptotic behavior of  weighted power variations of  fractional Brownian motion in Brownian time $Z_t:=  X_{Y_t}, \:\: t \geq 0$, where $X$ is a fractional Brownian motion  and $Y$ is an independent Brownian motion. 
\end{abstract}

\section{Introduction}
Our aim in this paper is to study the asymptotic behavior of  weighted power variations of the so-called {\it fractional Brownian motion in Brownian time} defined as
\[
Z_t=X_{Y_t}, \quad t\geq 0,
\]
where $X$ is a two-sided fractional Brownian motion, with Hurst parameter $H \in (0,1)$, and $Y$ is a standard (one-sided) Brownian
motion independent of $X$. It is a self-similar process (of order $H/2$) with stationary increments, which is  not Gaussian.
When $H=1/2$, one recovers the celebrated {\it iterated Brownian motion.}

In the present paper we follow, and we are inspired by the previous papers \cite{kh-lewis1}; \cite{NNT}; \cite{NP}; \cite{RZ1}, and our work may be seen a natural follow-up of \cite{NP} and \cite{RZ1}.
  
Let $f:\R \to \R$ be a function belonging to $C_b^\infty$, the class of those functions that are $C^\infty$ and bounded together with their derivatives. Then, for any $t\geq 0$ and any integer $p\geq 1$, the  weighted $p$-variation of $Z$ is defined as 
\begin{eqnarray*}
R_n^{(p)}(t)=\sum_{k=0}^{\lfloor 2^n t\rfloor-1}\frac12\big(f(Z_{k2^{-n}})+f(Z_{(k+1)2^{-n}})\big)(Z_{(k+1)2^{-n}}- Z_{k2^{-n}})^p. 
\end{eqnarray*} 
 After proper normalization we may expect the  convergence (in some sense) to a non-degenerate limit (to be determined) of
\begin{eqnarray}
S_n^{(p)}(t)&=&2^{n\kappa}\sum_{k=0}^{\lfloor 2^n t\rfloor-1}\frac12\big(f(Z_{k2^{-n}})+f(Z_{(k+1)2^{-n}})\big)\big[(Z_{(k+1)2^{-n}}- Z_{k2^{-n}})^p - \notag\\
&& \hspace{7cm}E[(Z_{(k+1)2^{-n}}- Z_{k2^{-n}})^p]\big], \label{S}
\end{eqnarray}
for some $\kappa$ to be discovered. Due to the fact that one cannot separate $X$ from $Y$ inside $Z$ in the definition
of $S_n^{(p)}$, working directly with (\ref{S}) seems to be a  difficult task (see also \cite[ Problem 5.1]{kh-lewis2}).
This is why, following an idea introduced by Khoshnevisan and Lewis \cite{kh-lewis1} in a study of the case $H=1/2$,
we will rather analyze $S_n^{(p)}$ by means of certain stopping times for $Y$. The idea is: by stopping $Y$ as it crosses certain levels, and by sampling $Z$ at these times, one can effectively separate $X$ from $Y$.
To be more specific, let us introduce the following collection of stopping times (with
respect to the natural filtration of $Y$), noted
\begin{equation}
\mathscr{T}_n=\{T_{k,n}: k\geq0\}, \quad n\geq 0, \label{T-N}
\end{equation}
which are in turn expressed in terms of the subsequent hitting
times of a dyadic grid cast on the real axis. More precisely, let
$\mathscr{D}_n= \{j2^{-n/2}:\,j\in\Z\}$, $n\geq 0$, be the dyadic
partition (of $\R$) of order $n/2$. For every $n\geq 0$, the
stopping times $T_{k,n}$, appearing in (\ref{T-N}), are given by
the following recursive definition: $T_{0,n}= 0$, and
\[
T_{k,n}= \inf\big\{s>T_{k-1,n}:\quad
Y(s)\in\mathscr{D}_n\setminus\{Y_{T_{k-1,n}}\}\big\},\quad k\geq 1.
\]
Note that the definition of $T_{k,n}$, and
therefore of $\mathscr{T}_n$, only involves the one-sided Brownian
motion $Y$, and that, for every $n\geq 0$, the discrete stochastic
process
\[
\mathscr{Y}_n=\{Y_{T_{k,n}}:k\geq0\}
\]
defines a simple and symmetric random
walk over $\mathscr{D}_n$.  As shown in
\cite{kh-lewis1}, as $n$ tends to
infinity the collection $\{T_{k,n}:\,1\leq k \leq 2^nt\}$ approximates the
common dyadic partition $\{k2^{-n}:\,1\leq k \leq 2^nt\}$ of order $n$ of the time interval $[0,t]$ (see
\cite[Lemma 2.2]{kh-lewis1} for a precise statement).
Based on this fact, one can introduce the counterpart of (\ref{S}) based on $\mathscr{T}_n$, namely,
\begin{equation*}
\tilde{S}_n^{(p)}(t)=2^{-n\tilde{\kappa}}\sum_{k=0}^{\lfloor 2^n t\rfloor-1}\frac12\big(f(Z_{T_{k,n}})+f(Z_{T_{k+1,n}})\big)\big[\big(2^{\frac{nH}{2}}(Z_{T_{k+1,n}}- Z_{T_{k,n}})\big)^p - \mu_p\big], 
\end{equation*}
for some $\tilde{\kappa}>0$ to be discovered and with $\mu_p:= E[N^p]$, where $N \sim \mathcal{N}(0,1)$. At this stage, it is worthwhile noting that we are dealing with symmetric weighted $p$-variation of $Z$, and symmetry will play an important role in our analysis as we will see in Lemma \ref{lemme-algebric}. 

In the particular case where $H=\frac12$, that is when $Z$ is the {\it iterated Brownian motion}, the asymptotic behavior of $\tilde{S}_n^{(p)}(\cdot) $ has been studied in \cite{NP}. In fact, one can deduce the following two finite dimensional distributions (f.d.d.) convergences in law from \cite[Theorem 1.2]{NP}.
\begin{itemize}
\item[1)] For $f\in C_b^2$ and for any integer $r \geq 1$, we have
\end{itemize}
\begin{eqnarray}
&&\bigg(2^{-\frac{3n}{4}}\sum_{k=0}^{\lfloor 2^n t \rfloor -1} \frac12\big(f(Z_{T_{k,n}})+f(Z_{T_{k+1,n}})\big)\big[\big(2^{\frac{n}{4}}(Z_{T_{k+1,n}}-Z_{T_{k,n}})\big)^{2r} - \mu_{2r}\big]  \bigg)_{ t \geq 0} \notag\\
&&\underset{n\to \infty}{\overset{\rm f.d.d.}{\longrightarrow}} \bigg( \sqrt{\mu_{4r}-\mu_{2r}^2}\int_{-\infty}^{+\infty} f(X_s)L_t^s(Y)dW_s \bigg)_{ t\geq 0},\label{new-NP1}
\end{eqnarray}
where $L_t^s(Y)$ stands for the local time of $Y$ before time $t$ at level $s$, $W$ is a two-sided Brownian motion independent of $(X,Y)$ and  $\int_{-\infty}^{+\infty} f(X_s)L_t^s(Y)dW_s$ is the Wiener-It\^o integral of $f(X_{\cdot})L^{\cdot}_t(Y)$ with respect to $W$.
\begin{itemize}
\item[2)] For $f\in C_b^2$ and for any integer $r \geq 2$,  we have
\end{itemize}
\begin{eqnarray}
&&\bigg( 2^{-\frac{n}{4}}\sum_{k=0}^{\lfloor 2^n t \rfloor -1} \frac12\big(f(Z_{T_{k,n}})+f(Z_{T_{k+1,n}})\big)\big(2^{\frac{n}{4}}(Z_{T_{k+1,n}}-Z_{T_{k,n}})\big)^{2r-1}  \bigg)_{t \geq 0} \notag\\
&&\underset{n\to \infty}{\overset{\rm f.d.d.}{\longrightarrow}} \bigg( \int_0^{Y_t} f(X_s)(\mu_{2r} d^{\circ}X_s + \sqrt{\mu_{4r-2} - \mu_{2r}^2}\,dW_s \bigg)_{t \geq 0}, \label{new-NP2}
\end{eqnarray}
where for all $t \in \R$, $\int_0^t f(X_s)d^{\circ}X_s$ is the Stratonovich integral of $f(X)$ with respect to $X$ defined as the limit in probability of $2^{-\frac{nH}{2}}W_{n}^{(1)}(f,t)$ as $n \to \infty$, with $W_{n}^{(1)}(f,t)$ defined in (\ref{Wn}), $W$ is a two-sided Brownian motion independent of $(X,Y)$ and for $u \in \R$, $\int_0^{u} f(X_s)dW_s$ is the Wiener-It\^o integral of $f(X)$ with respect to $W$ defined in (\ref{integrale}).

A natural follow-up of (\ref{new-NP1}) and (\ref{new-NP2}) is to study the asymptotic behavior of $\tilde{S}_n^{(p)}(\cdot)$ when $H \neq \frac12$. In fact, the following more general result is our main finding in the present paper.

\begin{theorem}\label{thm1} Let $f: \R \to \R$ be a function belonging to $C_b^{\infty}$ and let $W$ denote a two-sided  Brownian motion independent of $(X,Y)$.
\begin{itemize}
\item[(1)]For $H> \frac16$, we have
\begin{equation}
\sum_{k=0}^{\lfloor 2^n t \rfloor -1} \frac12\big(f(Z_{T_{k,n}})+f(Z_{T_{k+1,n}})\big)(Z_{T_{k+1,n}}-Z_{T_{k,n}})\underset{n\to \infty}{\overset{P}{\longrightarrow}} \int_0^{Y_t} f(X_s)d^{\circ}X_s, \label{H >1/6 , r=1}
\end{equation}
where for all $t \in \R$, $\int_0^t f(X_s)d^{\circ}X_s$ is the Stratonovich integral of $f(X)$ with respect to $X$ defined as the limit in probability of $2^{-\frac{nH}{2}}W_{n}^{(1)}(f,t)$ as $n \to \infty$, with $W_{n}^{(1)}(f,t)$ defined in (\ref{Wn}).\\
For $H= \frac16$, we have
\begin{equation}\label{H = 1/6 , r=1}
\sum_{k=0}^{\lfloor 2^n t \rfloor -1} \frac12\big(f(Z_{T_{k,n}})+f(Z_{T_{k+1,n}})\big)(Z_{T_{k+1,n}}-Z_{T_{k,n}})\underset{n\to \infty}{\overset{law}{\longrightarrow}} \int_0^{Y_t} f(X_s)d^{*}X_s,
\end{equation}
where for all $t \in \R$, $\int_0^t f(X_s)d^{*}X_s$ is the Stratonovich integral of $f(X)$ with respect to $X$ defined as the limit in law of $2^{-\frac{nH}{2}}W_{n}^{(1)}(f,t)$ as $n \to \infty$.
\end{itemize}
\begin{itemize}
\item[(2)] For $\frac16 < H < \frac12$ and for any integer $r\geq 2$, we have
\end{itemize}
\begin{eqnarray}
&&\bigg( 2^{-\frac{n}{4}}\sum_{k=0}^{\lfloor 2^n t \rfloor -1} \frac12\big(f(Z_{T_{k,n}})+f(Z_{T_{k+1,n}})\big)\big(2^{\frac{nH}{2}}(Z_{T_{k+1,n}}-Z_{T_{k,n}})\big)^{2r-1}  \bigg)_{t \geq 0} \notag\\
&&\underset{n\to \infty}{\overset{\rm f.d.d.}{\longrightarrow}} \bigg( \beta_{2r-1}\int_0^{Y_t} f(X_s)dW_s \bigg)_{t\geq 0}, \label{fdd thm1}
\end{eqnarray}
where  for $u \in \R$, $\int_0^{u} f(X_s)dW_s$ is the Wiener-It\^o integral of $f(X)$ with respect to $W$ defined in (\ref{integrale}), $\beta_{2r-1} = \sqrt{\sum_{l=2}^r\kappa^2_{r,l}\,\alpha^2_{2l-1}}$,  with $\alpha_{2l-1}$  defined in (\ref{alpha}) and $\kappa_{r,l}$  defined in (\ref{hermite polynomial}).
\begin{itemize}
\item[(3)]  Fix a time $t \geq 0$, for $H>\frac12$ and for any integer $r \geq 1$, we have
\end{itemize}
\begin{eqnarray}
2^{-\frac{nH}{2}}\sum_{k=0}^{\lfloor 2^n t \rfloor -1} \frac12\big(f(Z_{T_{k,n}})+f(Z_{T_{k+1,n}})\big)\big(2^{\frac{nH}{2}}(Z_{T_{k+1,n}}-Z_{T_{k,n}})\big)^{2r-1} \underset{n\to \infty}{\overset{L^2}{\longrightarrow}}  \frac{(2r)!}{r!2^r}\int_0^{Y_t}f(X_s)d^{\circ}X_s,\notag\\ \label{L2 thm1}
\end{eqnarray}
where for all $t \in \R$, $\int_0^t f(X_s)d^{\circ}X_s$ is defined as in (\ref{H >1/6 , r=1}). 
\begin{itemize}
\item[(4)] For $\frac14 < H \leq \frac12$ and for any integer $r\geq 1$,  we have
\end{itemize}
\begin{eqnarray}
&&\bigg(2^{-\frac{3n}{4}}\sum_{k=0}^{\lfloor 2^n t \rfloor -1} \frac12\big(f(Z_{T_{k,n}})+f(Z_{T_{k+1,n}})\big)\big[\big(2^{\frac{nH}{2}}(Z_{T_{k+1,n}}-Z_{T_{k,n}})\big)^{2r} - \mu_{2r}\big]  \bigg)_{ t \geq 0} \notag\\
&&\underset{n\to \infty}{\overset{\rm f.d.d.}{\longrightarrow}} \bigg( \gamma_{2r}\int_{-\infty}^{+\infty} f(X_s)L_t^s(Y)dW_s \bigg)_{ t\geq 0}, \label{fdd thm2}
\end{eqnarray}
where $\int_{-\infty}^{+\infty} f(X_s)L_t^s(Y)dW_s$ is the Wiener-It\^o integral of $f(X_{\cdot})L^{\cdot}_t(Y)$ with respect to $W$, $\gamma_{2r} := \sqrt{\sum_{a=1}^rb_{2r,a}^2\,\alpha^2_{2a}}$,  with $\alpha_{2a}$  defined in (\ref{alpha}) and $b_{2r,a}$  defined in (\ref{hermite polynomial pair}).
\end{theorem}
Theorem \ref{thm1} is also a natural follow-up  of \cite[Corollary 1.2]{RZ1} where we have studied the asymptotic behavior of the power variations of the fractional Brownian motion in Brownian time. In fact, taking $f$ equal to 1 in (\ref{L2 thm1}), we deduce the following Corollary. 
\begin{corollary} Assume that $H>\frac12$, for any $t\geq 0$ and any integer $r \geq 1$, we have
\begin{eqnarray*}
2^{-\frac{nH}{2}}\sum_{k=0}^{\lfloor 2^n t \rfloor -1} \big(2^{\frac{nH}{2}}(Z_{T_{k+1,n}}-Z_{T_{k,n}})\big)^{2r-1} \underset{n\to \infty}{\overset{L^2}{\longrightarrow}}  \frac{(2r)!}{r!2^r}\,Z_t,
\end{eqnarray*}
thus, we understand the asymptotic behavior of the {\it signed} power variations of odd order of the fractional Brownian motion in Brownian time, in the case $H>\frac12$, which was missing in the first point in \cite[Corollary 1.2]{RZ1}.
\end{corollary}

\begin{remark}
\begin{itemize}
\item[1.] For $H= \frac16$, it has been proved in  \cite[(3.17)]{RZ2} that
\end{itemize}
\[
\big(\sum_{k=0}^{\lfloor 2^n t \rfloor -1} \frac12\big(f(Z_{T_{k,n}})+f(Z_{T_{k+1,n}})\big)(Z_{T_{k+1,n}}-Z_{T_{k,n}})^3\big)_{t \geq 0}\underset{n\to \infty}{\overset{f.d.d.}{\longrightarrow}} \big(\kappa_3\int_0^{Y_t}f(X_s)dW_s\big)_{t\geq 0},
\]
with $W$ a standard two-sided Brownian motion independent of the pair $(X,Y)$ and $\kappa_3\simeq 2.322$. Thus, (\ref{fdd thm1}) continues to hold  for $H=\frac16$ and $r=2$.
\begin{itemize}
\item[2.] In the particular case where $H=1/2$ (that is, when $Z$ is the {\it iterated Brownian motion}), we emphasize that the fourth point of Theorem \ref{thm1} allows one to recover (\ref{new-NP1}). In fact, since $H=\frac12$, then, for any integer $a\geq 1$, by (\ref{alpha}) and its related explanation, $\alpha^2_{2a} = (2a)!$. So, using the decomposition (\ref{hermite polynomial pair}) and (\ref{Hp-Hq}), the reader can verify that $\sqrt{\mu_{4r} - \mu_{2r}^2}$ appearing in (\ref{new-NP1}) is equal to $\gamma_{2r}$ appearing in (\ref{fdd thm2}).
\end{itemize}
\begin{itemize}
\item[3.] The limit process in (\ref{new-NP2}) is $\bigg( \int_0^{Y_t} f(X_s)(\mu_{2r} d^{\circ}X_s + \sqrt{\mu_{4r-2} - \mu_{2r}^2}\,dW_s \bigg)_{t \geq 0} $. Observe that $\mu_{2r}=E[N^{2r}]= \frac{(2r)!}{r!2^r}$ and since $H=\frac12$, then, for any integer $l\geq 1$, by (\ref{alpha}) and its related explanation, $\alpha^2_{2l-1} = (2l-1)!$. So, using the decomposition (\ref{hermite polynomial}) and (\ref{Hp-Hq}), the reader can verify that $\sqrt{\mu_{4r-2} - \mu_{2r}^2}$ is equal to $\beta_{2r-1}$ appearing in (\ref{fdd thm1}). We deduce that the limit process in (\ref{new-NP2}) is $\bigg( \frac{(2r)!}{r!2^r}\int_0^{Y_t} f(X_s)d^{\circ}X_s + \beta_{2r-1}\int_0^{Y_t} f(X_s)dW_s \bigg)_{t \geq 0} $. Thus, one can say that,  for any integer $r\geq2$, the limit of the  weighted $(2r-1)$-variation of $Z$  for $H=\frac12$ is mixing between the limit of the  weighted $(2r-1)$-variation of $Z$ for $H> \frac12$  and the limit of the  weighted $(2r-1)$-variation of $Z$ for $\frac16<H<\frac12$.
\end{itemize}
\end{remark}

A brief outline of the paper is as follows. In section 2, we give the preliminaries  to the  proof of Theorem \ref{thm1}. In section 3, we start the preparation to our proof. In section 4, we prove (\ref{H >1/6 , r=1}) and (\ref{H = 1/6 , r=1}). In sections  5, 6 and 7 we prove (\ref{fdd thm1}), (\ref{L2 thm1}) and (\ref{fdd thm2}). Finally, in section 8, we give the proof of a  technical lemma.

\section{Preliminaries}

\subsection{Elements of Malliavin calculus}

 In this section, we gather some elements of Malliavin calculus we shall need in the sequel. The reader in referred to \cite{NP2} for details and any unexplained result.

We continue to denote by  $X = (X_{t})_{t \in \R}$ a two-sided fractional Brownian motion with Hurst parameter $ H \in (0,1).$ That is, $X$ is a zero mean Gaussian process, defined on a complete probability space $(\Omega, \mathscr{A}, P)$, with  covariance function, \[ C_{H}(t,s) = E(X_{t}X_{s})=\frac{1}{2}(|s|^{2H} + |t|^{2H} -|t-s|^{2H}),\text{\:\:\:} s,t \in \R.\]
   We suppose that $\mathscr{A}$ is the $\sigma$-field generated by $X$. For all $n \in \N^*$, we let $\mathscr{E}_n$ be the set of step functions on $[-n,n]$, and $\displaystyle{\mathscr{E}:= \cup_n \mathscr{E}_n}$. Set $\varepsilon_t = \textbf{1}_{[0,t]}$ (resp. $\textbf{1}_{[t,0]}$) if $t \geq 0$ (resp. $t < 0$). Let $\mathscr{H}$ be the Hilbert space defined as the closure of $\mathscr{E}$ with respect to the inner product
   \begin{equation}
    \langle \varepsilon_t, \varepsilon_s \rangle_{\mathscr{H}} = C_{H}(t,s),\quad s,t \in \R. \label{inner product}
   \end{equation}
     The mapping $\varepsilon_t \mapsto X_{t}$ can be extended to an isometry between $\mathscr{H}$ and the Gaussian space $\mathbb{H}_{1}$ associated with $X$. We will denote this isometry by $\varphi \mapsto X(\varphi).$

     Let $\mathscr{F}$ be the set of all smooth cylindrical random variables, i.e. of the form
   \[ F = \phi (X_{t_{1}},...,X_{t_{l}}),\] where $l \in \N^*$, $\phi : \mathbb{R}^{l}\rightarrow \mathbb{R}$ is a $C^{\infty}$-function such that $f$ and its partial derivatives have at most polynomial growth, and  $ t_{1} < . . . <t_{l}$ are some real numbers. The derivative of $F$ with respect to $X$ is the element of $L^{2}(\Omega, \mathscr{H})$ defined by \[ D_{s}F = \sum_{i=1}^{l}\frac{\partial\phi}{\partial x_{i}}(X_{t_{1}}, ...,X_{t_{l}})\varepsilon_{t_{i}}(s), \text{\: \: \:} s \in \R.\]In particular $D_{s}X_{t} = \varepsilon_t(s)$. For any integer $k \geq 1$, we denote by $\mathbb{D}^{k,2}$ the closure of $\mathscr{F}$ with respect to the norm
   \[ \|F\|_{k,2}^{2} = E(F^{2}) + \sum_{j=1}^{k} E[ \|D^{j}F\|_{\mathscr{H}^{\otimes j}}^{2}].\]
   The Malliavin derivative $D$ satisfies the chain rule. If $\varphi : \mathbb{R}^{n} \rightarrow \mathbb{R}$ is $C_{b}^{1}$ and if $F_1,\ldots,F_n$ are in $\mathbb{D}^{1,2}$, then $\varphi(F_{1},...,F_{n}) \in \mathbb{D}^{1,2}$ and we have
   \[ D\varphi(F_{1},...,F_{n}) = \sum_{i=1}^{n} \frac{\partial \varphi}{\partial x_{i}}(F_{1},...,F_{n})DF_{i}.\]
   We have the following Leibniz formula, whose proof is straightforward  by induction on $q$. Let $\varphi, \psi \in C_{b}^{q}$ $(q\geq 1)$, and fix $0 \leq u<v $ and $0\leq s<t .$ Then  $\big( \varphi(X_{s})+ \varphi(X_{t}) \big) \big(\psi(X_{u})+ \psi(X_{v}) \big) \in \mathbb{D}^{q,2}$ and
   \begin{eqnarray}
  && D^{q}\bigg( \big( \varphi(X_{s})+\varphi(X_{t})\big) \big(  \psi(X_{u})+\psi(X_{v})\big)\bigg)\label{Leibnitz1}\\
& =& \sum_{l=0}^{q} \binom{q}{l} \bigg( \varphi^{(l)}(X_{s})\varepsilon_s^{\otimes l}+\varphi^{(l)}(X_{t})\varepsilon_t^{\otimes l}\bigg)\tilde{\otimes}\bigg( \psi^{(q-l)}(X_{u})\varepsilon_u^{\otimes (q-l)}+\psi^{(q-l)}(X_{v})\varepsilon_v^{\otimes (q-l)}\bigg) \notag
   \end{eqnarray}
   where $\tilde{\otimes}$ stands for the symmetric tensor product and $ \varphi^{(l)}$ (resp. $ \psi^{(q-l)}$) means that $\varphi$ is differentiated $l$ times (resp. $\psi$ is differentiated $q-l$ times). A similar statement holds fo $ u<v \leq 0 $ and $ s<t \leq 0$.

   If a random element $u \in L^{2}(\Omega, \mathscr{H})$ belongs to the domain of the divergence operator, that is, if it satisfies
   \[ |E\langle DF,u\rangle_{\mathscr{H}}|\leq c_{u}\sqrt{E(F^{2})} \text{\: for  any\:} F\in \mathscr{F},\] then $I(u)$ is defined by the duality relationship
\[
E \big( FI(u)\big) = E \big( \langle DF,u\rangle_{\mathscr{H}}\big),
\]
for every $F \in \mathbb{D}^{1,2}.$

For every $n\geq 1$, let $\mathbb{H}_{n}$ be the $n$th Wiener chaos of $X$, that is, the closed linear subspace of $ L^{2}(\Omega, \mathscr{A},P)$ generated by the random variables $\lbrace H_{n}(B(h)), h \in \mathscr{H}, \|h\|_{\mathscr{H}}=1 \rbrace,$ where $H_{n}$ is the $n$th Hermite polynomial. Recall that  $H_0=0$, $H_p(x)= (-1)^p \exp(\frac{x^2}{2})\frac{d^p}{dx^p}\exp(-\frac{x^2}{2})$ for $p\geq 1$, and that
\begin{equation}\label{Hp-Hq}
E(H_p(X)H_q(Y))= \left\{
\begin{array}{lcl}
 p!(E[XY])^p & &\mbox{if $p=q$,}\\
 0 & &\mbox{otherwise}\\
  \end{array}
\right.,
\end{equation}
for jointly Gaussian $X,Y$ and integers $p,q \geq 1$.
 The mapping
\begin{equation}
I_{n}(h^{\otimes n}) = H_{n}(B(h))\label{linear-isometry}
\end{equation}
 provides a linear isometry between the symmetric tensor product $\mathscr{H}^{\odot n}$ and $\mathbb{H}_{n}$. For $H =\frac{1}{2}$, $I_{n}$ coincides with the multiple Wiener-It\^{o} integral of order $n$. The following duality formula holds
\begin{eqnarray}\label{adjoint}
  E \big( FI_{n}(h)\big) = E \big( \langle D^{n}F,h\rangle_{\mathscr{H}^{\otimes n}}\big),
  \end{eqnarray}
  for any element $ h\in \mathscr{H}^{\odot n}$ and any random variable $F \in \mathbb{D}^{n,2}.$

  Let $\lbrace e_{k}, k \geq 1\rbrace $ be a complete orthonormal system in $\mathscr{H}.$ Given $f \in \mathscr{H}^{\odot n}$ and $g \in \mathscr{H}^{\odot m},$ for every $r= 0,...,n\wedge m,$ the contraction of $f$ and $g$ of order $r$ is the element of $ \mathscr{H}^{\otimes(n+m-2r)}$ defined by
  \[ f\otimes_{r} g = \sum_{k_{1},...,k_{r} =1}^{\infty} \langle f, e_{k_{1}}\otimes...\otimes e_{k_{r}}\rangle_{\mathscr{H}^{\otimes r}}\otimes \langle g,e_{k_{1}}\otimes...\otimes e_{k_{r}}\rangle_{\mathscr{H}^{\otimes r}}.\]
 Finally, we recall the following product formula: if $f\in \mathscr{H}^{\odot n}$ and $g\in \mathscr{H}^{\odot m}$ then
  \begin{eqnarray}\label{product}
  I_{n}(f)I_{m}(g)= \sum_{r=0}^{n\wedge m} r! \binom{n}{r}\binom{m}{r} I_{n+m-2r}(f\otimes_{r}g).
  \end{eqnarray}
  
  \subsection{Some technical results}
  For all $k\in \Z$ and $n \in \N$, we write 
\begin{eqnarray*}
\delta_{(k+1)2^{-n/2}} = \varepsilon_{(k+1)2^{-n/2}} - \varepsilon_{k2^{-n/2}}.
\end{eqnarray*}
  The following lemma will play a pivotal role in the proof of Theorem \ref{thm1}. The reader can find an original version of this lemma in \cite[Lemma 5, Lemma 6]{NNT}.

\begin{lemma}\label{tech-lemma}
\begin{enumerate}

\item  If $H\leq \frac12$, for all integer $q\geq1$, for all $j \in \N$ and $u\in\R$,
\begin{eqnarray}
 |\langle \varepsilon_u^{\otimes q}, \delta_{(j+1)2^{-n/2}}^{\otimes q} \rangle_{\mathscr{H}^{\otimes q}} | &\leq & 2^{-nqH}.\label{0}
 \end{eqnarray}
 
\item  If $H > \frac12$, for all integer $q\geq 1$, for all $t \in \R_+$ and  $j, j' \in \{0, \ldots, \lfloor 2^{n/2}t\rfloor -1\}$,
\begin{eqnarray}
 |\langle \varepsilon_{j2^{-n/2}}^{\otimes q}, \delta_{(j'+1)2^{-n/2}}^{\otimes q} \rangle_{\mathscr{H}^{\otimes q}} | &\leq & 2^q 2^{-\frac{nq}{2}}t^{(2H-1)q},\label{1}\\
 |\langle \varepsilon_{(j+1)2^{-n/2}}^{\otimes q}, \delta_{(j'+1)2^{-n/2}}^{\otimes q} \rangle_{\mathscr{H}^{\otimes q}} | &\leq & 2^q 2^{-\frac{nq}{2}}t^{(2H-1)q}.\label{1'}
 \end{eqnarray}

\item For all integers $r,n \geq 1$ and $t\in\R_+$, and with $C_{H,r}$ a constant depending  only on $H$  and $r$ (but independent of $t$ and $n$),
\begin{enumerate}

\item if $H< 1-\frac{1}{2r}$,
 \begin{eqnarray}\label{2}
 \sum_{k,l=0}^{\lfloor 2^{n/2} t \rfloor - 1}|\langle \delta_{(k+1)2^{-n/2}} ; \delta_{(l+1)2^{-n/2}}\rangle_\mathscr{H}|^r &\leq& C_{H,r}  \: t\:  2^{n(\frac12-rH)}
 \end{eqnarray}

\item if $H=1 - \frac{1}{2r}$,
\begin{eqnarray}\label{3}
\sum_{k,l=0}^{\lfloor 2^{n/2} t \rfloor - 1}|\langle \delta_{(k+1)2^{-n/2}} ; \delta_{(l+1)2^{-n/2}}\rangle_\mathscr{H}|^r &\leq& C_{H,r}\:2^{n(\frac12-rH)}(t(1+n)  +t^2)
\end{eqnarray}

\item if $H> 1 - \frac{1}{2r}$,
\begin{eqnarray}
\sum_{k,l=0}^{\lfloor 2^{n/2} t \rfloor - 1}|\langle \delta_{(k+1)2^{-n/2}} ; \delta_{(l+1)2^{-n/2}}\rangle_\mathscr{H}|^r &\leq& C_{H,r}\big(t\: 2^{n(\frac12 -rH)} + \:t^{2-(2-2H)r}\:2^{n(1-r)}\big).\notag \\
\label{4}
\end{eqnarray}

\end{enumerate}
\item For $H\in(0,1)$. For all integer $n \geq 1$ and  $t\in\R_+$,
\begin{eqnarray}
 \sum_{k,l=0}^{\lfloor 2^{n/2} t \rfloor - 1}|\langle \varepsilon_{k2^{-n/2}} ; \delta_{(l+1)2^{-n/2}}\rangle_\mathscr{H}| &\leq&  2^{\frac{n}{2} +1} t^{2H + 1}, \label{5}\\
 \sum_{k,l=0}^{\lfloor 2^{n/2} t \rfloor - 1}|\langle \varepsilon_{(k+1)2^{-n/2}} ; \delta_{(l+1)2^{-n/2}}\rangle_\mathscr{H}| &\leq&  2^{\frac{n}{2}+1} t^{2H + 1}.\label{6}
\end{eqnarray}
\end{enumerate}
\end{lemma}
\begin{proof}{}
 The proof, which is quite long and technical, is postponed in Section 8.
 \end{proof}

It has been mentioned in \cite{kh-lewis1} that  $\{\|Y_{T_{\lfloor 2^n t\rfloor,n}}\|_{4} : n \geq 0\}$ is a bounded sequence. More generally, we have the following result. 
\begin{lemma}\label{bounded sequence}  For any integer $k \geq 1$,
$\{\|Y_{T_{\lfloor 2^n t\rfloor,n}}\|_{2k} : n \geq 0\}$ is a bounded sequence.
 
\end{lemma}
\begin{proof}{}
Recall from the introduction that $\{Y_{T_{k,n}} : k \geq 0\}$ is a simple and symmetric random walk on  $\mathscr{D}_n$, and observe that $Y_{T_{\lfloor 2^n t\rfloor,n}} = \sum_{l=0}^{\lfloor 2^n t\rfloor -1} (Y_{T_{l+1,n}} - Y_{T_{l,n}})$. So,  we have
\begin{eqnarray}
&& E\big[\big(Y_{T_{\lfloor 2^n t\rfloor,n}}\big)^{2k}\big] = \sum_{l_1, \ldots , l_{2k} =0}^{\lfloor 2^n t\rfloor -1}E\big[(Y_{T_{l_1+1,n}} - Y_{T_{l_1,n}})\times \ldots \times (Y_{T_{l_{2k}+1,n}} - Y_{T_{l_{2k},n}})\big]\notag\\
&=& \sum_{m=1}^k \sum_{a_1 + \ldots + a_m = 2k} C_{a_1, \ldots ,a_m}\sum_{\underset{ l_i \neq l_j \text{ for }i\neq j}{l_1, \ldots, l_m =0}}^{\lfloor 2^n t\rfloor -1}E\big[(Y_{T_{l_1+1,n}} - Y_{T_{l_1,n}})^{a_1}\big]\times \ldots \times E\big[(Y_{T_{l_m+1,n}} - Y_{T_{l_m,n}})^{a_m}\big],\notag\\ \label{lemma bounded}
\end{eqnarray}
where $\forall i\in\{1, \ldots ,m\}$ $a_i$ is an even integer, $ \forall m \in \{1, \ldots, k\}$ $C_{a_1,\ldots ,a_m} \geq 0$, is some combinatorial constant whose explicit value is immaterial here. Now observe that the quantity in (\ref{lemma bounded}) is equal to
\begin{eqnarray*}
&&\sum_{m=1}^k \sum_{a_1 + \ldots + a_m = 2k} C_{a_1, \ldots ,a_m} \lfloor 2^n t\rfloor \big(\lfloor 2^n t\rfloor -1\big)\times \ldots \times \big(\lfloor 2^n t\rfloor -m +1\big)2^{-\frac{n}{2}(a_1 + \ldots + a_m)}\\
&=&\sum_{m=1}^k \sum_{a_1 + \ldots + a_m = 2k} C_{a_1, \ldots ,a_m} \lfloor 2^n t\rfloor \big(\lfloor 2^n t\rfloor -1\big)\times  \ldots \times \big(\lfloor 2^n t\rfloor -m +1\big)2^{-nk},
\end{eqnarray*}
so, since $1\leq m \leq k$, we deduce that $\big\{ E\big[\big(Y_{T_{\lfloor 2^n t\rfloor,n}}\big)^{2k}\big] : n\geq 0 \big\}$ is a bounded sequence, which proves the lemma.  
\end{proof}
 Also, in order to prove the fourth point of Theorem \ref{thm1} we will need estimates on the local time of $Y$
taken from  \cite{kh-lewis1}, that we collect in the following statement.
\begin{prop}\label{local time}
\begin{enumerate}
\item For every $x\in\R$, $p \in \N^{*}$ and $t > 0$, we have
\[
E\big[( L_t^{x}(Y))^p\big] \leq 2\,E\big[(L_1^0(Y))^p\big]\,t^{p/2}\,
{\rm exp}\left(-\frac{x^2}{2t}\right).
\]
\item There exists a positive constant $\mu$ such that, for every $a,b\in\R$ with $ab\geq 0$
and $t >  0$,
\[
E\big[|L_t^b(Y)-L_t^a(Y)|^2\big]^{1/2}\leq\mu\,\sqrt{|b-a|}\,t^{1/4}\,
{\rm exp}\left(-\frac{a^2}{4t}\right).
\]
\item There exists a positive random variable $K\in L^8$ such that, for every $j\in\Z$, every $n\geq 0$
and every $t > 0$, one has that
\[
|{\cal L}_{j,n}(t)-L_t^{j2^{-n/2}}(Y)|\leq 2Kn2^{-n/4}\sqrt{L_t^{j2^{-n/2}}(Y)},
\]
where ${\cal L}_{j,n}(t)=2^{-n/2}(U_{j,n}(t)+D_{j,n}(t))$.
\end{enumerate}
\end{prop}

\subsection{Notation}
  
Throughout all the forthcoming proofs, we shall use the following notation.
 For all $t\in \R$ and $n \in \N$, we define $X^{(n)}_t := 2^{\frac{nH}{2}}X_{t2^{-\frac{n}{2}}}$.
For all $k \in \Z$ and $H \in (0,1)$, we write 
\begin{equation}\label{rho}
\rho(k) = \frac{1}{2}(|k+1|^{2H} + |k-1|^{2H} -2|k|^{2H}),
\end{equation}
it is clear  that $\rho(-k)=\rho(k)$. Observe that, by (\ref{inner product}), we have
\begin{eqnarray}
|\langle \delta_{(k+1)2^{-n/2}} ; \delta_{(l+1)2^{-n/2}}\rangle_\mathscr{H}| &=& \big|E[(X_{(k+1)2^{-n/2}}-X_{k2^{-n/2}})(X_{(l+1)2^{-n/2}}-X_{l2^{-n/2}})] \big|\notag\\
& =& \big|2^{-nH-1}(|k-l+1|^{2H} + |k-l-1|^{2H} -2|k-l|^{2H})\big|\notag\\
& =& 2^{-nH}|\rho(k-l)|. \label{inner product 2}
\end{eqnarray}

If $H\leq \frac12$, for all $r \in \N^*$, we define 
\begin{equation}\label{alpha}
\alpha_{r}:=\sqrt{r!\sum_{a\in\Z}\rho(a)^{r}}.
\end{equation}
Note that
$\sum_{a\in\Z}|\rho(a)|^{r} < \infty$ if and only if $H < 1- 1/(2r)$, which is satisfied for all $r \geq 1 $ if we suppose that $H \leq 1/2$ (in the case $H=1/2$, we have $\rho(0)=1$ and $\rho(a) =0$ for all $a \neq 0$. So, for any $r \in \N^*$, we have $\sum_{a\in \Z}|\rho(a)|^r=1$).

For simplicity, throughout the paper we remove the subscript $\mathscr{H}$ in the inner product defined in (\ref{inner product}), that is, we write $\langle\:\: ;\:\: \rangle$ instead of $\langle\:\: ;\:\: \rangle_\mathscr{H}$.

For any sufficiently smooth function $f : \R \to \R$, the notation $\partial^{l}f$ means that $f$ is differentiated $l$ times. We denote for any $j \in \Z$ , $\Delta_{j,n} f(X):=  \frac12(f(X_{j2^{-n/2}})+ f(X_{(j+1)2^{-n/2}})).$

In the proofs contained in this paper, $C$ shall denote a positive, finite constant that may change value from line to line.

\section{Preparation to the proof of Theorem \ref{thm1}}

\subsection{A key algebraic lemma}
For each
integer $n\geq 1$, $k\in\Z$ and real number $t\geq 0$, let $U_{j,n}(t)$ (resp.
$D_{j,n}(t)$) denote the number of \textit{upcrossings} (resp.
\textit{downcrossings}) of the interval
$[j2^{-n/2},(j+1)2^{-n/2}]$ within the first $\lfloor 2^n
t\rfloor$ steps of the random walk $\{Y_{T_{k,n}}\}_{k\geq 0}$, that is,
\begin{eqnarray}
U_{j,n}(t)=\sharp\big\{k=0,\ldots,\lfloor 2^nt\rfloor -1 :&&
\notag
\\ Y_{T_{k,n}}\!\!\!\!&=&\!\!\!\!j2^{-n/2}\mbox{ and }Y_{T_{k+1,n}}=(j+1)2^{-n/2}
\big\}; \notag\\
D_{j,n}(t)=\sharp\big\{k=0,\ldots,\lfloor 2^nt\rfloor -1:&&
\notag
\\ Y_{T_{k,n}}\!\!\!\!&=&\!\!\!\!(j+1)2^{-n/2}\mbox{ and }Y_{T_{k+1,n}}=j2^{-n/2}
\big\}.\notag
\end{eqnarray}
The following lemma taken from \cite[Lemma 2.4]{kh-lewis1} is going to be the key when
studying the asymptotic behavior of the weighted power variation $V_n^{(r)}(f,t)$ of order $r\geq 1$, defined as:
\begin{equation}
V_n^{(r)}(f,t)=\sum_{k=0}^{\lfloor 2^n t \rfloor -1} \frac12\big(f(Z_{T_{k,n}})+f(Z_{T_{k+1,n}})\big)\big[\big(2^{\frac{nH}{2}}(Z_{T_{k+1,n}}-Z_{T_{k,n}})\big)^r - \mu_r\big],\quad t\geq 0, \label{Vn}
\end{equation}
where $\mu_r:= E[N^r]$, with $N \sim \mathcal{N}(0,1)$.
Its main feature is to separate $X$ from $Y$, thus providing a representation of
$V_n^{(r)}(f,t)$ which is amenable to analysis.

\begin{lemme}\label{lemme-algebric}
Fix $f\in C^\infty_b$, $t\geq 0$ and $r\in \N^{*}$.
Then
\begin{eqnarray}
&&V_n^{(r)}(f,t) = \sum_{j\in\Z}
\frac12\left(
f(X_{j2^{-\frac{n}{2}}}) + f(X_{(j+1)2^{-\frac{n}{2}}})\right) \big[\big(2^{\frac{nH}{2}}(X_{(j+1)2^{-\frac{n}2}}-
X_{j2^{-\frac{n}{2}}})\big)^r - \mu_r \big] \notag\\
&& \hspace{6cm}\times\big(U_{j,n}(t)+ (-1)^rD_{j,n}(t)\big).
\end{eqnarray}
\end{lemme}

\subsection{Transforming the weighted power variations of odd order}

By \cite[Lemma 2.5]{kh-lewis1}, one has
\[
U_{j,n}(t) - D_{j,n}(t)= \left\{
\begin{array}{lcl}
 1_{\{0\leq j< j^*(n,t)\}} & &\mbox{if $j^*(n,t) > 0$}\\
 0 & &\mbox{if $j^{*}(n,t) = 0$}\\
  -1_{\{j^*(n,t)\leq j<0\}} & &\mbox{if $j^*(n,t)< 0$}
  \end{array}
\right.,
\]
where $j^*(n,t)=2^{n/2}Y_{T_{\lfloor 2^n t\rfloor,n}}$.
As a consequence, $V_n^{(2r-1)}(f,t)$ is equal to
\begin{eqnarray*}
\left\{
\begin{array}{lcl}
\sum_{j=0}^{j^*(n,t)-1}\frac12\big(f(X^+_{j2^{-n/2}}) + f(X^+_{(j+1)2^{-n/2}})\big) \big(X^{n,+}_{j+1}- X^{n,+}_j\big)^{2r-1} & &\mbox{if $j^*(n,t) > 0$}\\
  0 &&\mbox{if $j^{*}(n,t) = 0$}\\
\sum_{j=0}^{|j^*(n,t)|-1} \frac12\big(f(X^-_{j2^{-n/2}}) + f(X^-_{(j+1)2^{-n/2}})\big) \big(X^{n,-}_{j+1}- X^{n,-}_j\big)^{2r-1} & &\mbox{if $j^*(n,t) < 0$}
  \end{array}
\right.,
\end{eqnarray*}
where $X^+_t := X_t$ for $t\geq 0$, $X^-_{-t} :=X_t$ for $t<0$, $X^{n,+}_{t} := 2^{\frac{nH}{2}}X^+_{2^{-\frac{n}{2}}t}$ for $ t \geq 0$ and $X^{n,-}_{-t} := 2^{\frac{nH}{2}}X^-_{2^{-\frac{n}{2}}(-t)}$ for $t < 0$.

Let us now
introduce the following sequence of processes $W_{\pm,n}^{(2r-1)}$, in which $H_p$ stands for the $p$th Hermite polynomial ($H_1(x) =x$, $H_2(x) =x^2 -1$, etc.):
\begin{equation}
W_{\pm,n}^{(2r-1)}(f,t)=
 \sum_{j=0}^{\lfloor 2^{n/2}t\rfloor -1} \frac12\big(f(X^\pm_{j2^{-\frac{n}{2}}}) + f(X^\pm_{(j+1)2^{-\frac{n}{2}}})\big) H_{2r-1}(X^{n,\pm}_{j+1}- X^{n,\pm}_{j}),\quad t \geq 0 \label{Wn}
 \end{equation}
 \begin{equation}
 W_{n}^{(2r-1)}(f,t):=\left \{ \begin{array}{lc}
                      W_{+,n}^{(2r-1)}(f,t) &\text{if $t \geq 0$}\notag\\
                      W_{-,n}^{(2r-1)}(f,-t) &\text{if $t < 0$}
                      \end{array}
                      \right. .
\end{equation}
We then have, using the decomposition
\begin{equation}
x^{2r-1}=\sum_{i=1}^{r}\kappa_{r,i}H_{2i-1}(x), \label{hermite polynomial}
\end{equation}
 \big(with $\kappa_{r,r}=1$, and $\kappa_{r,1}=\frac{(2r)!}{r!2^r}=E[N^{2r}]$, with $N\sim \mathcal{N}(0,1)$. If interested, the reader can find the explicit value of $\kappa_{r,i}$, for $1<i<r$, e.g., in \cite[Corollary 1.2]{RZ1}\big),
\begin{eqnarray}
V_n^{(2r-1)}(f,t)
= \sum_{i=1}^{r}\kappa_{r,i}W_{n}^{(2i-1)}(f,Y_{T_{\lfloor 2^n t\rfloor,n}}).
\label{transforming}
\end{eqnarray}

\section{Proofs of (\ref{H >1/6 , r=1}) and (\ref{H = 1/6 , r=1})}

\subsection{Proof of (\ref{H >1/6 , r=1})}
In \cite[Theorem 2.1]{RZ2}, we have proved that for $H > \frac16$ and  $f \in C_b^{\infty}$, the following change-of-variable formula holds true
\begin{equation}
F(Z_t) -F(0) = \int_0^t f(Z_s)d^\circ Z_s, \quad t\geq 0 \label{N-Z H> 1/6}
\end{equation}
where $F$ is a primitive of $f$ and $\int_0^t f(Z_s)d^\circ Z_s$ is the limit in probability of $2^{-\frac{nH}{2}}V_n^{(1)}(f,t)$ as $n \to \infty$, with $V_n^{(1)}(f,t)$ defined in (\ref{Vn}). On the other hand, it has been proved in \cite[Theorem 4]{NNT} (see also \cite[Theorem 1.3]{RZ3} for an extension of this formula to the bi-dimensional case) that for  all $t\in \R$, the following change-of-variable formula holds true for $H > \frac16$
\begin{equation}
F(X_t)-F(0) = \int_0^t f(X_s)d^{\circ}X_s, \label{N-Tudor}
\end{equation}
where  $\int_0^t f(X_s)d^{\circ}X_s$ is the Stratonovich integral of $f(X)$ with respect to $X$ defined as the limit in probability of $2^{-\frac{nH}{2}}W_{n}^{(1)}(f,t)$ as $n \to \infty$, with $W_{n}^{(1)}(f,t)$ defined in (\ref{Wn}). 
Thanks to (\ref{N-Tudor}), we deduce that \[ F(Z_t) -F(0) = \int_0^{Y_t} f(X_s)d^{\circ}X_s, \quad t\geq0 \] by combining this last equality with (\ref{N-Z H> 1/6}), we get $\int_0^t f(Z_s)d^\circ Z_s = \int_0^{Y_t} f(X_s)d^{\circ}X_s$. So, we deduce that, for $H>\frac16$,
\begin{equation*}
\sum_{k=0}^{\lfloor 2^n t \rfloor -1} \frac12\big(f(Z_{T_{k,n}})+f(Z_{T_{k+1,n}})\big)(Z_{T_{k+1,n}}-Z_{T_{k,n}})\underset{n\to \infty}{\overset{P}{\longrightarrow}} \int_0^{Y_t} f(X_s)d^{\circ}X_s, 
\end{equation*}
thus (\ref{H >1/6 , r=1}) holds true.

\subsection{Proof of (\ref{H = 1/6 , r=1})}
In \cite[Theorem 2.1]{RZ2}, we have proved that for $H = \frac16$ and  $f \in C_b^{\infty}$, the following change-of-variable formula holds true
\begin{eqnarray}
 F(Z_t) - F(0) + \frac{\kappa_3}{12}\int_0^{Y_t}f''(X_s)dW_s  &\overset{\rm (law)}{=}& \int_0^{t}f(Z_s)d^{\circ}Z_s, \quad t\geq 0 \label{N-Z H=1/6}
\end{eqnarray}
where $F$ is a primitive of $f$, $W$ is a standard two-sided Brownian motion independent of the pair $(X,Y)$, $\kappa_3\simeq 2.322$ and $\int_0^{t}f(Z_s)d^{\circ}Z_s$ is the limit in law of $2^{-\frac{nH}{2}}V_n^{(1)}(f,t)$ as $n\to \infty$, with $V_n^{(1)}(f,t)$ defined in (\ref{Vn}). On the other hand, it has been proved in (2.19) in \cite{NRS} that for  all $t\in \R$, the following change-of-variable formula holds true for $H = \frac16$
\begin{equation}
F(X_t)-F(0) + \frac{\kappa_3}{12}\int_0^{t}f''(X_s)dW_s=\int_0^tf(X_s)d^{*}X_s, \label{N-Swan}
\end{equation}
where  $\kappa_3$ and $W$ are the same as in (\ref{N-Z H=1/6}), $\int_0^t f(X_s)d^{*}X_s$ is the Stratonovich integral of $f(X)$ with respect to $X$ defined as the limit in law of $2^{-\frac{nH}{2}}W_{n}^{(1)}(f,t)$ as $n \to \infty$, with $W_{n}^{(1)}(f,t)$ defined in (\ref{Wn}). 
Thanks to (\ref{N-Swan}), we deduce that \[ F(Z_t)-F(0) + \frac{\kappa_3}{12}\int_0^{Y_t}f''(X_s)dW_s=\int_0^{Y_t}f(X_s)d^* X_s, \quad t\geq0. \] By combining this last equality with (\ref{N-Z H=1/6}), we get $\int_0^t f(Z_s)d^\circ Z_s \overset{law}{=} \int_0^{Y_t} f(X_s)d^{*}X_s$. So, we deduce that, for $H=\frac16$,
\[
\sum_{k=0}^{\lfloor 2^n t \rfloor -1} \frac12\big(f(Z_{T_{k,n}})+f(Z_{T_{k+1,n}})\big)(Z_{T_{k+1,n}}-Z_{T_{k,n}})\underset{n\to \infty}{\overset{law}{\longrightarrow}} \int_0^{Y_t} f(X_s)d^{*}X_s,
\]
thus (\ref{H = 1/6 , r=1}) holds true.

\section{Proof of (\ref{fdd thm1})}
 Thanks to (\ref{Vn}) and  (\ref{transforming}), for any integer $r\geq 2$, we have
\begin{eqnarray} 
2^{-n/4}V_n^{(2r-1)}(f,t)&=& 2^{-n/4} \sum_{k=0}^{\lfloor 2^n t\rfloor-1}\frac12\big(f(Z_{T_{k,n}})+f(Z_{T_{k+1,n}})\big) (2^{\frac{nH}{2}}(Z_{T_{k+1,n}}- Z_{T_{k,n}}))^{2r-1}\notag \\
&=& 2^{-n/4}\sum_{l=1}^{r}\kappa_{r,l}W_{n}^{(2l-1)}(f,Y_{T_{\lfloor 2^n t\rfloor,n}}) \label{decomposition1 thm1}
\end{eqnarray}
The proof of (\ref{fdd thm1}) will be done in several steps.

\subsection{\underline{Step 1}: Limit of $2^{-n/4}\sum_{l=2}^{r}\kappa_{r,l}W_{n}^{(2l-1)}(f,t)$ }
Observe that, by (\ref{hermite polynomial}), we have
\begin{eqnarray*}
 &&\sum_{j=0}^{\lfloor 2^{n/2}t\rfloor -1} \frac12\big(f(X^\pm_{j2^{-\frac{n}{2}}}) + f(X^\pm_{(j+1)2^{-\frac{n}{2}}})\big) (X^{n,\pm}_{j+1}- X^{n,\pm}_{j})^{2r-1} = \sum_{l=1}^{r} \kappa_{r,l} W_{\pm,n}^{(2l-1)}(f,t).
 \end{eqnarray*}
We have the following proposition:
\begin{prop}\label{principal proposition} If $H \in ( \frac16, \frac12)$, if $r\geq 2$  then,
for any $f \in C_b^{\infty}$, 
\begin{equation}\label{principal conv2} 
\bigg(X_x, 2^{-\frac{n}{4}}\sum_{l=2}^{r} \kappa_{r,l} W_{\pm,n}^{(2l-1)}(f,t)\bigg)_{x\in \R,\, t \geq 0} \underset{n\to \infty}{\overset{f.d.d.}{\longrightarrow}} \bigg(X_x, \beta_{2r-1}\int_0^{t}f(X^\pm_s)dW^\pm_s \bigg)_{x\in \R,\, t \geq 0},
\end{equation}
where $\beta_{2r-1} = \sqrt{\sum_{l=2}^r\kappa^2_{r,l}\,\alpha^2_{2l-1}}$,  $\alpha_{2l-1}$ is given by (\ref{alpha}), $W^{+}_t=W_t$ if $t>0$ and
$W^{-}_t=W_{-t}$ if $t<0$, with $W$ a two-sided Brownian motion independent of $(X,Y)$, and where $\int_0^t f(X^{\pm}_s)dW^{\pm}_s$ must be  understood in the Wiener-It\^o sense.
\end{prop}
\begin{proof}{}
For all $t\geq 0$, we define $F_{\pm,n}^{(2r-1)}(f,t):= 2^{-\frac{n}{4}}\sum_{l=2}^{r} \kappa_{r,l} W_{\pm,n}^{(2l-1)}(f,t)$. In what follows we may study separately the finite dimensional distributions convergence in law  of $\big(X, F_{+,n}^{(2r-1)}(f, \cdot),$ $ F_{-,n}^{(2r-1)}(f,\cdot)\big)$ when $n$ is even and when $n$ is odd.
For the sake of simplicity, we will only consider the even case, 
the analysis when $n$ is odd being {\sl mutatis mutandis} the same. So,
assume that $n$ is even and let $m$ be another even integer such that $n \geq m \geq 0$. 
We shall apply a coarse gaining argument. We have
 \begin{eqnarray}
  F_{\pm,n}^{(2r-1)}(f, t)&=&  2^{-n/4} \sum_{i=1}^{\lfloor 2^{m/2}t\rfloor} \sum_{j=(i-1)2^{\frac{n-m}{2}}}^{i2^{\frac{n-m}{2}}-1}\frac12\left(f(X^\pm_{j2^{-n/2}}) + f(X^\pm_{(j+1)2^{-n/2}})\right)\label{analyse1}\\
 && \hskip6cm \times \left(\sum_{l=2}^r\kappa_{r,l} H_{2l-1}\big(X^{n,\pm}_{j+1}- X^{n,\pm}_j\big)\right)\notag\\
 &&+ 2^{-n/4}\sum_{j= \lfloor 2^{m/2}t\rfloor 2^{\frac{n-m}{2}}}^{\lfloor 2^{n/2}t\rfloor -1}\frac12\left(f(X^\pm_{j2^{-n/2}}) + f(X^\pm_{(j+1)2^{-n/2}})\right)\label{analyse2}\\
 && \hskip6cm \times \left(\sum_{l=2}^r\kappa_{r,l} H_{2l-1}\big(X^{n,\pm}_{j+1}- X^{n,\pm}_j\big)\right).\notag
 \end{eqnarray}
Observe that $2^{\frac{n-m}{2}}$ is an integer precisely because we have assumed that $n$ and $m$ are even numbers. We have 
 \[
 (\ref{analyse1})= A^\pm_{n,m}(t) + B^\pm_{n,m}(t) + C^\pm_{n,m}(t),
 \]
 where
 \begin{eqnarray*}
 && A^{\pm}_{n,m}(t)= 2^{-n/4} \sum_{i=1}^{\lfloor 2^{m/2}t\rfloor}\frac12\left(f(X^{\pm}_{i2^{-m/2}}) + f(X^{\pm}_{(i+1)2^{-m/2}})\right)
 \sum_{j=(i-1)2^{\frac{n-m}{2}}}^{i2^{\frac{n-m}{2}}-1}\\
 && \hskip7cm \bigg(\sum_{l=2}^r\kappa_{r,l} H_{2l-1}\big(X^{n,\pm}_{j+1}- X^{n,\pm}_j\big)\bigg)
 \end{eqnarray*}
 \begin{eqnarray*}
 && B^{\pm}_{n,m}(t)= 2^{-n/4} \sum_{i=1}^{\lfloor 2^{m/2}t\rfloor}\sum_{j=(i-1)2^{\frac{n-m}{2}}}^{i2^{\frac{n-m}{2}}-1}
 \frac12\left(f(X^{\pm}_{(j+1)2^{-n/2}}) - f(X^{\pm}_{(i+1)2^{-m/2}})\right)\\
 &&\hskip6cm \times \bigg(\sum_{l=2}^r\kappa_{r,l} H_{2l-1}\big(X^{n,\pm}_{j+1}- X^{n,\pm}_j\big)\bigg)
 \end{eqnarray*}
 \begin{eqnarray*}
 && C^{\pm}_{n,m}(t)= 2^{-n/4} \sum_{i=1}^{\lfloor 2^{m/2}t\rfloor}\sum_{j=(i-1)2^{\frac{n-m}{2}}}^{i2^{\frac{n-m}{2}}-1}        
 \frac12\left(f(X^{\pm}_{j2^{-n/2}}) - f(X^{\pm}_{i2^{-m/2}})\right)\\
 &&\hskip6cm \times \bigg(\sum_{l=2}^r\kappa_{r,l} H_{2l-1}\big(X^{n,\pm}_{j+1}- X^{n,\pm}_j\big)\bigg)
 \end{eqnarray*}
 Here is a sketch of what remains to be done in order to complete the proof of (\ref{principal conv2}). 
 Firstly, we will prove {\bf (a)} the f.d.d. convergence in law of  $(X, A^+_{n,m}, A^-_{n,m})$  to $(X, \beta_{2r-1}\int_0^{\cdot}f(X^+_s)dW^+_s,$ $ \beta_{2r-1}\int_0^{\cdot}f(X^-_s)dW^-_s)$  as $n\to\infty$ and then $m\to\infty$. Secondly, we will show that {\bf (b)} $B^{\pm}_{n,m}(t)$ converges to $0$ in $L^2(\Omega)$ as $n\to\infty$ and then $m\to\infty$. By applying the same techniques, we would also obtain that the same holds with $C^{\pm}_{n,m}(t)$. Thirdly, we will prove that {\bf (c)} (\ref{analyse2}) converges to $0$ in $L^2(\Omega)$ as $n\to\infty$ and then $m\to\infty$.
 Once this has been done, one can easily deduce the f.d.d. convergence in law of $(X, F_{+,n}^{(2r-1)}(f, \cdot), F_{-,n}^{(2r-1)}(f,\cdot))$ to $(X, \beta_{2r-1}\int_0^{\cdot}f(X^+_s)dW^+_s, \beta_{2r-1}\int_0^{\cdot}f(X^-_s)dW^-_s)$ as $n\to\infty$, which is equivalent to (\ref{principal conv2}).\\

{\bf  (a) Finite dimensional distributions convergence in law of  $(X, A^+_{n,m}, A^-_{n,m})$}\\
Fix $m$. Showing the f.d.d. convergence in law of $(X, A^+_{n,m}, A^-_{n,m})$ as $n\to\infty$ can be easily reduced to checking the f.d.d. convergence in law
of the following random-vector valued process:
\begin{eqnarray*}
\bigg( X_x: x\in \R,\, 2^{-n/4}\sum_{j=(i-1)2^{\frac{n-m}{2}}}^{i2^{\frac{n-m}{2}}-1}H_{2l-1}\big(X^{n,+}_{j+1}- X^{n,+}_j\big) \:, \: 2^{-n/4}\sum_{j=(i-1)2^{\frac{n-m}{2}}}^{i2^{\frac{n-m}{2}}-1}H_{2l-1}\big(X^{n,-}_{j+1}- X^{n,-}_j\big) :\\ 2 \leq l \leq r, \: 1 \leq i \leq \lfloor 2^{m/2}t \rfloor \bigg).       
\end{eqnarray*}
Thanks to  $(3.27)$ in \cite{RZ1} (see also (3.4) in \cite{RZ1} and page 1073 in \cite{NNT}), we have
\begin{eqnarray*}
  &&\bigg(2^{-n/4}\sum_{j=(i-1)2^{\frac{n-m}{2}}}^{i2^{\frac{n-m}{2}}-1} H_{2l-1}(X_{j+1}^{n,+} - X_j^{n,+}), 2^{-n/4}\sum_{j=(i-1)2^{\frac{n-m}{2}}}^{i2^{\frac{n-m}{2}}-1} H_{2l-1}(X_{j+1}^{n,-} - X_j^{n,-}):\\
  &&\hskip6cm \, 2 \leq l \leq   r ,\: 1 \leq i \leq \lfloor 2^{m/2}t \rfloor \bigg) \underset{n\to \infty}{\overset{ \rm law}{\longrightarrow}}
  \end{eqnarray*} 
  \begin{eqnarray*}
  && \bigg( \alpha_{2l-1}(B_{(i+1)2^{-m/2}}^{l,+} - B_{i2^{-m/2}}^{l,+})\: , \: \alpha_{2l-1}(B_{(i+1)2^{-m/2}}^{l,-} - B_{i2^{-m/2}}^{l,-}):\\
  && \hskip6cm \, 2 \leq l \leq r , 1 \leq i \leq \lfloor 2^{m/2}t \rfloor \bigg)
  \end{eqnarray*}
 where $(B^{(2)}, \ldots , B^{(r)})$ is a $(r-1 )$-dimensional two-sided Brownian motion and $\alpha_{2l-1}$ is defined in (\ref{alpha}),  for all $t \geq 0$, $B_t^{r,+} := B^{(r)}_t$, $B_t^{r,-}:= B^{(r)}_{-t}$. \\
  Since $E[ X_x H_{2r-1}(X_{j+1}^{n,\pm} - X_j^{n,\pm})] = 0$ when $r \geq 2$ (Hermite polynomials of different orders are orthogonal), Peccati-Tudor Theorem (see, e.g., \cite[Theorem $6.2.3$]{NP2}) applies and yields
  \begin{equation*}
   \bigg( X_x, 2^{-n/4}\sum_{j=(i-1)2^{\frac{n-m}{2}}}^{i2^{\frac{n-m}{2}}-1}H_{2l-1}\big(X^{n,+}_{j+1}- X^{n,+}_j\big), 2^{-n/4}\sum_{j=(i-1)2^{\frac{n-m}{2}}}^{i2^{\frac{n-m}{2}}-1}H_{2l-1}\big(X^{n,+}_{j+1}- X^{n,+}_j\big)  : 
  \end{equation*}
  \[
   \hskip5cm 2 \leq l \leq  r, \: 1 \leq i \leq \lfloor 2^{m/2}t \rfloor \bigg)_{x \geq 0}   \underset{n\to \infty}{\overset{\rm f.d.d.}{\longrightarrow}} 
   \]
 \begin{eqnarray*}
  &&\bigg( X_x, \alpha_{2l-1}(B_{(i+1)2^{-m/2}}^{l,+} - B_{i2^{-m/2}}^{l,+}), \alpha_{2l-1}(B_{(i+1)2^{-m/2}}^{l,-} - B_{i2^{-m/2}}^{l,-}) :\\
  && \hskip7cm 2 \leq l \leq  r , 1 \leq i \leq \lfloor 2^{m/2}t \rfloor \bigg)_{x \geq 0},
  \end{eqnarray*}
with $(B^{(2)}, \ldots , B^{(r-1)})$ is  independent of $X$ (and independent of $Y$ as well). 
We then have, as $n\to\infty$ and $m$ is fixed,
\begin{eqnarray*}
&&(X, A^+_{n,m}, A^-_{n,m}) \overset{\rm f.d.d.}{\longrightarrow}\\
&& \bigg(X, \beta_{2r-1}\sum_{i=1}^{\lfloor 2^{m/2}.\rfloor}\frac12\big(f(X^+_{i2^{-m/2}}) + f(X^+_{(i+1)2^{-m/2}})\big)( W^+_{(i+1)2^{-m/2}} - W^+_{i2^{-m/2}}) ,\\
&& \hskip2cm \beta_{2r-1}\sum_{i=1}^{\lfloor 2^{m/2}.\rfloor}\frac12\big(f(X^-_{i2^{-m/2}}) + f(X^-_{(i+1)2^{-m/2}})\big)( W^-_{(i+1)2^{-m/2}} - W^-_{i2^{-m/2}})\bigg),
\end{eqnarray*}
with $\beta_{2r-1} := \sqrt{\sum_{l=2}^r\kappa^2_{r,l}\,\alpha^2_{2l-1}}$ and $W$ is a two-sided Brownian motion independent of $X$ (and independent of $Y$ as well).
One can write
\[\sum_{i=1}^{\lfloor 2^{m/2}t\rfloor}\frac12\big(f(X^{\pm}_{i2^{-m/2}}) + f(X^{\pm}_{(i+1)2^{-m/2}})\big)( W^{\pm}_{(i+1)2^{-m/2}} - W^{\pm}_{i2^{-m/2}}) = K^{\pm}_m(t) + L^{\pm}_m(t),
\]
with
\[
K^{\pm}_m(t)= \sum_{i=1}^{\lfloor 2^{m/2}t\rfloor} f(X^{\pm}_{i2^{-m/2}})(W^{\pm}_{(i+1)2^{-m/2}}- W^{\pm}_{i2^{-m/2}}),
\]
\[
L^{\pm}_m(t) = \sum_{i=1}^{\lfloor 2^{m/2}t\rfloor}\frac12\big(f(X^{\pm}_{(i+1)2^{-m/2}}) - f(X^{\pm}_{i2^{-m/2}})\big)( W^{\pm}_{(i+1)2^{-m/2}} - W^{\pm}_{i2^{-m/2}}).
\]
It is clear that $\displaystyle{K^{\pm}_m(t) \underset{m\to\infty}{\overset{L^2}{\longrightarrow}} \int_0^t f(X^{\pm}_s)dW^{\pm}_s}$. On the other hand $L^{\pm}_m(t)$ converges to $0$ in $L^2$ as $m\to\infty$. Indeed, by independence,
\begin{align}
 &E[L^{\pm}_m(t)^2]\notag\\ 
&= \frac14\sum_{i,j=1}^{\lfloor 2^{m/2}t\rfloor}E\bigg[ \big(f(X^{\pm}_{(i+1)2^{-m/2}}) - f(X^{\pm}_{i2^{-m/2}})\big)\big(f(X^{\pm}_{(j+1)2^{-m/2}}) - f(X^{\pm}_{j2^{-m/2}})\big)\bigg] \notag \\
& \hskip1.6cm \times E\bigg[ (W^{\pm}_{(i+1)2^{-m/2}} - W^{\pm}_{i2^{-m/2}})(W^{\pm}_{(j+1)2^{-m/2}} - W^{\pm}_{j2^{-m/2}})\bigg] \notag \\
&= \frac14\sum_{i=1}^{\lfloor 2^{m/2}t\rfloor}E\bigg[ \big(f(X^{\pm}_{(i+1)2^{-m/2}}) - f(X^{\pm}_{i2^{-m/2}})\big)^2\bigg] \times E\bigg[\big(W^{\pm}_{(i+1)2^{-m/2}} - W^{\pm}_{i2^{-m/2}}\big)^2\bigg] \notag \\
&= \frac{2^{-m/2}}{4} \sum_{i=1}^{\lfloor 2^{m/2}t\rfloor}E\bigg[ f'(X_{\theta_i})^2\big( X^{\pm}_{(i+1)2^{-m/2}}) - X^{\pm}_{i2^{-m/2}} \big)^2\bigg]\label{inequalite},
\end{align}
where $\theta_i$ denotes a random real number satisfying $i2^{-m/2} < \theta_i < (i+1)2^{-m/2}$. Since $f \in C_b^{\infty}$ and by Cauchy-Schwarz inequality, we deduce that  
\begin{align*}
(\ref{inequalite}) &\leq C_f 2^{-m/2} \sum_{i=1}^{\lfloor 2^{m/2}t\rfloor}E\bigg[\big( X^{\pm}_{(i+1)2^{-m/2}}) - X^{\pm}_{i2^{-m/2}} \big)^4\bigg]^{1/2}\\
& = C_f 2^{-m/2}\lfloor 2^{m/2}t\rfloor 2^{-mH} \sqrt{3}\\
& \leq  C_f  2^{-mH} t,  
\end{align*}
from which the claim follows.
Summarizing, we just showed that  
\[(X, A^+_{n,m}, A^-_{n,m}) \overset{\rm f.d.d.}{\longrightarrow} (X, \beta_{2r-1}\int_0^{.} f(X^+_s)dW^+_s,\beta_{2r-1}\int_0^{.} f(X^-_s)dW^-_s) \]
 as $n\to\infty$ then $m\to\infty$.\\

{\bf (b) $B^{\pm}_{n,m}(t)$ converges to $0$ in $L^2(\Omega)$ as $n\to\infty$ and then $m\to\infty$}.

It suffices to prove that for all $k \in \{ 2, \ldots,r\}$,  
\begin{equation}
B^{\pm,k}_{n,m}(t)\overset{L^2}{\longrightarrow} 0, \label{long1}
\end{equation}
 as $n\to\infty$ and then $m\to\infty$, where $B^{\pm,k}_{n,m}(t)$ is defined as follows
\[
B^{\pm,k}_{n,m}(t):= 2^{-n/4} \sum_{i=1}^{\lfloor 2^{m/2}t\rfloor}\sum_{j=(i-1)2^{\frac{n-m}{2}}}^{i2^{\frac{n-m}{2}}-1}
 \frac12\big(f(X^{\pm}_{(j+1)2^{-n/2}}) - f(X^{\pm}_{(i+1)2^{-m/2}})\big) H_{2k-1}\big(X^{n,\pm}_{j+1}- X^{n,\pm}_j\big).
\]
With obvious notation, we have that 
\[
B^{\pm,k}_{n,m}(t)= 2^{-n/4} \sum_{i=1}^{\lfloor 2^{m/2}t\rfloor}\sum_{j=(i-1)2^{\frac{n-m}{2}}}^{i2^{\frac{n-m}{2}}-1}\Delta_{i,j}^{n,m}f(X^{\pm}) H_{2k-1}\big(X^{n,\pm}_{j+1}- X^{n,\pm}_j\big).
\]
It suffices to prove the convergence to $0$ of $B^{+,k}_{n,m}(t)$, the proof for $B^{-,k}_{n,m}(t)$ being exactly the same. In fact, the reader can find this proof in the proof of \cite[Theorem 1, (1.15)]{NNT} at page 1073.

{\bf (c)  (\ref{analyse2}) converges to $0$ in $L^2(\Omega)$ as $n\to\infty$ and then $m\to\infty$}. 
 
It suffices to prove that for all $k \in \{ 2, \ldots,r\}$,  $\displaystyle{ J^{\pm,k}_{n,m}(t)\overset{L^2}{\longrightarrow} 0 }$ as $n\to\infty$ and then $m\to\infty$, where $J^{\pm,k}_{n,m}(t)$ is defined as follows,
\begin{eqnarray*}
&& J_{n,m}^{\pm,k}(t) = 2^{-n/4}\sum_{j= \lfloor 2^{m/2}t\rfloor 2^{\frac{n-m}{2}}}^{\lfloor 2^{n/2}t\rfloor -1}\frac12\big(f(X^{\pm}_{j2^{-n/2}}) + f(X^{\pm}_{(j+1)2^{-n/2}})\big)H_{2k-1}\big(X^{n,\pm}_{j+1}- X^{n,\pm}_j\big)\\
 && \hskip1.4cm = 2^{-n/4}\sum_{j= \lfloor 2^{m/2}t\rfloor 2^{\frac{n-m}{2}}}^{\lfloor 2^{n/2}t\rfloor -1}\Theta_j^n f(X^{\pm}) H_{2k-1}\big(X^{n,\pm}_{j+1}- X^{n,\pm}_j\big),
\end{eqnarray*}
with obvious notation. We will only prove the convergence to 0 of $J_{n,m}^{+,k}(t)$, the proof for $J_{n,m}^{-,k}(t)$ being exactly the same. Using the relationship between Hermite polynomials and multiple stochastic integrals, namely
$H_r\big(2^{nH/2}(X^{+}_{(j+1)2^{-n/2}} - X^{+}_{j2^{-n/2}}) \big) = 2^{nrH/2}I_r(\delta_{(j+1)2^{-n/2}}^{\otimes r})$, we obtain, 
using (\ref{product}) as well,
\begin{eqnarray}
E\big[(J^{+,k}_{n,m}(t))^2 \big] &=& \bigg| 2^{-n/2} 2^{nH(2k-1)}\sum_{j,j'= \lfloor 2^{m/2}t\rfloor 2^{\frac{n-m}{2}}}^{\lfloor 2^{n/2}t \rfloor -1}\sum_{l=0}^{2k-1}l! 
\binom{2k-1}{l}^2 \times
\notag
\end{eqnarray}
\begin{eqnarray}
&&  \hskip1.3cm  E\big[ \Theta_j^n f(X^+)\Theta_{j'}^n f(X^+)I_{2(2k-1)-2l}(\delta_{(j+1)2^{-n/2}}^{\otimes (2k-1-l)}\otimes\delta_{(j'+1)2^{-n/2}}^{\otimes (2k-1-l)})\big] \notag\\
&& \hskip7cm \times \langle \delta_{(j+1)2^{-n/2}}; \delta_{(j'+1)2^{-n/2}}\rangle^l \bigg| \notag\\
&\leq & 2^{-n/2} 2^{nH(2k-1)}\sum_{j,j'= \lfloor 2^{m/2}t\rfloor 2^{\frac{n-m}{2}}}^{\lfloor 2^{n/2}t \rfloor -1}\sum_{l=0}^{2k-1}l!\binom{2k-1}{l}^2 \notag\\
&& \hskip1cm  \times \bigg| E\big[ \Theta_j^n f(X^+)\Theta_{j'}^n f(X^+)I_{2(2k-1)-2l}(\delta_{(j+1)2^{-n/2}}^{\otimes (2k-1-l)}\otimes\delta_{(j'+1)2^{-n/2}}^{\otimes (2k-1-l)})
\big] \bigg| \notag\\
&& \hskip6.7cm \times \big|\langle \delta_{(j+1)2^{-n/2}}; \delta_{(j'+1)2^{-n/2}}\rangle \big|^l \notag \\
&=&  \sum_{l=0}^{2k-1} l!\binom{2k-1}{l}^2 Q_{n,m}^{+,l}(t), \label{grand1}
\end{eqnarray}
with obvious notation. Thanks both to the duality formula (\ref{adjoint}) and to (\ref{Leibnitz1}), we have 
\begin{eqnarray*}
d_n^{(+,l)}(j,j')&:=& E\big[ \Theta_j^n f(X^+)\Theta_{j'}^n f(X^+)I_{2(2k-1)-2l}(\delta_{(j+1)2^{-n/2}}^{\otimes (2k-1-l)}\otimes\delta_{(j'+1)2^{-n/2}}^{\otimes (2k-1-l)})
\big]\\ 
&=& E \big[ \big \langle D^{2(2k-1-l)}(\Theta_j^n f(X^+)\Theta_{j'}^n f(X^+))\: ; \: \delta_{(j+1)2^{-n/2}}^{\otimes (2k-1-l)}\otimes\delta_{(j'+1)2^{-n/2}}^{\otimes (2k-1-l)} \big \rangle \big]\\
&=& \frac{1}{4} \sum_{a=0}^{2(2k-1-l)}\binom{2(2k-1-l)}{a} E\bigg[ \bigg \langle \big(f^{(a)}(X^+_{j2^{-n/2}})\xi_{j2^{-n/2}}^{\otimes a} + f^{(a)}(X^+_{(j+1)2^{-n/2}})\\
&& \times \xi_{(j+1)2^{-n/2}}^{\otimes a} \big)\tilde{\otimes} \big(f^{(2(2k-1-l)-a)}(X^+_{j'2^{-n/2}})\xi_{j'2^{-n/2}}^{\otimes (2(2k-1-l)-a)} + \\
&& f^{(2(2k-1-l)-a)}(X^+_{(j'+1)2^{-n/2}}) \xi_{(j'+1)2^{-n/2}}^{\otimes (2(2k-1-l)-a)} \big) \: ; \: \delta_{(j+1)2^{-n/2}}^{\otimes (2k-1-l)}\otimes\delta_{(j'+1)2^{-n/2}}^{\otimes (2k-1-l)} \bigg \rangle \bigg].
\end{eqnarray*}
At this stage, the proof of the claim ${\bf (c)}$ is going to be different according to the value of $l$:
\begin{itemize}
\item If $l= 2k-1$  in (\ref{grand1}) then
\begin{eqnarray}
&&  Q_{n,m}^{+,2k-1}(t) = 2^{-n/2} 2^{nH(2k-1)}\sum_{j,j'= \lfloor 2^{m/2}t\rfloor 2^{\frac{n-m}{2}}}^{\lfloor 2^{n/2}t \rfloor -1} \bigg|E\big[ \Theta_j^n f(X^+)\Theta_{j'}^n f(X^+)\big]\bigg| \notag\\
&& \hskip8.5cm \times \big|\langle \delta_{(j+1)2^{-n/2}}; \delta_{(j'+1)2^{-n/2}}\rangle \big|^{2k-1} \notag\\
&\leq & C_f 2^{-n/2} 2^{nH(2k-1)}\sum_{j,j'= \lfloor 2^{m/2}t\rfloor 2^{\frac{n-m}{2}}}^{\lfloor 2^{n/2}t \rfloor -1} \big|\langle \delta_{(j+1)2^{-n/2}}; \delta_{(j'+1)2^{-n/2}}\rangle \big|^{2k-1} \notag \\
&=& C_f 2^{-n/2}\sum_{j,j'= \lfloor 2^{m/2}t\rfloor 2^{\frac{n-m}{2}}}^{\lfloor 2^{n/2}t \rfloor -1}
\big|\frac12( |j-j'+1|^{2H} + |j-j'-1|^{2H} - 2 |j-j'|^{2H}) \big|^{2k-1}\notag 
\end{eqnarray}
\begin{eqnarray}
&=&  C_f 2^{-n/2}\sum_{j = \lfloor 2^{m/2}t\rfloor 2^{\frac{n-m}{2}}}^{\lfloor 2^{n/2}t \rfloor -1}\sum_{p =j - \lfloor 2^{n/2}t \rfloor + 1}^{j - \lfloor 2^{m/2}t\rfloor 2^{\frac{n-m}{2}}} \big|\frac12( |p+1|^{2H} + |p-1|^{2H} - 2 |p|^{2H}) \big|^{2k-1}\notag \\ \label{idiot1}
\end{eqnarray}
where we have the first inequality since $f$ belongs to $C_b^{\infty}$ and the last one follows by the change of variable $p=j-j'$. Using the notation (\ref{rho}), and by a Fubini argument, we get that the  quantity given in (\ref{idiot1}) is equal to 
\begin{eqnarray}
&& C_f 2^{-n/2} \sum_{p = \lfloor 2^{m/2}t\rfloor 2^{\frac{n-m}{2}} - \lfloor 2^{n/2}t \rfloor +1}^{\lfloor 2^{n/2}t \rfloor - \lfloor 2^{m/2}t\rfloor 2^{\frac{n-m}{2}} -1} |\rho(p)|^{2k-1} \bigg( (p+ \lfloor 2^{n/2}t \rfloor)\wedge 
\lfloor 2^{n/2}t \rfloor - ( p+ \notag \\
&& \hskip6cm \lfloor 2^{m/2}t\rfloor 2^{\frac{n-m}{2}})\vee \lfloor 2^{m/2}t\rfloor 2^{\frac{n-m}{2}} \bigg).\label{idiot2}
\end{eqnarray}
By separating the cases when $ 0 \leq p \leq \lfloor 2^{n/2}t \rfloor - \lfloor 2^{m/2}t\rfloor 2^{\frac{n-m}{2}} -1$ or when $\lfloor 2^{m/2}t\rfloor 2^{\frac{n-m}{2}} - \lfloor 2^{n/2}t \rfloor +1 \leq p < 0$ we deduce that
 \begin{eqnarray*}
 && 0 \leq \bigg(\frac{(p+ \lfloor 2^{n/2}t \rfloor)}{2^{n/2}}\wedge \frac{(\lfloor 2^{n/2}t \rfloor)}{2^{n/2}} - \frac{( p+ \lfloor 2^{m/2}t\rfloor 2^{\frac{n-m}{2}})}{2^{n/2}}\vee \lfloor 2^{m/2}t\rfloor 2^{-m/2}\bigg) \\
 && \leq \lfloor 2^{n/2}t \rfloor 2^{-n/2} - \lfloor 2^{m/2}t \rfloor 2^{-m/2} = \big|\lfloor 2^{n/2}t \rfloor 2^{-n/2} - \lfloor 2^{m/2}t \rfloor 2^{-m/2} \big|\\
 && \leq \big|\lfloor 2^{n/2}t \rfloor 2^{-n/2} -t \big|+ \big|t- \lfloor 2^{m/2}t \rfloor 2^{-m/2} \big| \leq 2^{-n/2} + 2^{-m/2}.
 \end{eqnarray*}
As a result, the quantity given in (\ref{idiot2}) is bounded by 
\[
 C_f \sum_{p \in \Z } |\rho(p)|^{2k-1} ( 2^{-n/2} + 2^{-m/2}),
\]
with $\sum_{p \in \Z} |\rho(p)|^{2k-1} < \infty$ (because  $H< 1/2 \leq 1 - \frac{1}{4k-2}$). Finally, we have
 \begin{equation}
 Q_{n,m}^{+,2k-1}(t) \leq C(2^{-n/2} + 2^{-m/2}). \label{grand2}
 \end{equation}

 \item \underline{Preparation to the cases  $ 0\leq l \leq 2k-2$} In order to handle the terms  $Q_{n,m}^{+,l}(t)$ whenever $ 0\leq l \leq 2k-2$, we will make use of the following decomposition:
\begin{equation}
|d_n^{(+ ,l)}(j,j')| \leq \frac14 \big(\Omega^{(1,l)}_n(j,j') + \Omega^{(2,l)}_n(j,j') + \Omega^{(3,l)}_n(j,j') + \Omega^{(4,l)}_n(j,j')\big), \label{decomposition}
\end{equation}
where
 \begin{eqnarray*}
 && \Omega^{(1,l)}_n(j,j') = \sum_{a=0}^{2(2k-1-l)}\binom{2(2k-1-l)}{a} \big| E[f^{(a)}(X^+_{j2^{-n/2}})f^{(2(2k-1-l)-a)}(X^+_{j'2^{-n/2}})]\big|\\
 && \hskip4cm \times \big| \big \langle \xi_{j2^{-n/2}}^{\otimes a}\tilde{\otimes}\xi_{j'2^{-n/2}}^{\otimes (2(2k-1-l)-a)}; \delta_{(j+1)2^{-n/2}}^{\otimes (2k-1-l)}\otimes\delta_{(j'+1)2^{-n/2}}^{\otimes (2k-1-l)}\big \rangle \big|\\
 && \Omega^{(2,l)}_n(j,j') = \sum_{a=0}^{2(2k-1-l)}\binom{2(2k-1-l)}{a} \big| E[f^{(a)}(X^+_{j2^{-n/2}})f^{(2(2k-1-l)-a)}(X^+_{(j'+1)2^{-n/2}})]\big|\\
 && \hskip4cm \times \big| \big \langle \xi_{j2^{-n/2}}^{\otimes a}\tilde{\otimes}\xi_{(j'+1)2^{-n/2}}^{\otimes (2(2k-1-l)-a)}; \delta_{(j+1)2^{-n/2}}^{\otimes (2k-1-l)}\otimes\delta_{(j'+1)2^{-n/2}}^{\otimes (2k-1-l)}\big \rangle \big|\\
 && \Omega^{(3,l)}_n(j,j') = \sum_{a=0}^{2(2k-1-l)}\binom{2(2k-1-l)}{a} \big| E[f^{(a)}(X^+_{(j+1)2^{-n/2}})f^{(2(2k-1-l)-a)}(X^+_{j'2^{-n/2}})]\big|\\
 && \hskip4cm \times \big| \big \langle \xi_{(j+1)2^{-n/2}}^{\otimes a}\tilde{\otimes}\xi_{j'2^{-n/2}}^{\otimes (2(2k-1-l)-a)}; \delta_{(j+1)2^{-n/2}}^{\otimes (2k-1-l)}\otimes\delta_{(j'+1)2^{-n/2}}^{\otimes (2k-1-l)}\big \rangle \big|\\
 && \Omega^{(4,l)}_n(j,j') = \sum_{a=0}^{2(2k-1-l)}\binom{2(2r-1-l)}{a} \big| E[f^{(a)}(X^+_{(j+1)2^{-n/2}})f^{(2(2k-1-l)-a)}(X^+_{(j'+1)2^{-n/2}})]\big|\\
 && \hskip4cm \times \big| \big \langle \xi_{(j+1)2^{-n/2}}^{\otimes a}\tilde{\otimes}\xi_{(j'+1)2^{-n/2}}^{\otimes (2(2k-1-l)-a)}; \delta_{(j+1)2^{-n/2}}^{\otimes (2k-1-l)}\otimes \delta_{(j'+1)2^{-n/2}}^{\otimes (2k-1-l)}\big \rangle \big|.
 \end{eqnarray*}
 
 \item \underline{For $ 1\leq l \leq 2k-2$ :} Since $f$ belongs to $C_b^{\infty}$ and thanks to (\ref{0}),  we deduce that
 \[
 d_n^{(+,l)}(j,j') \leq C (2^{-nH})^{2(2k-1-l)}.
 \]
  As a consequence of  this previous inequality we have 
 \begin{eqnarray}
 && Q_{n,m}^{+,l}(t) \notag\\
 &&\leq C (2^{-nH})^{2(2k-2)}\:2^{-n/2}\: 2^{nH(2k-1)}\sum_{j,j'= \lfloor 2^{m/2}t\rfloor 2^{\frac{n-m}{2}}}^{\lfloor 2^{n/2}t \rfloor -1} \big|\langle \delta_{(j+1)2^{-n/2}}; \delta_{(j'+1)2^{-n/2}}\rangle \big|^l \notag\\
 && \leq C (2^{-nH})^{2(2k-2)}\: 2^{nH(2k-1)} 2^{-nHl} \bigg(\sum_{p \in \Z}|\rho(p)|^l\bigg) ( 2^{-n/2} + 2^{-m/2})\notag\\
 && \leq C\: 2^{-nH(2k-2)}(2^{-n/2} + 2^{-m/2}), \label{grand3}
 \end{eqnarray}
 where we have the second inequality by the same arguments that have been used previously in the case $l =2k-1$.

 \item \underline{For $l=0$ :} Thanks to the decomposition (\ref{decomposition}) we get
 \begin{equation}
   Q_{n,m}^{+,0}(t) \leq \frac14 2^{-n/2}2^{nH(2k-1)}\sum_{k'=1}^4 \sum_{j,j'= \lfloor 2^{m/2}t\rfloor 2^{\frac{n-m}{2}}}^{\lfloor 2^{n/2}t \rfloor -1}\Omega^{(k',0)}_n(j,j') \label{Omega1}
  \end{equation}
 
 We will study only the term corresponding to $\Omega^{(2,0)}_n(j,j')$  in (\ref{Omega1}), which is representative to the difficulty. It is given by
 
 \begin{eqnarray*}
  && \frac14 2^{-n/2}2^{nH(2k-1)}\sum_{j,j'= \lfloor 2^{m/2}t\rfloor 2^{\frac{n-m}{2}}}^{\lfloor 2^{n/2}t \rfloor -1} \sum_{a=0}^{2(2k-1)}\binom{2(2k-1)}{a} \big| E[f^{(a)}(X^+_{j2^{-n/2}})\\
  &&  \times f^{(2(2k-1)-a)}(X^+_{(j'+1)2^{-n/2}})]\big| \big| \big \langle \xi_{j2^{-n/2}}^{\otimes a}\tilde{\otimes}\xi_{(j'+1)2^{-n/2}}^{\otimes (2(2k-1)-a)}; \delta_{(j+1)2^{-n/2}}^{\otimes (2k-1)}\otimes\delta_{(j'+1)2^{-n/2}}^{\otimes (2k-1)}\big \rangle \big|\notag\\
 &\leq &  C 2^{-n/2}2^{nH(2k-1)}\sum_{j,j'= \lfloor 2^{m/2}t\rfloor 2^{\frac{n-m}{2}}}^{\lfloor 2^{n/2}t \rfloor -1} \sum_{a=0}^{2(2k-1)}\big| \big \langle \xi_{j2^{-n/2}}^{\otimes a}\tilde{\otimes}\xi_{(j'+1)2^{-n/2}}^{\otimes (2(2k-1)-a)};  \\
 && \hskip9cm \delta_{(j+1)2^{-n/2}}^{\otimes (2k-1)}\otimes\delta_{(j'+1)2^{-n/2}}^{\otimes (2k-1)}\big \rangle \big|. \notag
 \end{eqnarray*}
 We define $E_n^{(a,k)}(j,j'):= \big| \big \langle \xi_{j2^{-n/2}}^{\otimes a}\tilde{\otimes}\xi_{(j'+1)2^{-n/2}}^{\otimes (2(2k-1)-a)}; \delta_{(j+1)2^{-n/2}}^{\otimes (2k-1)}\otimes\delta_{(j'+1)2^{-n/2}}^{\otimes (2k-1)}\big \rangle \big|.$
 By (\ref{0}), we thus get, with $\tilde{c}_a$  some combinatorial constants,
  \begin{eqnarray*}
  E_n^{(a,k)}(j,j')  &\leq& \tilde{c}_a\,2^{-nH(4k-3)}
  \big(
   |\langle \xi_{j2^{-n/2}}; \delta_{(j+1)2^{-n/2}}\rangle|
  + | \langle \xi_{j2^{-n/2}}; \delta_{(j'+1)2^{-n/2}}\rangle|\\
&&  \hskip1cm+  |\langle \xi_{(j'+1)2^{-n/2}}; \delta_{(j+1)2^{-n/2}}\rangle|
  + | \langle \xi_{(j'+1)2^{-n/2}}; \delta_{(j'+1)2^{-n/2}}\rangle|
  \big).
  \end{eqnarray*}
For instance,
  we can write
  \begin{eqnarray*}
 && \sum_{j,j'=\lfloor 2^{m/2}t\rfloor 2^{\frac{n-m}{2}}}^{\lfloor 2^{n/2}t\rfloor -1}
 | \langle \xi_{(j'+1)2^{-n/2}}; \delta_{(j+1)2^{-n/2}}\rangle|\\
&  =&2^{-nH-1} \sum_{j,j'=\lfloor 2^{m/2}t\rfloor 2^{\frac{n-m}{2}}}^{\lfloor 2^{n/2}t\rfloor -1}
 \big|
 (j+1)^{2H}-j^{2H}+|j'-j+1|^{2H}-|j'-j|^{2H}
 \big|\\
 &\leq& 2^{-nH-1} \sum_{j,j'=\lfloor 2^{m/2}t\rfloor 2^{\frac{n-m}{2}}}^{\lfloor 2^{n/2}t\rfloor -1}
 \big((j+1)^{2H}-j^{2H}\big)\\
 &&+
 2^{-nH-1} \sum_{\lfloor 2^{m/2}t\rfloor 2^{\frac{n-m}{2}}\leq j\leq j'\leq \lfloor 2^{n/2}t\rfloor-1}
 \big((j'-j+1)^{2H}-(j'-j)^{2H}\big)\\
  &&+
 2^{-nH-1} \sum_{\lfloor 2^{m/2}t\rfloor 2^{\frac{n-m}{2}}\leq j'<j\leq \lfloor 2^{n/2}t\rfloor-1}
 \big((j-j')^{2H}-(j-j'-1)^{2H}\big)\\
 &\leq&\frac32\,2^{-nH}\big(\lfloor 2^{n/2}t\rfloor - \lfloor 2^{m/2}t\rfloor 2^{\frac{n-m}{2}}\big) \lfloor 2^{n/2}t\rfloor ^{2H} \leq \frac{3t^{2H}}2\big( 2^{n/2}t - \lfloor 2^{m/2}t\rfloor 2^{\frac{n-m}{2}}\big).
  \end{eqnarray*}
  Similarly,
  \begin{eqnarray*}
   \sum_{j,j'=\lfloor 2^{m/2}t\rfloor 2^{\frac{n-m}{2}}}^{\lfloor 2^{n/2}t\rfloor -1}
  |\langle \xi_{j2^{-n/2}}; \delta_{(j+1)2^{-n/2}}\rangle|&\leq&\frac{3t^{2H}}2\big( 2^{n/2}t - \lfloor 2^{m/2}t\rfloor 2^{\frac{n-m}{2}}\big);\\
   \sum_{j,j'=\lfloor 2^{m/2}t\rfloor 2^{\frac{n-m}{2}}}^{\lfloor 2^{n/2}t\rfloor -1}
  |\langle \xi_{j2^{-n/2}}; \delta_{(j'+1)2^{-n/2}}\rangle|&\leq&\frac{3t^{2H}}2\big( 2^{n/2}t - \lfloor 2^{m/2}t\rfloor 2^{\frac{n-m}{2}}\big);\\
     \sum_{j,j'=\lfloor 2^{m/2}t\rfloor 2^{\frac{n-m}{2}}}^{\lfloor 2^{n/2}t\rfloor -1}
  |\langle \xi_{(j'+1)2^{-n/2}}; \delta_{(j'+1)2^{-n/2}}\rangle|&\leq&\frac{3t^{2H}}2\big( 2^{n/2}t - \lfloor 2^{m/2}t\rfloor 2^{\frac{n-m}{2}}\big).
  \end{eqnarray*}
  As a consequence, we deduce
  \begin{equation}
  Q_{n,m}^{(+,0)}(t)\leq  C\,2^{-nH(2k-2)} (t - \lfloor 2^{m/2}t\rfloor 2^{\frac{-m}{2}}) \leq C\,2^{-nH(2k-2)}2^{-m/2}. \label{grand4}
  \end{equation}
  \end{itemize}
  Combining (\ref{grand2}), (\ref{grand3}) and (\ref{grand4}) finally shows 
  \begin{eqnarray*}
  && E\big[ \big(J^{+,k}_{n,m}(t)\big)^2 \big] \leq C \bigg( 2^{-n/2} + 2^{-m/2}+ 2^{-nH(2k-2)}(2^{-n/2} + 2^{-m/2}) \\
  && \hskip4cm + 2^{-nH(2k-2)} 2^{-m/2} \big).
  \end{eqnarray*}
  So, we deduce that $J^{+,k}_{n,m}(t)$ converges to $0$ in $L^2(\Omega)$ as $n\to\infty$ and then $m\to\infty$.
 Finally, thanks to (a), (b) and (c), (\ref{principal conv2}) holds true. 
 \end{proof}

 \subsection{\underline{Step 2}: Limit of $2^{-n/4}W_{n}^{(1)}(f,Y_{T_{\lfloor 2^n t\rfloor,n}})$ }
 Thanks to  (\ref{H >1/6 , r=1}), for $H> \frac16$, $2^{-\frac{nH}{2}}W_{n}^{(1)}(f,Y_{T_{\lfloor 2^n t\rfloor,n}}) \underset{n\to \infty}{\overset{P}{\longrightarrow}}\int_0^{Y_t} f(X_s)d^{\circ}X_s$. Thus, since $H< \frac12$, we deduce that 
 \begin{equation}
 2^{-n/4}W_{n}^{(1)}(f,Y_{T_{\lfloor 2^n t\rfloor,n}}) \underset{n\to \infty}{\overset{P}{\longrightarrow}}0. \label{H>1/6 , r=1, fin}
\end{equation}  

\subsection{\underline{Step 3}: Moment bounds for $W_n^{(2r-1)}(f,\cdot)$}
We recall the following result from \cite{RZ2}. Fix an integer $r\geq 1$ as well as a function $f\in C^\infty_b$. There exists a constant $c >0$   such that, for all real numbers $s<t$ and all $n\in\N$,
\begin{eqnarray*}
E\big[ \big(W_{n}^{(2r-1)}(f,t) -W_{n}^{(2r-1)}(f,s) \big)^2\big]
&\leq& c\,\max(|s|^{2H},|t|^{2H})\big(|t-s|2^{n/2}+1\big).
\end{eqnarray*}

\subsection{\underline{Step 4}: Last step in the proof of (\ref{fdd thm1})}
Following \cite{kh-lewis1}, we introduce the following natural definition for two-sided stochastic integrals: for $u\in \R$, let
\begin{eqnarray}\label{integrale}
\int_0^uf(X_s)dW_s =\left\{
\begin{array}{lcl}
\int_0^uf(X^+_s)dW^+_s  &&\mbox{if $ u \geq 0$}\\
 \int_0^{-u}f(X^-_s)dW^-_s  &&\mbox{if $u < 0$}
  \end{array}
\right.,
\end{eqnarray}
where $W^+$ and $W^-$ are defined in Proposition \ref{principal proposition}, $X^+$ and $X^-$ are defined in Section 4, and $\int_0^\cdot f(X^{\pm}_s)dW^{\pm}_s$ must be  understood in the Wiener-It\^o sense.

 Using (\ref{transforming}), (\ref{H>1/6 , r=1, fin}), the conclusion of Step 3 (to pass from $Y_{T_{\lfloor 2^n t\rfloor,n}}$ to $Y_t$)  and since by \cite[Lemma 2.3]{kh-lewis1}, we have $Y_{T_{\lfloor 2^n t\rfloor,n}} \overset{L^2}{\longrightarrow} Y_t$ as $n \to \infty$, we deduce that the limit
of
$2^{-n/4}V_n^{(2r-1)}(f,t) $
is the same as that of
\[
2^{-n/4}\sum_{l=2}^{r}\kappa_{r,l}W_{n}^{(2l-1)}(f,Y_t).
\]

Thus, the proof of (\ref{fdd thm1}) follow directly from  (\ref{principal conv2}), the definition of the integral in (\ref{integrale}),  as well as the fact that $X$, $W$ and $Y$ are independent.

\section{Proof of (\ref{L2 thm1})}

We suppose that $H > \frac12$.  The proof of (\ref{L2 thm1}) will be done in several steps:

\subsection{\underline{Step 1}: Limits and moment bounds for $W_n^{(2i-1)}(f,\cdot)$}
We recall the following It\^o-type formula from \cite[Theorem 4]{NNT} (see also \cite[Theorem 1.3]{RZ3} for an extension  of this formula to the bi-dimensional case). For  all $t\in \R$, the following change-of-variable formula holds true for $H > \frac12$
\begin{equation}
F(X_t)-F(0) = \int_0^t f(X_s)d^{\circ}X_s, \label{limit 1'}
\end{equation}
where $F$ is a primitive of $f$ and $\int_0^t f(X_s)d^{\circ}X_s$ is the Stratonovich integral of $f(X)$ with respect to $X$ defined as the limit in probability of $2^{-\frac{nH}{2}}W_{n}^{(1)}(f,t)$ as $n \to \infty$.

For the rest of the proof, we suppose that  $f\in C_b^\infty$. The following proposition will play a pivotal role in the proof of (\ref{L2 thm1}).
\begin{prop}
 There exists a positive constant $C$, independent of $n$ and $t$, such that
for all $i\geq 1$ and $t\in \R$, we have
\begin{equation}
 E\big[\big(2^{-\frac{nH}{2}}W_{n}^{(2i-1)}(f,t)\big)^2\big] 
\leq  C\: \psi(t,H,i,n),\label{limit 2}
\end{equation}
where, we have
\begin{eqnarray*}
&& \psi(t,H,i,n):= |t|^{(2H-1)(4i-3)}\:|t|^{2H + 1}\:2^{-n(2i-2)(1-H)} \notag\\
&& + C\sum_{a=1}^{2i-2}\bigg(\big[ |t|(1+n) +t^2\big] \:|t|^{2(2H-1)(2i-1-a)}\:2^{-\frac{n}{2}(2H-1)}\:2^{-n(1-H)[2i-1-a]} \notag\\
 &&+  |t|^{2(1-(1-H)a)}\:|t|^{2(2H-1)(2i-1-a)}\:2^{-n(1-H)[2i-2]}{\bf 1}_{\{H > 1 - \frac{1}{2a}\}}\bigg) + C\big[ |t|(1+n) +t^2\big] 2^{-n(H- \frac12)} \notag\\
&& + C\:|t|^{2(1-(1-H)(2i-1))}\: 2^{-n(1-H)(2i-2)}{\bf 1}_{\{H > 1 - \frac{1}{(4i-2)}\}}.
\end{eqnarray*}
\end{prop}

\begin{proof}{}
 Set $
\phi_n(j,j'):= \Delta_{j,n} f(X)\Delta_{j',n} f(X),$ where we recall that
 $\Delta_{j,n} f(X):=  \frac12(f(X_{j2^{-n/2}})+ f(X_{(j+1)2^{-n/2}})$. Fix $t \geq 0$ (the proof in the case $t < 0$ is similar), for all $i\geq 1$, we have
\begin{eqnarray*}
 &&E\big[ \big(2^{-\frac{nH}{2}}W_{n}^{(2i-1)}(f,t) \big)^2 \big]= E\big[ \big(2^{-\frac{nH}{2}} W_{+,n}^{(2i-1)}(f,t)\big)^2\big] \notag\\
 &=& 2^{-nH}\sum_{j,j'=0}^{\lfloor 2^{\frac{n}{2}} t \rfloor -1}E\bigg(\phi_n(j,j')H_{2i-1}\big(X^{n,+}_{j+1} - X^{n,+}_{j}\big) H_{2i-1}\big(X^{n,+}_{j'+1} - X^{n,+}_{j'}\big)\bigg)\notag\\
 &=& 2^{-nH(1 - (2i-1))}\sum_{j,j'=0}^{\lfloor 2^{\frac{n}{2}} t \rfloor -1}E\bigg(\phi_n(j,j')I_{2i-1}\big(\delta_{(j+1)2^{-n/2}}^{\otimes 2i-1}\big) I_{2i-1}\big(\delta_{(j'+1)2^{-n/2}}^{\otimes 2i-1}\big)\bigg)\notag\\
 \end{eqnarray*}
 \begin{eqnarray}
 &=& 2^{-nH(2-2i)}\sum_{a=0}^{2i-1} a! \binom{2i-1}{a}^2\sum_{j,j'=0}^{\lfloor 2^{\frac{n}{2}} t \rfloor -1}E\bigg(\phi_n(j,j')\notag\\
 && \hspace{2cm} \times I_{4i-2-2a}\big(\delta_{(j+1)2^{-n/2}}^{\otimes 2i-1-a}\otimes\delta_{(j'+1)2^{-n/2}}^{\otimes 2i-1-a}\big)\bigg) \langle \delta_{(j+1)2^{-n/2}}, \delta_{(j'+1)2^{-n/2}} \rangle^a\notag\\
 &=& 2^{-nH(2 - 2i)}\sum_{a=0}^{2i-1} a! \binom{2i-1}{a}^2\sum_{j,j'=0}^{\lfloor 2^{\frac{n}{2}} t \rfloor -1}E\bigg(\bigg \langle D^{4i-2-2a}\big(\phi_n(j,j')\big) ,\notag\\
 && \hspace{2cm} \delta_{(j+1)2^{-n/2}}^{\otimes 2i-1-a}\otimes\delta_{(j'+1)2^{-n/2}}^{\otimes 2i-1-a}\bigg\rangle \bigg) \langle \delta_{(j+1)2^{-n/2}}, \delta_{(j'+1)2^{-n/2}} \rangle^a\notag\\
 &=& \sum_{a=0}^{2i-1}a! \binom{2i-1}{a}^2 Q_n^{(i,a)}(t),\label{moment bounds}
 \end{eqnarray}
 with obvious notation at the last equality and  with the third equality following from (\ref{linear-isometry}), the fourth one from (\ref{product}) and the fifth one from (\ref{adjoint}). We have the following estimates.
  
 \begin{itemize}
 
 \item \underline{Case $a= 2i-1$} 
 \begin{eqnarray*}
 |Q_n^{(i,2i-1)}(t)| &\leq & 2^{-nH(2 - 2i)}\sum_{j,j'=0}^{\lfloor 2^{\frac{n}{2}} t \rfloor -1}E\big( \big|\phi_n(j,j')\big| \big)\\
 && \hspace{4cm} \times \big|\langle \delta_{(j+1)2^{-n/2}}, \delta_{(j'+1)2^{-n/2}} \rangle\big|^{2i-1}\\
 & \leq & C 2^{-nH(2-2i)}\sum_{j,j'=0}^{\lfloor 2^{\frac{n}{2}} t \rfloor -1}\big|\langle \delta_{(j+1)2^{-n/2}}, \delta_{(j'+1)2^{-n/2}} \rangle\big|^{2i-1}.
 \end{eqnarray*}
Now, we distinguish three cases: 
\begin{itemize}

\item [(a)] \underline{If $H < 1 - \frac{1}{(4i-2)}:$} by (\ref{2}) we have

\begin{eqnarray*}
 |Q_n^{(i,2i-1)}(t)| & \leq & C\: t\: 2^{-nH(2-2i)}\: 2^{n(\frac12 -(2i-1)H)}= C\: t\: 2^{-n(H- \frac12)}.
 \end{eqnarray*}

\item[(b)] \underline{If $H = 1 - \frac{1}{(4i-2)}:$}  by (\ref{3}) we have

\begin{eqnarray*}
 |Q_n^{(i,2i-1)}(t)| & \leq & C [ t(1+n) +t^2] 2^{-nH(2-2i)}\: 2^{n(\frac12 -(2i-1)H)}\\
 &=& C [ t(1+n) +t^2] 2^{-n(H- \frac12)}.
 \end{eqnarray*}

\item[(c)] \underline{If $H > 1 - \frac{1}{(4i-2)}:$} by (\ref{4}) we have

\begin{eqnarray*}
 |Q_n^{(i,2i-1)}(t)| & \leq & C \:t \:2^{-nH(2-2i)}\: 2^{n(\frac12 -(2i-1)H)}\\
 && + C\:t^{2-(2-2H)(2i-1)}\:2^{-nH(2-2i)}\:2^{n(1-(2i-1))}\\
 &=& C\:t\: 2^{-n(H- \frac12)} + C\:t^{2(1-(1-H)(2i-1))}\: 2^{-n(1-H)(2i-2)}.
 \end{eqnarray*}
So, we deduce that 
\begin{eqnarray}
 |Q_n^{(i,2i-1)}(t)| &\leq& C \big[ |t|(1+n) +t^2\big] 2^{-n(H- \frac12)} \label{majoration 1}\\
&& + C\:|t|^{2(1-(1-H)(2i-1))}\: 2^{-n(1-H)(2i-2)}{\bf 1}_{\{H > 1 - \frac{1}{(4i-2)}\}} \notag
\end{eqnarray}
\end{itemize}

\item \underline{Preparation to the cases where $0\leq a \leq 2i-2$} \\
  
 Thanks to (\ref{Leibnitz1}) we have
 \begin{eqnarray}
 && D^{4i-2-2a}\big(\phi_n(j,j')\big)= D^{4i-2-2a}\big( \Delta_{j,n} f(X)\Delta_{j',n} f(X)\big)\leq C\sum_{l=0}^{4i-2-2a}\notag\\
 && \bigg(f^{(l)}(X_{j2^{-n/2}})\varepsilon_{j2^{-n/2}}^{\otimes l} + f^{(l)}(X_{(j+1)2^{-n/2}})\varepsilon_{(j+1)2^{-n/2} }^{\otimes l}\bigg) \tilde{\otimes}\notag
 \end{eqnarray}
 \begin{eqnarray}
 && \bigg(f^{(4i-2-2a-l)}(X_{j'2^{-n/2}})\varepsilon_{j'2^{-n/2}}^{\otimes 4i-2-2a -l} + f^{(4i-2-2a-l)}(X_{(j'+1)2^{-n/2}})\varepsilon_{(j'+1)2^{-n/2} }^{\otimes 4i-2-2a -l}\bigg)\notag\\
 && = C \sum_{l=0}^{4i-2-2a}\bigg(f^{(l)}(X_{j2^{-n/2}})f^{(4i-2-2a-l)}(X_{j'2^{-n/2}})\varepsilon_{j2^{-n/2}}^{\otimes l}\tilde{\otimes}\varepsilon_{j'2^{-n/2}}^{\otimes 4i-2-2a -l}+ f^{(l)}(X_{j2^{-n/2}})\notag\\
 &&\times f^{(4i-2-2a-l)}(X_{(j'+1)2^{-n/2}}) \varepsilon_{j2^{-n/2}}^{\otimes l}\tilde{\otimes}\varepsilon_{(j'+1)2^{-n/2} }^{\otimes 4i-2-2a -l}+f^{(l)}(X_{(j+1)2^{-n/2}})f^{(4i-2-2a-l)}(X_{j'2^{-n/2}})\notag\\
 && \times \varepsilon_{(j+1)2^{-n/2} }^{\otimes l}\tilde{\otimes}\varepsilon_{j'2^{-n/2}}^{\otimes 4i-2-2a -l} + f^{(l)}(X_{(j+1)2^{-n/2}})f^{(4i-2-2a-l)}(X_{(j'+1)2^{-n/2}})\varepsilon_{(j+1)2^{-n/2} }^{\otimes l}\tilde{\otimes}\notag\\
 &&\varepsilon_{(j'+1)2^{-n/2} }^{\otimes 4i-2-2a -l}\bigg)\label{Leibnitz2}
 \end{eqnarray}
  So, we have
 
 \item \underline{Case $ 1\leq a \leq 2i-2$} 
 
 \begin{eqnarray*}
 &&|Q_n^{(i,a)}(t)|\\
  &\leq & C 2^{-nH(2-2i)}\sum_{l=0}^{4i-2-2a}\sum_{j,j'=0}^{\lfloor 2^{\frac{n}{2}} t \rfloor -1} \\
 && \bigg| \bigg \langle \bigg(\varepsilon_{j2^{-n/2}}^{\otimes l} + \varepsilon_{(j+1)2^{-n/2} }^{\otimes l}\bigg) \tilde{\otimes}\bigg(\varepsilon_{j'2^{-n/2}}^{\otimes 4i-2-2a -l} + \varepsilon_{(j'+1)2^{-n/2} }^{\otimes 4i-2-2a -l}\bigg),\\
 &&  \delta_{(j+1)2^{-n/2}}^{\otimes 2i-1-a}\otimes\delta_{(j'+1)2^{-n/2}}^{\otimes 2i-1-a} \bigg \rangle \bigg|\big|\langle \delta_{(j+1)2^{-n/2}}, \delta_{(j'+1)2^{-n/2}} \rangle\big|^a\\
 &\leq & C \:t^{(2H-1)(4i-2-2a)}\:2^{-nH(2-2i)}(2^{-\frac{n}{2}})^{4i-2-2a}\sum_{j,j'=0}^{\lfloor 2^{\frac{n}{2}} t \rfloor -1}\big|\langle \delta_{(j+1)2^{-n/2}}, \delta_{(j'+1)2^{-n/2}} \rangle\big|^a,\\
 \end{eqnarray*}
 where we have the first inequality because $f\in C^\infty_b$ and thanks to (\ref{Leibnitz2}), and the second one thanks to (\ref{1}) and (\ref{1'}). Now, we distinguish three cases: 
\begin{itemize}

\item [(a)] \underline{If $H < 1 - \frac{1}{2a}:$} by (\ref{2}) we have

\begin{eqnarray*}
 |Q_n^{(i,a)}(t)| & \leq & C\: t\:t^{(2H-1)(4i-2-2a)}\:2^{-nH(2-2i)}2^{-n(2i-1-a)}\: 2^{n(\frac12 -aH)}\\
 &=& C\:t^{2(2H-1)(2i-1-a)+1}\:2^{-\frac{n}{2}(2H-1)}\:2^{-n(1-H)[2i-1-a]}.
 \end{eqnarray*}

\item[(b)] \underline{If $H = 1 - \frac{1}{2a}:$}  by (\ref{3}) we have

\begin{eqnarray*}
 &&|Q_n^{(i,a)}(t)|\\
  & \leq & C [ t(1+n) +t^2] \:t^{(2H-1)(4i-2-2a)}\:2^{-nH(2-2i)}2^{-n(2i-1-a)}\: 2^{n(\frac12 -aH)}\\
 &=& C[ t(1+n) +t^2] \:t^{2(2H-1)(2i-1-a)}\:2^{-\frac{n}{2}(2H-1)}\:2^{-n(1-H)[2i-1-a]}.
 \end{eqnarray*}

\item[(c)] \underline{If $H > 1 - \frac{1}{2a}:$} by (\ref{4}) we have

\begin{eqnarray*}
 |Q_n^{(i,a)}(t)| & \leq & C\: t\:t^{(2H-1)(4i-2-2a)}\:2^{-nH(2-2i)}2^{-n(2i-1-a)}\: 2^{n(\frac12 -aH)}\\
 && + C\: t^{2-(2-2H)a}\:t^{(2H-1)(4i-2-2a)}\:2^{-nH(2-2i)}2^{-n(2i-1-a)}\: 2^{n(1-a)}\\
 &=& C\:t^{2(2H-1)(2i-1-a)+1}\:2^{-\frac{n}{2}(2H-1)}\:2^{-n(1-H)[2i-1-a]}\\
 && + C\: t^{2(1-(1-H)a)}\:t^{2(2H-1)(2i-1-a)}\:2^{-n(1-H)[2i-2]}.
 \end{eqnarray*}
 So, we deduce that
 \begin{eqnarray}
 |Q_n^{(i,a)}(t)| & \leq & C\big[ |t|(1+n) +t^2\big] \:|t|^{2(2H-1)(2i-1-a)}\:2^{-\frac{n}{2}(2H-1)}\:2^{-n(1-H)[2i-1-a]} \notag\\
 &&+ C\: |t|^{2(1-(1-H)a)}\:|t|^{2(2H-1)(2i-1-a)}\:2^{-n(1-H)[2i-2]}{\bf 1}_{\{H > 1 - \frac{1}{2a}\}}\notag\\ \label{majoration 2}
 \end{eqnarray}
\end{itemize}

\item \underline{Case $a=0$}
\begin{eqnarray*}
Q_n^{(i,0)}(t)= 2^{-nH(2 - 2i)}\sum_{j,j'=0}^{\lfloor 2^{\frac{n}{2}} t \rfloor -1}E\bigg(\bigg \langle D^{4i-2}\big(\phi_n(j,j')\big) ,\delta_{(j+1)2^{-n/2}}^{\otimes 2i-1}\otimes\delta_{(j'+1)2^{-n/2}}^{\otimes 2i-1}\bigg\rangle \bigg). 
\end{eqnarray*}
By (\ref{Leibnitz2}) we deduce that
\begin{eqnarray}
 &&|Q_n^{(i,0)}(t)|\leq  C 2^{-nH(2-2i)}\sum_{l=0}^{4i-2}\sum_{j,j'=0}^{\lfloor 2^{\frac{n}{2}} t \rfloor -1} \big| \big \langle \big(\varepsilon_{j2^{-n/2}}^{\otimes l} + \varepsilon_{(j+1)2^{-n/2} }^{\otimes l}\big) \notag \\         
 && \tilde{\otimes}\big(\varepsilon_{j'2^{-n/2}}^{\otimes 4i-2 -l} + \varepsilon_{(j'+1)2^{-n/2} }^{\otimes 4i-2 -l}\big),
 \delta_{(j+1)2^{-n/2}}^{\otimes 2i-1}\otimes\delta_{(j'+1)2^{-n/2}}^{\otimes 2i-1} \big \rangle \big|. \label{Q_n^0}
 \end{eqnarray}
 We define 
 \begin{eqnarray*}
 E_n^{(i,l)}(j,j')&:=& \big| \big \langle \big(\varepsilon_{j2^{-n/2}}^{\otimes l} + \varepsilon_{(j+1)2^{-n/2}}^{\otimes l}\big) \tilde{\otimes}\big(\varepsilon_{j'2^{-n/2}}^{\otimes 4i-2 -l} + \varepsilon_{(j'+1)2^{-n/2}}^{\otimes 4i-2 -l}\big),\\
 && \hspace{4cm} \delta_{(j+1)2^{-n/2}}^{\otimes 2i-1}\otimes\delta_{(j'+1)2^{-n/2}}^{\otimes 2i-1} \big \rangle \big|.
 \end{eqnarray*}
   
 Observe that by (\ref{1}) and (\ref{1'}), we have
 \begin{eqnarray*}
 && E_n^{(i,l)}(j,j') \leq C t^{(2H-1)(4i-3)}(2^{-\frac{n}{2}})^{4i-3}\: \bigg(\big|\big \langle \big(\varepsilon_{j2^{-n/2}} + \varepsilon_{(j+1)2^{-n/2} }\big),\delta_{(j'+1)2^{-n/2}} \big \rangle \big|\\
 & +& \big|\big \langle \big(\varepsilon_{j'2^{-n/2}} + \varepsilon_{(j'+1)2^{-n/2} }\big),\delta_{(j+1)2^{-n/2}} \big \rangle \big| + \big|\big \langle \big(\varepsilon_{j2^{-n/2}} + \varepsilon_{(j+1)2^{-n/2} }\big),\delta_{(j+1)2^{-n/2}} \big \rangle \big|\\
& +& \big|\big \langle \big(\varepsilon_{j'2^{-n/2}} + \varepsilon_{(j'+1)2^{-n/2} }\big),\delta_{(j'+1)2^{-n/2}} \big \rangle \big| \bigg).
 \end{eqnarray*}
By combining these previous estimates with (\ref{Q_n^0}), (\ref{5}) and (\ref{6}) we deduce that
\begin{eqnarray}
 |Q_n^{(i,0)}(t)|&\leq & C\:|t|^{(2H-1)(4i-3)}\:|t|^{2H + 1}\:2^{-nH(2-2i)}\:(2^{-\frac{n}{2}})^{4i-3}\:2^{\frac{n}{2} }\notag\\
 &=& C\:|t|^{(2H-1)(4i-3)}\:|t|^{2H + 1}\:2^{-n(2i-2)(1-H)}. \label{majoration 3}
  \end{eqnarray}
  By combining (\ref{moment bounds}) with (\ref{majoration 1}), (\ref{majoration 2}) and (\ref{majoration 3}), we deduce that (\ref{limit 2}) holds true.
 \end{itemize}
\end{proof}

\subsection{\underline{Step 2}: Limit of $2^{-\frac{nH}{2}}W_{n}^{(2i-1)}(f,Y_{T_{\lfloor 2^n t\rfloor,n}})$}
Let us prove that for $i\geq 2$, 
\begin{equation}
2^{-\frac{nH}{2}}W_{n}^{(2i-1)}(f,Y_{T_{\lfloor 2^n t\rfloor,n}}) \underset{n\to \infty}{\overset{L^2}{\longrightarrow}} 0. \label{W_n}
\end{equation}
 Due to the independence between $X$ and $Y$ and thanks to (\ref{limit 2}), we have
\begin{eqnarray*}
&& E\big[\big(2^{-\frac{nH}{2}}W_{n}^{(2i-1)}(f,Y_{T_{\lfloor 2^n t\rfloor,n}})\big)^2\big] = E\bigg[E\big[\big(2^{-\frac{nH}{2}}W_{n}^{(2i-1)}(f,Y_{T_{\lfloor 2^n t\rfloor,n}})\big)^2 | Y \big] \bigg]\\
& \leq& C E\big[\psi(Y_{T_{\lfloor 2^n t\rfloor,n}},H,i,n)\big].
\end{eqnarray*}
It suffices to prove that 
\begin{equation}
E\big[\psi(Y_{T_{\lfloor 2^n t\rfloor,n}},H,i,n)\big] \underset{n\to \infty}{\longrightarrow} 0. \label{fi_n Y_n}
\end{equation}
For simplicity, we write $Y_n(t)$ instead of $Y_{T_{\lfloor 2^n t\rfloor,n}}$. We have 
\begin{eqnarray}
&& E[\psi(Y_n(t),H,i,n)]= E\big[|Y_n(t)|^{(2H-1)(4i-3)}\:|Y_n(t)|^{2H + 1}\big]\:2^{-n(2i-2)(1-H)} \label{big fi_n}\\
&&+ C\sum_{a=1}^{2i-2}\bigg(E\bigg[\big[ |Y_n(t)|(1+n) +\big(Y_n(t)\big)^2\big] \:|Y_n(t)|^{2(2H-1)(2i-1-a)}\bigg]\:2^{-\frac{n}{2}(2H-1)}\:2^{-n(1-H)[2i-1-a]} \notag\\
 &&+  E\bigg[|Y_n(t)|^{2(1-(1-H)a)}\:|Y_n(t)|^{2(2H-1)(2i-1-a)}\bigg]\:2^{-n(1-H)[2i-2]}{\bf 1}_{\{H > 1 - \frac{1}{2a}\}}\bigg)\notag\\
 &&+ CE\big[ |Y_n(t)|(1+n) +\big(Y_n(t)\big)^2\big] 2^{-n(H- \frac12)}  + C E\big[|Y_n(t)|^{2(1-(1-H)(2i-1))}\big]\: 2^{-n(1-H)(2i-2)}\notag\\
 && \hspace{8.2cm} \times{\bf 1}_{\{H > 1 - \frac{1}{(4i-2)}\}}.\notag
\end{eqnarray}
Let us prove that, for all $1 \leq a \leq 2i-2$
\[
E\big[|Y_n(t)|^{2(1-(1-H)a)}\:|Y_n(t)|^{2(2H-1)(2i-1-a)}\big]\:2^{-n(1-H)[2i-2]}{\bf 1}_{\{H > 1 - \frac{1}{2a}\}} \underset{n \to \infty}{\longrightarrow} 0,
\]
 (the proof of the convergence to 0 of the other terms in (\ref{big fi_n}) is similar). In fact, by H\"older inequality, we have
 \begin{eqnarray*}
&& E\big[|Y_n(t)|^{2(1-(1-H)a)}\:|Y_n(t)|^{2(2H-1)(2i-1-a)}\big]{\bf 1}_{\{H > 1 - \frac{1}{2a}\}}\\
&& \leq E\big[|Y_n(t)|^{4(1-(1-H)a)}\big]^\frac12 E\big[|Y_n(t)|^{4(2H-1)(2i-1-a)}\big]^\frac12{\bf 1}_{\{H > 1 - \frac{1}{2a}\}}.
\end{eqnarray*}
 Observe that for $ H > 1 - \frac{1}{2a}$ we have $2 < 4(1-(1-H)a)< 4$. So, by H\"older inequality, we deduce that $E\big[|Y_n(t)|^{4(1-(1-H)a)}\big]^\frac12 \leq E\big[(Y_n(t))^4\big]^{\frac12(1-(1-H)a)} \leq C$ for all $n \in \N$, where we have the last inequality by Lemma \ref{bounded sequence}.  On the other hand since $ H> \frac12$ we have $4(2H-1)(2i-1-a) >0$ and it is clear that there exists an integer $k_0> 1$ such that $\frac{2k_0}{4(2H-1)(2i-1-a)} > 1$. Thus, by H\"older inequality, we have $E\big[|Y_n(t)|^{4(2H-1)(2i-1-a)}\big]^\frac12 \leq E\big[ (Y_n(t))^{2k_0}\big]^{\frac{(2H-1)(2i-1-a)}{k_0}}\leq C$ for all $n \in \N$, where we have the last inequality by Lemma \ref{bounded sequence}. Finally, we deduce that
 \[
 E\big[|Y_n(t)|^{2(1-(1-H)a)}\:|Y_n(t)|^{2(2H-1)(2i-1-a)}\big]\:2^{-n(1-H)[2i-2]}{\bf 1}_{\{H > 1 - \frac{1}{2a}\}} \leq C 2^{-n(1-H)[2i-2]} \underset{n\to \infty}{\to} 0.
 \]
 Thus, (\ref{fi_n Y_n}) holds true. 
 
\subsection{\underline{Step 3}: Limit of $V_n^{(1)}(f,\cdot)$}
 Recall that for all $t \geq 0$ and $r\geq 1$,
 \[
V_n^{(r)}(f,t):=\sum_{k=0}^{\lfloor 2^n t \rfloor -1} \frac12\big(f(Z_{T_{k,n}})+f(Z_{T_{k+1,n}})\big)\big[\big(2^{\frac{nH}{2}}(Z_{T_{k+1,n}}-Z_{T_{k,n}})\big)^{r} - \mu_r \big].
\]
We claim that 
\begin{equation}
2^{-\frac{nH}{2}}V_n^{(1)}(f,t)\underset{n\to \infty}{\overset{L^2}{\longrightarrow}} \int_0^{Y_t} f(X_s)d^{\circ}X_s. \label{V_n^1}
\end{equation}
We will make use of the following Taylor's type formula (if interested the reader can find a proof of this formula, e.g., in \cite{GRV} page 1788).  Fix $f \in C_b^{\infty}$, let $F$ be a primitive of $f$. For any $a,\: b \in \mathbb{R}$,
  \begin{eqnarray}
F(b) - F(a) \notag
 &=& \frac12\big(f(a)+f(b)\big)(b-a) -
 \frac1{24}\big(f''(a)+f''(b)\big)(b-a)^3 \\
&&+ O ( |b-a|^{5}),\notag
  \end{eqnarray}
 where $|O ( |b-a|^{5})| \leq C_{F}|b-a|^{5}$, $C_F$ being a constant depending only on $F$.
One can thus write
\begin{eqnarray}
&& F(Z_{T_{\lfloor 2^n t\rfloor,n}}) - F(0)=\sum_{k=0}^{\lfloor 2^{n}t\rfloor -1}\!\!\!\big( F(Z_{T_{k+1,n}}) - F(Z_{T_{k,n}})\big)\notag\\
&\!\!\!=&
2^{-\frac{nH}{2}}V_n^{(1)}(f,t) -\frac{2^{-\frac{3nH}{2}}}{12}
V_n^{(3)}(f'',t)+ \sum_{k=0}^{\lfloor 2^n t \rfloor -1}O ( |Z_{T_{k+1,n}}-Z_{T_{k,n}}|^{5}). \label{Ito ancien}
\end{eqnarray}
Thanks to the Minkowski inequality, we have 
 \begin{eqnarray*}
&& \| \sum_{k=0}^{\lfloor 2^n t \rfloor -1}O ( |Z_{T_{k+1,n}}-Z_{T_{k,n}}|^{5})\|_2 \leq C_F \sum_{k=0}^{\lfloor 2^n t \rfloor -1}\||Z_{T_{k+1,n}}-Z_{T_{k,n}}|^{5}\|_2.
\end{eqnarray*}
Due to the independence between $X$ and $Y$, the self-similarity and the stationarity of increments of $X$, we have 
\begin{eqnarray*}
\||Z_{T_{k+1,n}}-Z_{T_{k,n}}|^{5}\|_2 &=& \big(E[(Z_{T_{k+1,n}}-Z_{T_{k,n}})^{10}]\big)^\frac12 = 
\big(E\big[E[(Z_{T_{k+1,n}}-Z_{T_{k,n}})^{10}\: |\: Y]\big]\big)^\frac12\\
&=& \big(2^{-5nH}E[X_1^{10}]\big)^\frac12 = 2^{-\frac{5nH}{2}} \|X_1^5\|_2.
\end{eqnarray*}
 Finally, thanks to the previous calculation and since $H> \frac12$, we deduce that
\begin{eqnarray}
\| \sum_{k=0}^{\lfloor 2^n t \rfloor -1}O ( |Z_{T_{k+1,n}}-Z_{T_{k,n}}|^{5})\|_2 \leq C_F \sum_{k=0}^{\lfloor 2^n t \rfloor -1} 2^{-\frac{5nH}{2}}\|X_1^{5}\|_2 \leq C_F \|X_1^{5}\|_2 t\,2^{n(1-\frac{5H}{2})} \underset{n\to\infty}{\longrightarrow} 0. \notag\\\label{reste 1}
\end{eqnarray}
By (\ref{transforming}), we have $2^{-\frac{3nH}{2}}V_n^{(3)}(f,t)
= 2^{-\frac{3nH}{2}}W_{n}^{(3)}(f,Y_{T_{\lfloor 2^n t\rfloor,n}}) + 32^{-\frac{3nH}{2}}W_{n}^{(1)}(f,Y_{T_{\lfloor 2^n t\rfloor,n}}).$ By (\ref{W_n}), we have that $2^{-\frac{3nH}{2}}W_{n}^{(3)}(f,Y_{T_{\lfloor 2^n t\rfloor,n}})$ converges to 0 in $L^2$ as $n \to \infty$. By (\ref{limit 2}) and thanks to the independence of $X$ and $Y$, we deduce that 
\begin{eqnarray*}
E\big[\big(2^{-\frac{3nH}{2}}W_{n}^{(1)}(f,Y_{T_{\lfloor 2^n t\rfloor,n}})\big)^2\big]&\leq& C 2^{-2nH}\bigg(2^{-n(H - \frac12)}\big[(1+n)E[|Y_{T_{\lfloor 2^n t\rfloor,n}}|] + E[(Y_{T_{\lfloor 2^n t\rfloor,n}})^2] \big]\\
&& + E\big[|Y_{T_{\lfloor 2^n t\rfloor,n}}|^{2H}\big] + E\big[|Y_{T_{\lfloor 2^n t\rfloor,n}}|^{4H}\big]\bigg),
\end{eqnarray*}
by H\"older inequality and thanks to Lemma \ref{bounded sequence}, we can prove easily that the last quantity converges to  0 as $n \to \infty$. Finally, we get
\begin{equation}
2^{-\frac{3nH}{2}}V_n^{(3)}(f,t) \underset{n\to \infty}{\overset{L^2}{\longrightarrow}} 0. \label{reste 2}
\end{equation}
Now, let us prove that
\begin{equation}
F(Z_{T_{\lfloor 2^n t\rfloor,n}}) - F(0) \underset{n\to \infty}{\overset{L^2}{\longrightarrow}} F(Z_t) -F(0). \label{reste 3}
\end{equation}
In fact, as it has been mentioned in the introduction, $T_{\lfloor 2^n t\rfloor,n} \overset{a.s.}{\longrightarrow}  t$ as $n\to \infty$ (see \cite[Lemma 2.2]{kh-lewis1} for a precise statement), and thanks to the continuity of $F$ as well as the continuity of the paths of $Z$, we have
 \begin{equation}
 F(Z_{T_{\lfloor 2^n t\rfloor,n}}) - F(0) \underset{n\to \infty}{\overset{a.s.}{\longrightarrow}} F(Z_t) - F(0). \label{convergence a.s.}
 \end{equation} 
  In addition, by the mean value theorem, and since $f$ is bounded, we have that
$
\big| F(Z_{T_{\lfloor 2^n t\rfloor,n}}) - F(0) \big| \leq \sup_{x\in \R}|f(x)| |Z_{T_{\lfloor 2^n t\rfloor,n}}|,
$
so, we deduce that 
\[
\|F(Z_{T_{\lfloor 2^n t\rfloor,n}}) - F(0) \|_4 \leq \sup_{x\in \R}|f(x)| \|Z_{T_{\lfloor 2^n t\rfloor,n}}\|_4.
\]
 Due to independence between $X$ and $Y$, and to the self-similarity of $X$, we have $\|Z_{T_{\lfloor 2^n t\rfloor,n}}\|_4 = \|X_{Y_{T_{\lfloor 2^n t\rfloor,n}}}\|_4 = \||Y_{T_{\lfloor 2^n t\rfloor,n}}|^H X_1\|_4 = \||Y_{T_{\lfloor 2^n t\rfloor,n}}|^H\|_4 \|X_1\|_4$. By H\"older inequality, we have $\||Y_{T_{\lfloor 2^n t\rfloor,n}}|^H\|_4 \leq (\|Y_{T_{\lfloor 2^n t\rfloor,n}}\|_4)^H.$ Finally, we have
\[
\|F(Z_{T_{\lfloor 2^n t\rfloor,n}}) - F(0) \|_4 \leq \sup_{x\in \R}|f(x)| \|X_1\|_4 (\|Y_{T_{\lfloor 2^n t\rfloor,n}}\|_4)^H.
\]
Thanks to Lemma \ref{bounded sequence} and to the previous inequality, we deduce that the sequence $\big(F(Z_{T_{\lfloor 2^n t\rfloor,n}}) - F(0)\big)_{n \in \N}$ is bounded in $L^4$. Combining this fact with (\ref{convergence a.s.}) we deduce that (\ref{reste 3}) holds true. 

Finally, combining (\ref{Ito ancien}) with (\ref{reste 1}), (\ref{reste 2}) and (\ref{reste 3}), we deduce that
\[
2^{-\frac{nH}{2}}V_n^{(1)}(f,t)\underset{n\to \infty}{\overset{L^2}{\longrightarrow}} F(Z_t) -F(0).
\]
By (\ref{limit 1'}), we have $F(X_t)-F(0) = \int_0^t f(X_s)d^{\circ}X_s$ which implies that $F(Z_t) -F(0)= \int_0^{Y_t} f(X_s)d^{\circ}X_s$. So, we deduce finally that (\ref{V_n^1}) holds true.

\subsection{\underline{Step 4}: Last step in the proof of (\ref{L2 thm1})}
Thanks to  (\ref{transforming}), we have
\begin{eqnarray*}
V_n^{(2r-1)}(f,t)
&=& \sum_{i=1}^{r}\kappa_{r,i}W_{n}^{(2i-1)}(f,Y_{T_{\lfloor 2^n t\rfloor,n}}).
\end{eqnarray*}
For $r=1$, (\ref{L2 thm1}) holds true by (\ref{V_n^1}). For $r\geq 2$, we have  $2^{-\frac{nH}{2}} V_n^{(2r-1)}(f,t) = \kappa_{r,1} 2^{-\frac{nH}{2}}V_n^{(1)}(f,t) + \sum_{i=2}^{r}\kappa_{r,i}2^{-\frac{nH}{2}}W_{n}^{(2i-1)}(f,Y_{T_{\lfloor 2^n t\rfloor,n}})$. Combining this equality with (\ref{W_n}) and (\ref{V_n^1}) we deduce that (\ref{L2 thm1}) holds true. 

\section{Proof of (\ref{fdd thm2})}
Recall that for all $t \geq 0$ and $r\geq 1$,
 \[
V_n^{(2r)}(f,t):=\sum_{k=0}^{\lfloor 2^n t \rfloor -1} \frac12\big(f(Z_{T_{k,n}})+f(Z_{T_{k+1,n}})\big)\big[\big(2^{\frac{nH}{2}}(Z_{T_{k+1,n}}-Z_{T_{k,n}})\big)^{2r} - \mu_{2r}\big],
\]
and for all $i \in \Z$,
 $\Delta_{i,n} f(X):=  \frac12(f(X_{i2^{-n/2}})+ f(X_{(i+1)2^{-n/2}})$. Thanks to Lemma \ref{lemme-algebric}, we have 
\begin{eqnarray*}
2^{-\frac{n}{2}}V_n^{(2r)}(f,t) &=& 2^{-\frac{n}{2}}\sum_{i\in \Z}\Delta_{i,n} f(X)\big[(X_{i+1}^{(n)} - X_i^{(n)})^{2r} - \mu_{2r}\big] (U_{i,n}(t) + D_{i,n}(t))\\
&=& \sum_{i\in \Z}\Delta_{i,n} f(X)\big[(X_{i+1}^{(n)} - X_i^{(n)})^{2r} - \mu_{2r}\big] \mathcal{L}_{i,n}(t),
\end{eqnarray*}
with obvious notation at the last line.
Fix $t\geq 0$. In order to study the asymptotic behavior of $2^{-\frac{n}{2}}V_n^{(2r)}(t)$ as $n$ tends to infinity (after using the adequate normalization according to the value of the Hurst parameter $H$) , we shall consider (separately) the cases when $n$ is even and when $n$ is odd.

 When $n$ is even, for any even integers $n\geq m\geq 0$ and any integer $p\geq 0$, one can decompose $2^{-\frac{n}{2}}V_n^{(2r)}(t)$ as
\[
2^{-\frac{n}{2}}V_n^{(2r)}(t) = A_{m,n,p}^{(2r)}(t) + B_{m,n,p}^{(2r)}(t) + C_{m,n,p}^{(2r)}(t) + D_{m,n,p}^{(2r)}(t) + E_{n,p}^{(2r)}(t) ,
\]
where
\begin{eqnarray*}
A_{m,n,p}^{(2r)}(t) &=&
 \sum_{-p 2^{m/2} +1\leq j\leq p 2^{m/2}} \sum_{i=(j-1)2^{\frac{n-m}{2}}}^{j2^{\frac{n-m}{2}}-1}
  \Delta_{i,n} f(X)\big[(X_{i+1}^{(n)} - X_i^{(n)})^{2r} - \mu_{2r}\big]\\
  && \hspace{6cm} \times(\mathcal{L}_{i,n}(t) - L_t^{i2^{-n/2}}(Y))\\
B_{m,n,p}^{(2r)}(t) &=&
 \sum_{-p 2^{m/2}+1\leq j\leq p 2^{m/2}}  \sum_{i=(j-1)2^{\frac{n-m}{2}}}^{j2^{\frac{n-m}{2}}-1}
  \Delta_{i,n} f(X)\big[(X_{i+1}^{(n)} - X_i^{(n)})^{2r} - \mu_{2r}\big]\\
  &&\hskip6cm \times(L_t^{i2^{-n/2}}(Y)-L_t^{j2^{-m/2}}(Y))\\
C_{m,n,p}^{(2r)}(t) &=&
 \sum_{-p 2^{m/2}+1\leq j\leq p 2^{m/2}}L_t^{j2^{-m/2}}(Y) \sum_{i=(j-1)2^{\frac{n-m}{2}}}^{j2^{\frac{n-m}{2}}-1}
  (\Delta_{i,n} f(X)-\Delta_{j,m} f(X))\\
  &&\hskip6cm \times \big[(X_{i+1}^{(n)} - X_i^{(n)})^{2r} - \mu_{2r}\big]\\
  \end{eqnarray*}
  \begin{eqnarray*}
 D_{m,n,p}^{(2r)}(t)&=& 
 \sum_{-p 2^{m/2}+1\leq j\leq p 2^{m/2}}\Delta_{j,m} f(X) L_t^{j2^{-m/2}}(Y) \sum_{i=(j-1)2^{\frac{n-m}{2}}}^{j2^{\frac{n-m}{2}}-1}\big[(X_{i+1}^{(n)} - X_i^{(n)})^{2r}\\
  &&\hskip10cm   - \mu_{2r}\big]\\
E_{n,p}^{(2r)}(t) &=&
 \sum_{i\geq p 2^{n/2}} \Delta_{i,n} f(X)\big[(X_{i+1}^{(n)} - X_i^{(n)})^{2r} - \mu_{2r}\big]\mathcal{L}_{i,n}(t)\\
 && + \sum_{i < -p2^{n/2}} \Delta_{i,n} f(X)\big[(X_{i+1}^{(n)} - X_i^{(n)})^{2r} - \mu_{2r}\big]\mathcal{L}_{i,n}(t).
\end{eqnarray*}
We can see  that since we have taken even integers $n\geq m\geq 0$ then $2^{m/2}$, $2^{\frac{n-m}{2}}$ and $2^{n/2}$ are integers as well. This justifies the validity of the previous decomposition.

When $n$ is odd, for any odd integers $n \geq m \geq 0$ we can work with the same decomposition for  $V_n^{(2r)}(t)$. The only difference is that we have to replace the sum $\sum_{-p 2^{m/2} +1\leq j\leq p 2^{m/2}}$ in $A_{m,n,p}^{(2r)}(t)$, $B_{m,n,p}^{(2r)}(t)$, $C_{m,n,p}^{(2r)}(t)$ and $D_{m,n,p}^{(2r)}(t)$ by $\sum_{-p 2^{\frac{m+1}{2}} +1\leq j\leq p 2^{\frac{m+1}{2}}}$. And instead of $\sum_{i\geq p 2^{n/2}}$ and $\sum_{i < -p2^{n/2}}$ in $E_{n,p}^{(2r)}(t)$, we must consider $\sum_{i\geq p 2^{\frac{n+1}{2}}}$ and $\sum_{i < -p2^{\frac{n+1}{2}}}$ respectively. The analysis can then be done {\sl mutatis mutandis}.

Suppose that $\frac14<H\leq \frac12$. Firstly, we will prove that $2^{-\frac{n}{4}}A_{m,n,p}^{(2r)}(t)$, $2^{-\frac{n}{4}}B_{m,n,p}^{(2r)}(t)$, $2^{-\frac{n}{4}}C_{m,n,p}^{(2r)}(t)$ and $2^{-\frac{n}{4}}E_{n,p}^{(2r)}(t)$ converge to $0$ in $L^2$  by letting $n$, then $m$, then $p$ tend to infinity. Secondly, we will study  the f.d.d. convergence in law  of $\big( 2^{-\frac{n}{4}}D_{m,n,p}^{(2r)}(t) \big)_{t\geq 0}$,
which will then be equivalent to the f.d.d. convergence in law  of
$\big( 2^{-\frac{3n}{4}}V_n^{(2r)}(t)\big)_{t\geq 0}$. 
\begin{itemize}
\item[(1)] \underline{$2^{-\frac{n}{4}}A_{m,n,p}^{(2r)}(t)\underset{n\to \infty}{\overset{L^2}{\longrightarrow}} 0$\::}
\end{itemize}
We have, for all $r \in \N^*$, 
\begin{equation}
x^{2r} =\sum_{a=1}^r b_{2r,a}H_{2a}(x)+ \mu_{2r}, \label{hermite polynomial pair}
\end{equation}
 where $H_n$ is the $n$th Hermite polynomial, $\mu_{2r}=E[N^{2r}]$ with $N \sim \mathcal{N}(0,1)$, and $b_{2r,a}$ are some explicit constants (if interested, the reader can find these explicit constants, e.g.,  in \cite[Corollary 1.2]{RZ1}). We deduce that 
\begin{eqnarray}
 A_{m,n,p}^{(2r)}(t)&=& \sum_{a=1}^r b_{2r,a}\,  \sum_{-p 2^{m/2} +1\leq j\leq p 2^{m/2}} \sum_{i=(j-1)2^{\frac{n-m}{2}}}^{j2^{\frac{n-m}{2}}-1}
  \Delta_{i,n} f(X)H_{2a}(X_{i+1}^{(n)} - X_i^{(n)})\notag\\
  && \hspace{6cm} \times(\mathcal{L}_{i,n}(t) - L_t^{i2^{-n/2}}(Y))\notag\\
  &=&\sum_{a=1}^rb_{2r,a}A_{m,n,p,a}^{(2r)}(t), \label{decomposition A}
\end{eqnarray}
with obvious notation at the last line. It suffices to prove that for any fixed $m$ and $p$ and for all $a\in\{1, \ldots,r\}$
\begin{equation}
2^{-\frac{n}{4}}A_{m,n,p,a}^{(2r)}(t)\underset{n\to \infty}{\overset{L^2}{\longrightarrow}} 0. \label{1, H< 3/4}
\end{equation}
Set $
\phi_n(i,i'):= \Delta_{i,n} f(X)\Delta_{i',n} f(X).$ Thanks to (\ref{linear-isometry}), (\ref{adjoint}), (\ref{product}) and to the independence of $X$ and $Y$, we have
\begin{eqnarray}
&&E[(2^{-\frac{n}{4}} A_{m,n,p,a}^{(2r)}(t))^2] = \bigg|2^{2nHa - \frac{n}{2}}\sum_{-p 2^{m/2} +1\leq j, j'\leq p 2^{m/2}} \sum_{i=(j-1)2^{\frac{n-m}{2}}}^{j2^{\frac{n-m}{2}}-1} \sum_{i'=(j'-1)2^{\frac{n-m}{2}}}^{j'2^{\frac{n-m}{2}}-1}E\big[\phi_n(i,i')\notag\\
&& \times I_{2a}(\delta_{(i+1)2^{-\frac{n}{2}}})I_{2a}(\delta_{(i'+1)2^{-\frac{n}{2}}})\big]E\big[(\mathcal{L}_{i,n}(t) - L_t^{i2^{-\frac{n}{2}}}(Y))(\mathcal{L}_{i',n}(t) - L_t^{i'2^{-\frac{n}{2}}}(Y))\big]\bigg|\notag
\end{eqnarray}
\begin{eqnarray}
&\leq & 2^{2nHa -\frac{n}{2}}\sum_{-p 2^{m/2} +1\leq j, j'\leq p 2^{m/2}} \sum_{i=(j-1)2^{\frac{n-m}{2}}}^{j2^{\frac{n-m}{2}}-1} \sum_{i'=(j'-1)2^{\frac{n-m}{2}}}^{j'2^{\frac{n-m}{2}}-1}\big|E\big[\phi_n(i,i')I_{2a}(\delta_{(i+1)2^{-\frac{n}{2}}}^{\otimes 2a})\times\notag\\
&&  I_{2a}(\delta_{(i'+1)2^{-\frac{n}{2}}}^{\otimes 2a})\big]\big|\|\mathcal{L}_{i,n}(t) - L_t^{i2^{-\frac{n}{2}}}(Y)\|_2\times \|\mathcal{L}_{i',n}(t) - L_t^{i'2^{-\frac{n}{2}}}(Y)\|_2\notag\\
&\leq & 2^{2nHa -\frac{n}{2}}\sum_{l=0}^{2a}l!\binom{2a}{l}^2\sum_{-p 2^{m/2} +1\leq j, j'\leq p 2^{m/2}} \sum_{i=(j-1)2^{\frac{n-m}{2}}}^{j2^{\frac{n-m}{2}}-1} \sum_{i'=(j'-1)2^{\frac{n-m}{2}}}^{j'2^{\frac{n-m}{2}}-1}\big|E\big[\phi_n(i,i')\times\notag\\
&& I_{4a-2l}\big(\delta_{(i+1)2^{-n/2}}^{\otimes 2a-l}\otimes\delta_{(i'+1)2^{-n/2}}^{\otimes 2a-l}\big)\big]\big| |\langle \delta_{(i+1)2^{-n/2}}, \delta_{(i'+1)2^{-n/2}} \rangle|^l \|\mathcal{L}_{i,n}(t) - L_t^{i2^{-\frac{n}{2}}}(Y)\|_2 \notag\\
&& \times \|\mathcal{L}_{i',n}(t) - L_t^{i'2^{-\frac{n}{2}}}(Y)\|_2 \notag\\
&=& 2^{2nHa -\frac{n}{2}}\sum_{l=0}^{2a}l!\binom{2a}{l}^2\sum_{-p 2^{m/2} +1\leq j, j'\leq p 2^{m/2}} \sum_{i=(j-1)2^{\frac{n-m}{2}}}^{j2^{\frac{n-m}{2}}-1} \sum_{i'=(j'-1)2^{\frac{n-m}{2}}}^{j'2^{\frac{n-m}{2}}-1}\big|E\big[\big\langle D^{4a-2l}\big(\phi_n(i,i')\big)\notag\\
&& , \delta_{(i+1)2^{-n/2}}^{\otimes 2a-l}\otimes\delta_{(i'+1)2^{-n/2}}^{\otimes 2a-l}\big\rangle\big]\big| |\langle \delta_{(i+1)2^{-n/2}}, \delta_{(i'+1)2^{-n/2}} \rangle|^l  \|\mathcal{L}_{i,n}(t) - L_t^{i2^{-\frac{n}{2}}}(Y)\|_2 \notag\\
&& \times \|\mathcal{L}_{i',n}(t) - L_t^{i'2^{-\frac{n}{2}}}(Y)\|_2\notag\\
&=& \sum_{l=0}^{2a}l!\binom{2a}{l}^2 \Upsilon_n^{(l,a)}(t), \label{reference 1}
\end{eqnarray}
by obvious notation at the last line. By the points 2 and 3 of Proposition \ref{local time}, see also (3.14) in \cite{RZ1} for the detailed proof, we have
\[ \|\mathcal{L}_{i,n}(t) - L_t^{i2^{-\frac{n}{2}}}(Y)\|_2 \leq 2\sqrt{\mu}\|K\|_4 \:t^{1/8} n2^{-n/4}\:2^{-n/8}|i|^{1/4} + 2\|K\|_{4}\|L_{t}^0(Y)\|_2^{1/2}n2^{-n/4}.\]
 Since $ -p 2^{m/2}+1\leq j \leq p 2^{m/2}$ and $(j-1)2^{\frac{n-m}{2}}\leq i\leq j2^{\frac{n-m}{2}}-1$, we deduce that $-p 2^{n/2} \leq i\leq p 2^{n/2} -1$. So, $|i| \leq p2^{n/2}$. Consequently we have that  $|i|^{1/4} \leq p^{1/4}2^{n/8}$, which shows that $ \|\mathcal{L}_{i,n}(t) - L_t^{i2^{-\frac{n}{2}}}(Y)\|_2 \leq C(p^{1/4} +1)n2^{-\frac{n}{4}}$. Finally, we deduce that
\begin{equation}
\|\mathcal{L}_{i,n}(t) - L_t^{i2^{-\frac{n}{2}}}(Y)\|_2\times \|\mathcal{L}_{i',n}(t) - L_t^{i'2^{-\frac{n}{2}}}(Y)\|_2 \leq C(p^{1/4} +1)^2n^22^{-\frac{n}{2}}. \label{temps local bound 1}
\end{equation}
Now, observe that, by the same arguments that has been used to show  (\ref{Leibnitz2}) and since $f\in C_b^\infty$, we have
\begin{eqnarray}
&&\Theta_{i,i',n}^{(a,l)} := \big|E\big[\big\langle D^{4a-2l}\big(\phi_n(i,i') \big) , \delta_{(i+1)2^{-n/2}}^{\otimes 2a-l}\otimes\delta_{(i'+1)2^{-n/2}}^{\otimes 2a-l}\big\rangle\big]\big| \notag\\
&\leq& C\sum_{k=0}^{4a-2l}\binom{4a-2l}{k} \bigg| \big \langle \big(\varepsilon_{i2^{-n/2}}^{\otimes k} + \varepsilon_{(i+1)2^{-n/2} }^{\otimes k}\big) \tilde{\otimes}\big(\varepsilon_{i'2^{-n/2}}^{\otimes 4a-2l -k} + \varepsilon_{(i'+1)2^{-n/2} }^{\otimes 4a-2l -k}\big),\notag\\ &&\hspace{3.5cm}
 \delta_{(i+1)2^{-n/2}}^{\otimes 2a-l}\otimes\delta_{(i'+1)2^{-n/2}}^{\otimes 2a-l} \big \rangle \bigg|.\notag
\end{eqnarray}
Since  $H\leq \frac12$, thanks to (\ref{0}), we have $\Theta_{i,i',n}^{(a,l)}\leq C 2^{-nH(4a-2l)}$. So, by combining (\ref{reference 1}) with (\ref{temps local bound 1}), for $l=0$, we have
\begin{equation} \Upsilon_n^{(0,a)}(t) \leq C 2^{2nHa-\frac{n}{2}} \sum_{i,i'= -p2^{\frac{n}{2}}}^{p2^{\frac{n}{2}} -1} (2^{-4nHa}(p^{\frac14} + 1)^2 n^2 2^{-\frac{n}{2}})\leq C(p(p^{\frac14} + 1))^2n^22^{-2nHa}, \label{upsilon1}
\end{equation}
for $l \neq 0$, we have
\[ \Upsilon_n^{(l,a)}(t) \leq C(p^{\frac14} + 1)^2 n^2 2^{2nHa-n}\bigg(2^{-nH(4a-2l)} \sum_{i,i'= -p2^{\frac{n}{2}}}^{p2^{\frac{n}{2}} -1}|\langle \delta_{(i+1)2^{-n/2}}, \delta_{(i'+1)2^{-n/2}} \rangle|^l\bigg).
\]
By the same arguments that has been used in the proof of (\ref{2}), one can prove that for $H< 1- \frac{1}{2l}$, we have
\begin{eqnarray}
\sum_{i,i'= -p2^{\frac{n}{2}}}^{p2^{\frac{n}{2}} -1}|\langle \delta_{(i+1)2^{-n/2}}, \delta_{(i'+1)2^{-n/2}} \rangle|^l \leq C_{H,l}\,p\, 2^{n(\frac12 -lH)}. \label{new delta 1}
\end{eqnarray}
For $H=\frac{1}{2}$, thanks to (\ref{inner product 2}) and to the discussion of the case $H= \frac12$ after (\ref{alpha}), we have
\begin{equation*}
\sum_{i,i'= -p2^{\frac{n}{2}}}^{p2^{\frac{n}{2}} -1}|\langle \delta_{(i+1)2^{-n/2}}, \delta_{(i'+1)2^{-n/2}} \rangle| \leq 2^{-\frac{n}{2}}\big(p2^{\frac{n}{2}} -(-p2^{\frac{n}{2}})\big)=2p, \label{new delta 2}
\end{equation*}
thus, (\ref{new delta 1}) holds true for $l=1$ and $H= \frac12$.
So, since  $H\leq \frac12$, we deduce that 
\begin{eqnarray}
\sum_{l=1}^{2a}\Upsilon_n^{(l,a)}(t) &\leq& C p(p^{\frac14} + 1)^2 n^2\sum_{l=1}^{2a} 2^{2nHa-n}(2^{-nH(4a-2l)}2^{n(\frac12 -lH)})\notag\\
&& = C p(p^{\frac14} + 1)^2 n^2 2^{-\frac{n}{2}}\sum_{l=1}^{2a}2^{-nH(2a-l)}. \label{upsilon2}
 \end{eqnarray}
 By combining (\ref{reference 1}) with (\ref{upsilon1}) and (\ref{upsilon2}), we deduce that (\ref{1, H< 3/4}) holds true for $H\leq \frac12$.

\begin{itemize}
\item[(2)] \underline{$2^{-\frac{n}{4}}B_{m,n,p}^{(2r)}(t)\overset{L^2}{\longrightarrow} 0$ as $m\to \infty$, uniformly on $n$\::}
\end{itemize}
Using (\ref{hermite polynomial pair}), we get 
\begin{eqnarray}
 B_{m,n,p}^{(2r)}(t)&=& \sum_{a=1}^r b_{2r,a}\,  \sum_{-p 2^{m/2} +1\leq j\leq p 2^{m/2}} \sum_{i=(j-1)2^{\frac{n-m}{2}}}^{j2^{\frac{n-m}{2}}-1}
  \Delta_{i,n} f(X)H_{2a}(X_{i+1}^{(n)} - X_i^{(n)})\notag\\
  && \hspace{6cm} \times(L_t^{i2^{-n/2}}(Y)- L_t^{j2^{-m/2}}(Y) )\notag\\
  &=&\sum_{a=1}^rb_{2r,a}B_{m,n,p,a}^{(2r)}(t), \label{decomposition B}
\end{eqnarray}
with obvious notation at the last line. It suffices to prove that for any fixed $p$ and for all $a\in\{1, \ldots,r\}$
\begin{equation}
2^{-\frac{n}{4}}B_{m,n,p,a}^{(2r)}(t)\underset{m\to\infty}{\overset{L^2}{\longrightarrow}} 0, \label{2, H< 3/4}
\end{equation}
uniformly on $n$.
By the same arguments that has been used to prove (\ref{reference 1}), we get
\begin{eqnarray*}
&&E[(2^{-\frac{n}{4}} B_{m,n,p,a}^{(2r)}(t))^2] \notag\\
&\leq& 2^{2nHa -\frac{n}{2}}\sum_{l=0}^{2a}l!\binom{2a}{l}^2\sum_{-p 2^{m/2} +1\leq j, j'\leq p 2^{m/2}} \sum_{i=(j-1)2^{\frac{n-m}{2}}}^{j2^{\frac{n-m}{2}}-1} \sum_{i'=(j'-1)2^{\frac{n-m}{2}}}^{j'2^{\frac{n-m}{2}}-1}\big|E\big[\big\langle D^{4a-2l}\big(\phi_n(i,i')\big)\notag\\
&& , \delta_{(i+1)2^{-n/2}}^{\otimes 2a-l}\otimes\delta_{(i'+1)2^{-n/2}}^{\otimes 2a-l}\big\rangle\big]\big| |\langle \delta_{(i+1)2^{-n/2}}, \delta_{(i'+1)2^{-n/2}} \rangle|^l  \big|E[(L_t^{i2^{-n/2}}(Y)- L_t^{j2^{-m/2}}(Y)) \notag\\
&& \times (L_t^{i'2^{-n/2}}(Y)- L_t^{j'2^{-m/2}}(Y))]\big|,
\end{eqnarray*}
by Proposition \ref{local time} (point 2) and Cauchy-Schwarz, we have
\begin{eqnarray*}
&&\big|E[(L_t^{i2^{-n/2}}(Y)- L_t^{j2^{-m/2}}(Y))(L_t^{i'2^{-n/2}}(Y)- L_t^{j'2^{-m/2}}(Y))]\big|\\
&\leq & \mu^2 \sqrt{t} \sqrt{|i2^{-n/2}-j2^{-m/2}||i'2^{-n/2}-j'2^{-m/2}|}\leq \mu^2 \sqrt{t}2^{-m/2}.
\end{eqnarray*}
So, we deduce that
\begin{eqnarray}
&&E[(2^{-\frac{n}{4}} B_{m,n,p,a}^{(2r)}(t))^2] \leq C 2^{-\frac{m}{2}} 2^{2nHa -\frac{n}{2}}\sum_{l=0}^{2a}l!\binom{2a}{l}^2\sum_{-p 2^{m/2} +1\leq j, j'\leq p 2^{m/2}} \sum_{i=(j-1)2^{\frac{n-m}{2}}}^{j2^{\frac{n-m}{2}}-1}\notag\\
&& \sum_{i'=(j'-1)2^{\frac{n-m}{2}}}^{j'2^{\frac{n-m}{2}}-1}\big|E\big[\big\langle D^{4a-2l}\big(\phi_n(i,i')\big)
 , \delta_{(i+1)2^{-n/2}}^{\otimes 2a-l}\otimes\delta_{(i'+1)2^{-n/2}}^{\otimes 2a-l}\big\rangle\big]\big| |\langle \delta_{(i+1)2^{-n/2}}, \delta_{(i'+1)2^{-n/2}} \rangle|^l \notag\\ 
 &=& \sum_{l=0}^{2a}l!\binom{2a}{l}^2\Lambda_{n,m}^{(l,a)}(t),\label{reference 2}
\end{eqnarray}
by obvious notation at the last line. 

By the same arguments that has been used in the proof of (\ref{1, H< 3/4}), we have, for $\frac14<H\leq \frac12$, and $l=0$
\begin{equation} 
\Lambda_{n,m}^{(0,a)}(t) \leq C 2^{-\frac{m}{2}} 2^{2nHa-\frac{n}{2}} \sum_{i,i'= -p2^{\frac{n}{2}}}^{p2^{\frac{n}{2}} -1} (2^{-4nHa})\leq C p^22^{-\frac{m}{2}}2^{-n(2Ha-\frac12)} \leq C p^2 2^{-\frac{m}{2}} , \label{lambda1}
\end{equation}
for $l \neq 0$, we have
\[ \Lambda_{n,m}^{(l,a)}(t) \leq C2^{-\frac{m}{2}} 2^{2nHa-\frac{n}{2}}\bigg(2^{-nH(4a-2l)} \sum_{i,i'= -p2^{\frac{n}{2}}}^{p2^{\frac{n}{2}} -1}|\langle \delta_{(i+1)2^{-n/2}}, \delta_{(i'+1)2^{-n/2}} \rangle|^l\bigg).
\]
So,  thanks to (\ref{new delta 1}), we deduce that 
\begin{eqnarray}
\sum_{l=1}^{2a}\Lambda_{n,m}^{(l,a)}(t) &\leq&  C\, p\, 2^{-\frac{m}{2}} 2^{2nHa-\frac{n}{2}}(\sum_{l=1}^{2a}2^{-nH(4a-2l)}2^{n(\frac12 -lH)})\notag\\
&& = C\, p\, 2^{-\frac{m}{2}} (\sum_{l=1}^{2a}2^{-nH(2a-l)})\leq C\,p\,2^{-\frac{m}{2}}. \label{lambda2}
 \end{eqnarray}
 By combining (\ref{reference 2}) with (\ref{lambda1}) and (\ref{lambda2}), we deduce that (\ref{2, H< 3/4}) holds true for $\frac14<H\leq \frac12$.

 \begin{itemize}
\item[(3)] \underline{$2^{-\frac{n}{4}}C_{m,n,p}^{(2r)}(t)\overset{L^2}{\longrightarrow} 0$ as $n\to \infty$, then $m\to \infty$\::}
\end{itemize}
Using (\ref{hermite polynomial pair}), we get
\begin{eqnarray*}
C_{m,n,p}^{(2r)}(t)&=& \sum_{a=1}^r b_{2r,a}\sum_{-p 2^{m/2}+1\leq j\leq p 2^{m/2}}L_t^{j2^{-m/2}}(Y) \sum_{i=(j-1)2^{\frac{n-m}{2}}}^{j2^{\frac{n-m}{2}}-1}
  (\Delta_{i,n} f(X)-\Delta_{j,m} f(X))\\
  &&\hskip6cm \times H_{2a}(X_{i+1}^{(n)} - X_i^{(n)})\\
  \end{eqnarray*}
  \begin{eqnarray*}
  &=&\sum_{a=1}^r b_{2r,a}\sum_{-p 2^{m/2}+1\leq j\leq p 2^{m/2}}L_t^{j2^{-m/2}}(Y) \sum_{i=(j-1)2^{\frac{n-m}{2}}}^{j2^{\frac{n-m}{2}}-1}
  \frac12(f(X_{i2^{-\frac{n}{2}}})-f(X_{j2^{-\frac{m}{2}}}))\\
  &&\hskip6cm \times H_{2a}(X_{i+1}^{(n)} - X_i^{(n)})\\
  && +\sum_{a=1}^r b_{2r,a}\sum_{-p 2^{m/2}+1\leq j\leq p 2^{m/2}}L_t^{j2^{-m/2}}(Y) \sum_{i=(j-1)2^{\frac{n-m}{2}}}^{j2^{\frac{n-m}{2}}-1}
  \frac12(f(X_{(i+1)2^{-\frac{n}{2}}})\\
  &&\hskip6cm -f(X_{(j+1)2^{-\frac{m}{2}}}))H_{2a}(X_{i+1}^{(n)} - X_i^{(n)})\\
  &=& \sum_{a=1}^rb_{2r,a} \big(C_{m,n,p,a}^{(1)}(t) + C_{m,n,p,a}^{(2)}(t)\big),
\end{eqnarray*}
with obvious notation. It suffices to prove that for any fixed $p$ and for all $a\in\{1, \ldots,r\}$
\begin{equation}
2^{-\frac{n}{4}}C_{m,n,p,a}^{(2)}(t)\overset{L^2}{\longrightarrow} 0, \label{3, H< 3/4}
\end{equation}
as $n \to \infty$, then $m \to \infty$. By obvious notation, we have
\[C_{m,n,p,a}^{(2)}(t)= \sum_{-p 2^{m/2}+1\leq j\leq p 2^{m/2}}L_t^{j2^{-m/2}}(Y) \sum_{i=(j-1)2^{\frac{n-m}{2}}}^{j2^{\frac{n-m}{2}}-1}\Delta_{i,j}^{n,m}f(X) H_{2a}(X_{i+1}^{(n)} - X_i^{(n)}).\]
Thanks to the independence of $X$ and $Y$, and to the first point of Proposition \ref{local time}, we have
\begin{eqnarray*}
&&E\big[\big(2^{-\frac{n}{4}}C_{m,n,p,a}^{(2)}(t)\big)^2\big]=2^{-\frac{n}{2}}\bigg|\sum_{-p 2^{m/2}+1\leq j,j'\leq p 2^{m/2}}E\big(L_t^{j2^{-m/2}}(Y)L_t^{j'2^{-m/2}}(Y)\big)\\
&& \sum_{i=(j-1)2^{\frac{n-m}{2}}}^{j2^{\frac{n-m}{2}}-1}\sum_{i'=(j'-1)2^{\frac{n-m}{2}}}^{j'2^{\frac{n-m}{2}}-1}
  E\bigg(\Delta_{i,j}^{n,m}f(X)\Delta_{i',j'}^{n,m}f(X)H_{2a}(X_{i+1}^{(n)} - X_i^{(n)})H_{2a}(X_{i'+1}^{(n)} - X_{i'}^{(n)})\bigg)\bigg|\\
  && \leq C 2^{-\frac{n}{2}}\sum_{-p 2^{m/2}+1\leq j,j'\leq p 2^{m/2}}\\
&&  \sum_{i=(j-1)2^{\frac{n-m}{2}}}^{j2^{\frac{n-m}{2}}-1}\sum_{i'=(j'-1)2^{\frac{n-m}{2}}}^{j'2^{\frac{n-m}{2}}-1}
  \bigg|E\bigg(\Delta_{i,j}^{n,m}f(X)\Delta_{i',j'}^{n,m}f(X)H_{2a}(X_{i+1}^{(n)} - X_i^{(n)})H_{2a}(X_{i'+1}^{(n)} - X_{i'}^{(n)})\bigg)\bigg|,
\end{eqnarray*}
by the same arguments that has been used previously for several times, we deduce that
\begin{eqnarray}
&&E\big[\big(2^{-\frac{n}{4}}C_{m,n,p,a}^{(2)}(t)\big)^2 \big]\leq    2^{-n/2} 2^{2nHa} \sum_{-p 2^{m/2}+1\leq j,j'\leq p 2^{m/2}}\sum_{i=(j-1)2^{\frac{n-m}{2}}}^{j2^{\frac{n-m}{2}}-1}\sum_{i'=(j'-1)2^{\frac{n-m}{2}}}^{j'2^{\frac{n-m}{2}}-1}  \notag\\
&& \sum_{l=0}^{2a}l! \binom{2a}{l}^2 \bigg| E\big[ \Delta_{i,j}^{n,m}f(X)\Delta_{i',j'}^{n,m}f(X)I_{4a-2l}(\delta_{(j+1)2^{-n/2}}^{\otimes (2a-l)}\otimes\delta_{(j'+1)2^{-n/2}}^{\otimes (2a-l)})
\big] \bigg| \notag\\
&& \hskip8.7cm \times \big|\langle \delta_{(j+1)2^{-n/2}}; \delta_{(j'+1)2^{-n/2}}\rangle \big|^l \notag \\
&=&  2^{-n/2} 2^{2nHa}\sum_{l=0}^{2a} l!
\binom{2a}{l}^2
O_{n,m}^{l}(t), \label{reference 3}
\end{eqnarray}
with obvious notation. Following the proof of (\ref{long1}), we get that
\begin{itemize}
\item If $l=2a$ then
the term $O_{n,m}^{2a}(t)$ in (\ref{reference 3}) can be bounded by
\begin{eqnarray*}
&&\frac14 \sup_{|x-y| \leq 2^{-m/2}} E(|f(X_x) - f(X_y)|^2)\sum_{i,i'= -p2^{\frac{n}{2}}}^{p2^{\frac{n}{2}} -1}|\langle \delta_{(i+1)2^{-n/2}} ; \delta_{(i'+1)2^{-n/2}} \rangle|^{2a}. 
\end{eqnarray*}
Since $H\leq \frac12$ and thanks to (\ref{new delta 1}), observe that 
 \begin{equation}
 O_{n,m}^{2a}(t) \leq C\, p\, 2^{n(\frac12 - 2Ha)}  \sup_{|x-y| \leq 2^{-m/2}} E(|f(X_x) - f(X_y)|^2). \label{O1}
 \end{equation}
 
\item If $1 \leq l \leq 2a-1$ then, by (\ref{new delta 1}) among other things used in the proof of (\ref{long1}), we have
\begin{eqnarray}
 O_{n,m}^{l}(t) &\leq &C (2^{-nH})^{(4a-2l)}\sum_{i,i'= -p2^{\frac{n}{2}}}^{p2^{\frac{n}{2}} -1}\big|\langle \delta_{(i+1)2^{-n/2}}; \delta_{(i'+1)2^{-n/2}}\rangle \big|^l \notag\\
&\leq& C\,p\, 2^{-nH(4a-2l)}\: 2^{n(\frac12 - lH)}. \label{O2}
\end{eqnarray}

\item If  $l=0$ then 
\begin{equation}
O_{n,m}^{0}(t) \leq C (2^{-nH})^{4a}(2p2^{\frac{n}{2}})^2 \leq C\,p^2\, 2^{-4nHa} 2^n. \label{O3}
\end{equation}
\end{itemize}
By combining (\ref{reference 3}) with (\ref{O1}), (\ref{O2}) and (\ref{O3}), we get
\begin{eqnarray*}
E\big[\big(2^{-\frac{n}{4}}C_{m,n,p,a}^{(2)}(t)\big)^2 \big]&\leq & C \bigg( \sup_{|x-y| \leq 2^{-m/2}} E(|f(X_x) - f(X_y)|^2) + p(\sum_{l=1}^{2a-1} 2^{-nH(2a-l)})\\
 && \hspace{3cm} + p^2 2^{-n(2Ha - \frac12)}\bigg), 
\end{eqnarray*}
it is then clear that, since $\frac14 < H \leq \frac12$, the last quantity converges to 0 as $n \to \infty$ and then $m\to \infty$. Finally, we have proved that (\ref{3,  H< 3/4}) holds true.
 
\begin{itemize}
\item[(4)] \underline{$2^{-\frac{n}{4}}E_{n,p}^{(2r)}(t)\overset{L^2}{\longrightarrow} 0$ as $p\to \infty$, uniformly on $n$\::}
\end{itemize}
Using (\ref{hermite polynomial pair}), we get 
\begin{eqnarray}
 E_{n,p}^{(2r)}(t)&=& \sum_{a=1}^r b_{2r,a}\bigg(
 \sum_{i\geq p 2^{n/2}} \Delta_{i,n} f(X)H_{2a}(X_{i+1}^{(n)} - X_i^{(n)}) \mathcal{L}_{i,n}(t)\notag\\
 && + \sum_{i < -p2^{n/2}} \Delta_{i,n} f(X)H_{2a}(X_{i+1}^{(n)} - X_i^{(n)})\mathcal{L}_{i,n}(t)\bigg)\notag\\
  &=&\sum_{a=1}^rb_{2r,a}E_{n,p,a}^{(2r)}(t), \label{decomposition E}
\end{eqnarray}
with obvious notation at the last line. It suffices to prove that for all $a\in\{1, \ldots,r\}$
\begin{equation}
2^{-\frac{n}{4}}E_{n,p,a}^{(2r)}(t)\underset{p\to\infty}{\overset{L^2}{\longrightarrow}} 0, \label{4, H< 3/4}
\end{equation}
 uniformly on $n$. By the same arguments that has been used previously, we have
\begin{eqnarray}
&&E[(2^{-\frac{n}{4}} E_{n,p,a}^{(2r)}(t))^2] \notag\\
&\leq& 22^{2nHa -\frac{n}{2}}\sum_{l=0}^{2a}l!\binom{2a}{l}^2\sum_{i,i'\geq p 2^{n/2}} \big|E\big[\big\langle D^{4a-2l}\big(\phi_n(i,i')\big), \delta_{(i+1)2^{-n/2}}^{\otimes 2a-l}\otimes\delta_{(i'+1)2^{-n/2}}^{\otimes 2a-l}\big\rangle\big]\big|\notag\\
&& \times |\langle \delta_{(i+1)2^{-n/2}}, \delta_{(i'+1)2^{-n/2}} \rangle|^l  \big|E[\mathcal{L}_{i,n}(t)\mathcal{L}_{i',n}(t) ]\big|\label{first E}\\
&&+ 22^{2nHa -\frac{n}{2}}\sum_{l=0}^{2a}l!\binom{2a}{l}^2\sum_{i,i' < -p2^{n/2}} \big|E\big[\big\langle D^{4a-2l}\big(\phi_n(i,i')\big), \delta_{(i+1)2^{-n/2}}^{\otimes 2a-l}\otimes\delta_{(i'+1)2^{-n/2}}^{\otimes 2a-l}\big\rangle\big]\big|\notag\\
&& \times |\langle \delta_{(i+1)2^{-n/2}}, \delta_{(i'+1)2^{-n/2}} \rangle|^l  \big|E[\mathcal{L}_{i,n}(t)\mathcal{L}_{i',n}(t) ]\big|.\notag
\end{eqnarray} 
It suffices to prove the convergence to 0 of the quantity given in (\ref{first E}). We have, 
\begin{eqnarray*}
&& 2^{2nHa -\frac{n}{2}}\sum_{l=0}^{2a}l!\binom{2a}{l}^2\sum_{i,i'\geq p 2^{n/2}} \big|E\big[\big\langle D^{4a-2l}\big(\phi_n(i,i')\big), \delta_{(i+1)2^{-n/2}}^{\otimes 2a-l}\otimes\delta_{(i'+1)2^{-n/2}}^{\otimes 2a-l}\big\rangle\big]\big|\notag\\
&&\hspace{2cm} \times |\langle \delta_{(i+1)2^{-n/2}}, \delta_{(i'+1)2^{-n/2}} \rangle|^l  \big|E[\mathcal{L}_{i,n}(t)\mathcal{L}_{i',n}(t) ]\big|\notag\\
&& = \sum_{l=0}^{2a}l!\binom{2a}{l}^2 \Omega_{n,p}^{(l,a)}(t), 
\end{eqnarray*}
with obvious notation at the last line. It is enough to prove that, for all $l \in \{0, \ldots,2a\}$:
\begin{equation}
\Omega_{n,p}^{(l,a)}(t)\underset{p\to\infty}{\longrightarrow} 0, \label{4', H<3/4}
\end{equation}
  uniformly on $n$.
By the same arguments that has been used in the proof of (\ref{1, H< 3/4}), for $\frac14<H\leq \frac12$, we have

{\bf \underline{For $l=0:$}}
\begin{equation*} 
\Omega_{n,p}^{(0,a)}(t) \leq C  2^{2nHa-\frac{n}{2}}2^{-4nHa} \sum_{i,i'\geq p 2^{n/2}}\big|E[\mathcal{L}_{i,n}(t)\mathcal{L}_{i',n}(t) ]\big|. 
\end{equation*}
By the third point of Proposition \ref{local time}, we have
\[ |\mathcal{L}_{i,n}(t)| \leq L_t^{i2^{-n/2}}(Y)+2Kn2^{-n/4}\sqrt{L_t^{i2^{-n/2}}(Y)}\] so that
\begin{equation*} 
E\big[ \mathcal{L}_{i,n}(t)^2\big] \leq 2E\big[ L_t^{i2^{-n/2}}(Y)^2\big] + 8n^2 2^{-n/2}\|K^2\|_2\|L_t^{i2^{-n/2}}(Y)\|_2,
\end{equation*}
which implies
\begin{equation}\label{local1}
\|\mathcal{L}_{i,n}(t)\|_2 \leq C\|L_t^{i2^{-n/2}}(Y)\|_2 + Cn 2^{-n/4}\|L_t^{i2^{-n/2}}(Y)\|_2^{\frac12}.
\end{equation}
On the other hand, thanks to the point 1 of Proposition \ref{local time}, we have
\begin{equation}\label{local2}
 E\big[ L_t^{i2^{-n/2}}(Y)^2\big] \leq C t \exp\big(-\frac{(i2^{-n/2})^2}{2t}\big).
\end{equation}
Consequently, we get
\begin{equation}\label{local3}
\|L_{t}^{i2^{-n/2}}(Y)\|_2 \leq C t^{1/2} \exp\big(-\frac{(i2^{-n/2})^2}{4t}\big).
\end{equation}
By combining (\ref{local1}) with (\ref{local2}) and (\ref{local3}), we deduce that
\begin{eqnarray}\label{local4}
\|\mathcal{L}_{i,n}(t)\|_2 &\leq& C \exp\big(-\frac{(i2^{-n/2})^2}{4t}\big) + Cn 2^{-n/4}\exp\big(-\frac{(i2^{-n/2})^2}{8t}\big)\notag\\
&\leq & C \exp\big(-\frac{(i2^{-n/2})^2}{4t}\big) + C\exp\big(-\frac{(i2^{-n/2})^2}{8t}\big).
\end{eqnarray}
Observe that, by Cauchy-Schwarz inequality, we have
\begin{eqnarray*} 
&&\Omega_{n,p}^{(0,a)}(t) \leq C  \big(2^{-\frac{n}{2}}\sum_{i\geq p 2^{n/2}}\|\mathcal{L}_{i,n}(t)\|_2\big)\big(2^{-2nHa}\sum_{i'\geq p 2^{n/2}}\|\mathcal{L}_{i',n}(t)\|_2\big).
\end{eqnarray*}
Thanks to (\ref{local4}), we get
\begin{align*}
 2^{-\frac{n}{2}}\sum_{i \geq p2^{n/2}}\|\mathcal{L}_{i,n}(t) \|_2&\leq C  2^{-n/2}\sum_{i \geq p2^{n/2}}\exp\big(-\frac{(i2^{-n/2})^2}{4t}\big)\\
&\hskip0.4cm + C 2^{-n/2}\sum_{i \geq p2^{n/2}}\exp\big(-\frac{(i2^{-n/2})^2}{8t}\big).
\end{align*}
But, for $k \in \{4,8\},$
\[\displaystyle{2^{-n/2}\sum_{i \geq p2^{n/2}}\exp\big(-\frac{(i2^{-n/2})^2}{kt}\big) \leq  \int_{p-1}^{\infty}\exp\big(\frac{-x^2}{kt}\big) dx  }.\]
On the other hand, since $H>\frac14$, we have
\[
2^{-2nHa}\sum_{i'\geq p 2^{n/2}}\|\mathcal{L}_{i',n}(t)\|_2 \leq 2^{-n(2Ha-\frac12)}2^{-\frac{n}{2}}\sum_{i'\geq p 2^{n/2}}\|\mathcal{L}_{i',n}(t)\|_2 \leq C 2^{-\frac{n}{2}}\sum_{i'\geq p 2^{n/2}}\|\mathcal{L}_{i',n}(t)\|_2 \]
Finally, we deduce that
\begin{equation}
\Omega_{n,p}^{(0,a)}(t) \leq C\bigg(\int_{p-1}^{\infty}\exp\big(\frac{-x^2}{4t}\big) dx + \int_{p-1}^{\infty}\exp\big(\frac{-x^2}{8t}\big) dx\bigg)^2\underset{p \rightarrow \infty}{\longrightarrow} 0, \label{omega1}
\end{equation}
uniformly on $n$.
 
{\bf \underline{For $l\neq 0:$}} By the same arguments that has been used in the proof of (\ref{1, H< 3/4}) and thanks to (\ref{inner product 2}),  the Cauchy-Schwarz inequality and (\ref{local4}), we have
\begin{eqnarray} 
\Omega_{n,p}^{(l,a)}(t) &\leq& C 2^{2nHa-\frac{n}{2}}\bigg(2^{-nH(4a-2l)}\sum_{i,i'\geq p 2^{n/2}} |\langle \delta_{(i+1)2^{-n/2}}, \delta_{(i'+1)2^{-n/2}} \rangle|^l\big|E[\mathcal{L}_{i,n}(t)\mathcal{L}_{i',n}(t) ]\big|\bigg)\notag\\
&\leq& C 2^{2nHa-\frac{n}{2}}2^{-nH(4a-2l)}2^{-nHl}\big(\sum_{i,i'\geq p 2^{n/2}}|\rho{(i-i')}|^l\|\mathcal{L}_{i,n}(t)\|_2\|\mathcal{L}_{i',n}(t)\|_2\big)\notag\\
&\leq& C2^{-nH(2a-l)}\big(2^{-\frac{n}{2}}\sum_{i\geq p 2^{n/2}}\|\mathcal{L}_{i,n}(t)\|_2\big)\big(\sum_{a\in\Z}|\rho{(a)}|^l\big)\notag\\
&\leq& C2^{-\frac{n}{2}}\sum_{i\geq p 2^{n/2}}\|\mathcal{L}_{i,n}(t)\|_2\notag\\
&\leq &C\bigg(\int_{p-1}^{\infty}\exp\big(\frac{-x^2}{4t}\big) dx + \int_{p-1}^{\infty}\exp\big(\frac{-x^2}{8t}\big) dx\bigg)\underset{p \rightarrow \infty}{\longrightarrow} 0, \label{omega2}
\end{eqnarray}
uniformly on $n$, and  we have the fourth inequality because , since $H\leq \frac12\leq 1 -\frac{1}{2l}$, $\sum_{a\in \Z} |\rho(a)|^l< \infty.$ By combining (\ref{omega1}) and (\ref{omega2}), we deduce that (\ref{4', H<3/4}) holds true for $\frac14 < H\leq \frac12$.

\begin{itemize}
\item[(5)] \underline{The convergence in law of $D_{m,n,p}^{(2r)}(t)$ as $n\to \infty$, then $m \to \infty$, then $p \to \infty$\,:}
\end{itemize}
Let us prove that 
\begin{equation}
\big(2^{-\frac{n}{4}}D_{m,n,p}^{(2r)}(t)\big)_{t\geq 0} \overset{f.d.d.}{\longrightarrow}\big( \gamma_{2r}\int_{-\infty}^{+\infty} f(X_s)L_t^s(Y)dW_s \big)_{ t\geq 0}, \label{5, H< 3/4}
\end{equation}
as $n\to \infty$, then $m\to \infty$, then $p\to \infty$, where $\gamma_{2r}$ and $\int_{-\infty}^{+\infty} f(X_s)L_t^s(Y)dW_s$ are defined in the point (3) of Theorem \ref{thm1}.  In fact, using the decomposition (\ref{hermite polynomial pair}), we have
\begin{eqnarray*}
2^{-\frac{n}{4}}D_{m,n,p}^{(2r)}(t)&=& 2^{-\frac{n}{4}}
 \sum_{-p 2^{m/2}+1\leq j\leq p 2^{m/2}}\Delta_{j,m} f(X) L_t^{j2^{-m/2}}(Y) \sum_{i=(j-1)2^{\frac{n-m}{2}}}^{j2^{\frac{n-m}{2}}-1}\big[(X_{i+1}^{(n)} - X_i^{(n)})^{2r}\\
  &&\hskip10cm   - \mu_{2r}\big]\\
  &=& 2^{-\frac{n}{4}}
 \sum_{-p 2^{m/2}+1\leq j\leq p 2^{m/2}}\Delta_{j,m} f(X) L_t^{j2^{-m/2}}(Y) \sum_{i=(j-1)2^{\frac{n-m}{2}}}^{j2^{\frac{n-m}{2}}-1}\sum_{a=1}^r\\
 && b_{2r,a}H_{2a}(X_{i+1}^{(n)} - X_i^{(n)}).
 \end{eqnarray*}
 It was been proved in (3.27) in \cite{RZ1} that
 \begin{equation*}
  \bigg(2^{-n/4}\sum_{i=(j-1)2^{\frac{n-m}{2}}}^{j2^{\frac{n-m}{2}}-1} H_{2a}(X_{i+1}^{(n)} - X_i^{(n)}),\, 1 \leq a \leq r : -p2^{m/2}+1 \leq j \leq p2^{m/2} \bigg) \overset{ \rm law}{\longrightarrow}
  \end{equation*}
  \[\hskip1cm \bigg( \alpha_{2a}(B_{(j+1)2^{-m/2}}^{(a)} - B_{j2^{-m/2}}^{(a)}),\, 1 \leq a \leq r : -p2^{m/2}+1 \leq j \leq p2^{m/2}\bigg)\]
 where $(B^{(1)}, \ldots , B^{(r)})$ is a $r$-dimensional two-sided Brownian motion and $\alpha_{2a}$ is defined in (\ref{alpha}). Since for any $x\in \R$, $E[ X_x H_{2a}(X_{j+1}^{n,\pm} - X_j^{n,\pm})] = 0$ (Hermite polynomials of different orders are orthogonal), and thanks to the independence between $X$ and $Y$, Peccati-Tudor Theorem (see, e.g., \cite[Theorem $6.2.3$]{NP2}) applies and yields
 \begin{equation*}
  \bigg(X_x, Y_y,2^{-n/4}\sum_{i=(j-1)2^{\frac{n-m}{2}}}^{j2^{\frac{n-m}{2}}-1} H_{2a}(X_{i+1}^{(n)} - X_i^{(n)}),\, 1 \leq a \leq r : -p2^{m/2}+1 \leq j \leq p2^{m/2} \bigg)_{x, y \in \R} \overset{ \rm f.d.d.}{\longrightarrow}
  \end{equation*}
  \[\hskip1cm \bigg(X_x, Y_y, \alpha_{2a}(B_{(j+1)2^{-m/2}}^{(a)} - B_{j2^{-m/2}}^{(a)}),\, 1 \leq a \leq r : -p2^{m/2}+1 \leq j \leq p2^{m/2}\bigg)_{x,y \in \R}\]
 where $(B^{(1)}, \ldots , B^{(r)})$ is a $r$-dimensional two-sided Brownian motion independent of $X$ and $Y$. Hence, for any fixed $m$ and $p$, we have
 \begin{equation}
 \big(2^{-\frac{n}{4}}D_{m,n,p}^{(2r)}(t)\big)_{t\geq 0} \underset{n\to \infty}{\overset{f.d.d.}{\longrightarrow}}\gamma_{2r}\bigg(\sum_{-p 2^{m/2}+1\leq j\leq p 2^{m/2}}\Delta_{j,m} f(X) L_t^{j2^{-m/2}}(Y)\big(W_{(j+1)2^{-m/2}} - W_{j2^{-m/2}}\big)\bigg)_{t\geq 0},
 \label{D1}
 \end{equation}
where  $\gamma_{2r}:= \sqrt{\sum_{a=1}^rb_{2r,a}^2 \alpha_{2a}^2}$ and $W$ is a two-sided Brownian motion. Fix $t\geq0$, observe that
\begin{eqnarray}
&&\sum_{-p 2^{m/2}+1\leq j\leq p 2^{m/2}}\Delta_{j,m} f(X) L_t^{j2^{-m/2}}(Y)\big(W_{(j+1)2^{-m/2}} - W_{j2^{-m/2}}\big)\label{D2}\\
&=&\sum_{-p 2^{m/2}+1\leq j\leq p 2^{m/2}} f(X_{j2^{-\frac{m}{2}}}) L_t^{j2^{-m/2}}(Y)\big(W_{(j+1)2^{-m/2}} - W_{j2^{-m/2}}\big) \notag\\
&&+ \sum_{-p 2^{m/2}+1\leq j\leq p 2^{m/2}} \frac12 \big(f(X_{(j+1)2^{-\frac{m}{2}}})-f(X_{j2^{-\frac{m}{2}}})\big) L_t^{j2^{-m/2}}(Y)\big(W_{(j+1)2^{-m/2}} - W_{j2^{-m/2}}\big)\notag\\
&=& \sum_{-p 2^{m/2}+1\leq j\leq p 2^{m/2}} f(X_{j2^{-\frac{m}{2}}}) L_t^{j2^{-m/2}}(Y)\big(W_{(j+1)2^{-m/2}} - W_{j2^{-m/2}}\big) \notag\\
&&+ N_{m,p}(t), \notag
\end{eqnarray} 
with obvious notation at the last line. Since $E[\int_{-\infty}^{+\infty}\big(f(X_s)L^{s}_{t}(Y)\big)^2ds]\leq C\int_{-\infty}^{+\infty}E[(L^{s}_{t}(Y))^2]ds$ $\leq C\int_{-\infty}^{+\infty}\exp{(\frac{-s^2}{2t})}ds < \infty$, where we have the second inequality by the point 1 of Proposition \ref{local time}, and thanks to the independence between $(X,Y)$ and $W$ and the a.s. continuity of $s\to f(X_s)$ and $s\to L_t^{s}(Y)$, we deduce that
\begin{eqnarray}
&&\sum_{-p 2^{m/2}+1\leq j\leq p 2^{m/2}} f(X_{j2^{-\frac{m}{2}}}) L_t^{j2^{-m/2}}(Y)\big(W_{(j+1)2^{-m/2}} - W_{j2^{-m/2}}\big)\notag\\
&&\underset{m\to \infty}{\overset{L^2}{\longrightarrow}}\int_{-p}^{+p}f(X_{s}) L_t^{s}(Y)dW_s \underset{p\to \infty}{\overset{L^2}{\longrightarrow}}\int_{-\infty}^{+\infty}f(X_{s}) L_t^{s}(Y)dW_s. \label{D3}
\end{eqnarray}
Now, let us prove that, for any fixed $p$,
\begin{equation}
N_{m,p}(t) \underset{m\to \infty}{\overset{L^2}{\longrightarrow}}0. \label{D4}
\end{equation} 
In fact, since  $f(X_{(j+1)2^{-\frac{m}{2}}})-f(X_{j2^{-\frac{m}{2}}}) = f'(X_{\theta_j})(X_{(j+1)2^{-\frac{m}{2}}}-X_{j2^{-\frac{m}{2}}})$ where $\theta_j$ is a random real number satisfying $j2^{-\frac{m}{2}}< \theta_j <(j+1)2^{-\frac{m}{2}}$, and thanks to the independence of $X$ , $Y$ and $W$, the independence of the increments of $W$, and the point 1 of Proposition \ref{local time},  we have
\begin{eqnarray*}
 E[(N_{m,p}(t))^2] &=& \frac14\sum_{-p 2^{m/2}+1\leq j,j'\leq p 2^{m/2}} E\big[(f(X_{(j+1)2^{-\frac{m}{2}}})-f(X_{j2^{-\frac{m}{2}}}))\\
&& \times(f(X_{(j'+1)2^{-\frac{m}{2}}})-f(X_{j'2^{-\frac{m}{2}}})) L_t^{j2^{-m/2}}(Y)L_t^{j'2^{-m/2}}(Y)]\\
&& \times E\big[\big(W_{(j+1)2^{-m/2}} - W_{j2^{-m/2}}\big)\big(W_{(j'+1)2^{-m/2}} - W_{j'2^{-m/2}}\big)\big]\\
\end{eqnarray*}
\begin{eqnarray*}
&=& \frac{2^{-\frac{m}{2}}}{4}\sum_{-p 2^{m/2}+1\leq j\leq p 2^{m/2}} E\big[(f'(X_{\theta_j})(X_{(j+1)2^{-\frac{m}{2}}}-X_{j2^{-\frac{m}{2}}}))^2\big] E\big[(L_t^{j2^{-m/2}}(Y))^2\big]\\
&\leq& C 2^{-\frac{m}{2}}\sum_{-p 2^{m/2}+1\leq j\leq p 2^{m/2}} E\big[(X_{(j+1)2^{-\frac{m}{2}}}-X_{j2^{-\frac{m}{2}}})^2\big]\\
&=& C 2^{-mH}2^{-\frac{m}{2}}2p2^{\frac{m}{2}}=Cp2^{-mH} \underset{m\to \infty}{\longrightarrow}0.
\end{eqnarray*} 
Thus (\ref{D4}) holds true. Thanks to (\ref{D1}), (\ref{D2}), (\ref{D3}) and (\ref{D4}), we deduce that (\ref{5, H< 3/4}) holds true.

Finally, by combining (\ref{1, H< 3/4}) with (\ref{2, H< 3/4}), (\ref{3, H< 3/4}), (\ref{4, H< 3/4}) and (\ref{5, H< 3/4}), we deduce that (\ref{fdd thm2}) holds true.

\section{Proof of Lemma \ref{tech-lemma}}

\begin{enumerate}

\item[1.]  We have, $\langle \varepsilon_u^{\otimes q}, \delta_{(j+1)2^{-n/2}}^{\otimes q} \rangle_{\mathscr{H}^{\otimes q}}= \langle \varepsilon_u, \delta_{(j+1)2^{-n/2}} \rangle_\mathscr{H}^q. $ 
Thanks to (\ref{inner product}), we have 
\[
 \langle \varepsilon_u, \delta_{(j+1)2^{-n/2}} \rangle_\mathscr{H} = E\big( X_{u}(X_{(j+1)2^{-n/2}}-X_{j2^{-n/2}})\big).
 \]
Observe that, for all $0 \leq s \leq t$ and $u\in\R$,
 \[ 
 E\big( X_{u}(X_{t}-X_{s})\big) = \frac{1}{2}\big( t^{2H} -s^{2H}\big) + \frac{1}{2}\big(|s-u|^{2H} -|t-u|^{2H}\big).
 \]
Since for $H \leq 1/2$  one has $|b^{2H} - a^{2H}|\leq |b-a|^{2H}$ for any $a, b \in \mathbb{R}_{+}$,  we immediately deduce (\ref{0}).

\item[2.] By (\ref{inner product}), for all $j, j'\in \{0, \ldots,  \lfloor 2^{n/2}t\rfloor -1\}$,
\begin{eqnarray}
&& |\langle \varepsilon_{j2^{-n/2}}, \delta_{(j'+1)2^{-n/2}} \rangle_\mathscr{H} |=\big|E[X_{j2^{-n/2}}(X_{(j'+1)2^{-n/2}}-X_{j'2^{-n/2}})]\big|\notag\\
&=& \big|2^{-nH -1}( |j'+1|^{2H} - |j'|^{2H}) + 2^{-nH-1}( |j-j'|^{2H} - |j-j'-1|^{2H}) \big|\notag\\
&\leq & 2^{-nH-1}\big| |j'+1|^{2H} - |j'|^{2H}\big| + 2^{-nH-1}\big| |j-j'|^{2H} - |j-j'-1|^{2H}\big| \label{lemme-1}.
\end{eqnarray}
We consider the function $f: [a,b] \to \R$ defined by
\[
f(x)= |x|^{2H}.
\]
Applying the mean value theorem to $f$, we have that
\begin{equation}
||b|^{2H} - |a|^{2H}| \leq 2H (|a|\vee|b|)^{2H-1}|b-a| \leq 2 (|a|\vee|b|)^{2H-1}|b-a|. \label{mean}
\end{equation}
We deduce from (\ref{mean}) that

$2^{-nH-1}||j'+1|^{2H} - |j'|^{2H}| \leq 2^{-nH}|j'+1|^{2H-1}\leq 2^{-nH}| \lfloor 2^{n/2}t\rfloor |^{2H-1}\leq 2^{-n/2}t^{2H-1}$,\\
similarly we have,\\
$2^{-nH-1}||j-j'|^{2H} - |j-j'-1|^{2H}| \leq 2^{-nH}| \lfloor 2^{n/2}t\rfloor |^{2H-1}\leq 2^{-n/2}t^{2H-1}$.

Combining the last two inequalities with (\ref{lemme-1}), and since $\langle \varepsilon_{j2^{-n/2}}^{\otimes q}, \delta_{(j'+1)2^{-n/2}}^{\otimes q} \rangle_{\mathscr{H}^{\otimes q}} = \langle \varepsilon_{j2^{-n/2}}, \delta_{(j'+1)2^{-n/2}} \rangle_\mathscr{H}^q$, we deduce that (\ref{1}) holds true. The proof of (\ref{1'}) may be done similarly.

\item [3.] By (\ref{inner product}) we have
\begin{eqnarray*}
&&|\langle \delta_{(k+1)2^{-n/2}} ; \delta_{(l+1)2^{-n/2}}\rangle_\mathscr{H}|^r = \big|E[(X_{(k+1)2^{-n/2}}-X_{k2^{-n/2}})(X_{(l+1)2^{-n/2}}-X_{l2^{-n/2}})] \big|^r\\
& =& \big|2^{-nH-1}(|k-l+1|^{2H} + |k-l-1|^{2H} -2|k-l|^{2H})\big|^r = 2^{-nrH}|\rho(k-l)|^r,
\end{eqnarray*}
where we have the last equality by the notation (\ref{rho}). So, we deduce that
\begin{eqnarray}
&&\sum_{k,l=0}^{\lfloor 2^{n/2} t \rfloor - 1}|\langle \delta_{(k+1)2^{-n/2}} ; \delta_{(l+1)2^{-n/2}}\rangle_\mathscr{H}|^r = 2^{-nrH} \sum_{k,l=0}^{\lfloor 2^{n/2} t \rfloor - 1} |\rho(k-l)|^r \notag\\
&=& 2^{-nrH} \sum_{k=0}^{\lfloor 2^{n/2} t \rfloor - 1}\sum_{p=k -\lfloor 2^{n/2} t \rfloor +1}^{k}|\rho(p)|^r \notag\\
& =& 2^{-nrH} \sum_{p=1 -\lfloor 2^{n/2} t \rfloor}^{\lfloor 2^{n/2} t \rfloor -1}|\rho(p)|^r \big((p+ \lfloor 2^{n/2} t \rfloor)\wedge \lfloor 2^{n/2} t \rfloor - p \vee 0\big)\notag\\
&\leq& 2^{-nrH}\lfloor 2^{n/2} t \rfloor \sum_{p=1 -\lfloor 2^{n/2} t \rfloor}^{\lfloor 2^{n/2} t \rfloor -1}|\rho(p)|^r \leq 2^{n(\frac12 -rH)}t\sum_{p=1 -\lfloor 2^{n/2} t \rfloor}^{\lfloor 2^{n/2} t \rfloor -1}|\rho(p)|^r, \label{lemme-2}
\end{eqnarray}
where we have the second equality by the change of variable $p=k-l$ and the third equality by a Fubini argument. Observe that $|\rho(p)|^r \sim (H(2H-1))^r p^{(2H-2)r}$ as $p \to +\infty$. So, we deduce that 
\begin{enumerate}

\item \underline{if $H< 1 - \frac{1}{2r}:$}  $\sum_{p \in\Z}|\rho(p)|^r < \infty$, by combining this fact with (\ref{lemme-2}) we deduce that (\ref{2}) holds true.

\item \underline{If $H = 1- \frac{1}{2r}:$} $|\rho(p)|^r \sim \frac{(H(2H-1))^r}{p}$ as $p \to +\infty$. So, we deduce that there exists a constant $C_{H,r} >0$ independent of $n$ and $t$ such that for all integer $n \geq 1$ and all $t \in \R_+$
\begin{eqnarray*}
&&\sum_{p=1 -\lfloor 2^{n/2} t \rfloor}^{\lfloor 2^{n/2} t \rfloor -1}|\rho(p)|^r \leq C_{H,r} \big( 1+ \sum_{p=2}^{\lfloor 2^{n/2} t \rfloor } \frac{1}{p}\big) \leq C_{H,r} \big(1+ \int_1^{2^{n/2}t} \frac{1}{x}dx\big)\\
&=& C_{H,r} \big(1+ \frac{n\log(2)}{2} + \log(t)\big)\leq C_{H,r} \big(1+ n + t\big).
\end{eqnarray*}
By combining this last inequality with (\ref{lemme-2}) we deduce that (\ref{3}) holds true.

\item \underline{If $H> 1- \frac{1}{2r}:$} $|\rho(p)|^r \sim \frac{(H(2H-1))^r}{p^{(2-2H)r}}$ as $p \to +\infty$ where $0 <(2-2H)r<1$.  So, we deduce that there exists a constant $K_{H,r} >0$ independent of $n$ and $t$ such that for all integer $n \geq 1$ and all $t \in \R_+$
\begin{eqnarray*}
&&\sum_{p=1 -\lfloor 2^{n/2} t \rfloor}^{\lfloor 2^{n/2} t \rfloor -1}|\rho(p)|^r \leq K_{H,r} \big( 1+ \sum_{p=1}^{\lfloor 2^{n/2} t \rfloor } \frac{1}{p^{(2-2H)r}}\big) \leq K_{H,r} \big(1+ \int_0^{2^{n/2}t} \frac{1}{x^{(2-2H)r}}dx\big)\\
&=&K_{H,r}\big( 1+ \frac{2^{\frac{n}{2}(1-(2-2H)r)} t^{1-(2-2H)r}}{1-(2-2H)r}\big) \leq C_{H,r}\big(1+2^{\frac{n}{2}(1-(2-2H)r)} t^{1-(2-2H)r}\big),
\end{eqnarray*}
where $C_{H,r}= K_{H,r} \vee \frac{K_{H,r}}{1-(2-2H)r}$. 
By combining the last inequality with (\ref{lemme-2}) we deduce that (\ref{4}) holds true.
\end{enumerate}

\item[4.] As it has been proved in (\ref{lemme-1}), we have
\begin{eqnarray*}
&& |\langle \varepsilon_{k2^{-n/2}}, \delta_{(l+1)2^{-n/2}} \rangle_\mathscr{H} |=\big|E[X_{k2^{-n/2}}(X_{(l+1)2^{-n/2}}-X_{l2^{-n/2}})]\big|\\
&\leq & 2^{-nH-1}\big| |l+1|^{2H} - |l|^{2H}\big| + 2^{-nH-1}\big| |k-l|^{2H} - |k-l-1|^{2H}\big|,
\end{eqnarray*}
so, by a telescoping argument we get
\begin{eqnarray}
&&\sum_{k,l=0}^{\lfloor 2^{n/2} t \rfloor - 1}|\langle \varepsilon_{k2^{-n/2}} ; \delta_{(l+1)2^{-n/2}}\rangle_\mathscr{H}|\notag\\
& \leq& 2^{\frac{n}{2}-1} t^{2H+1} + 2^{-nH -1}\sum_{k,l=0}^{\lfloor 2^{n/2} t \rfloor - 1}\big||k-l|^{2H} - |k-l-1|^{2H}\big|, \label{lemme-3}
\end{eqnarray}
by using the change of variable $p=k-l$ and a Fubini argument, among other things that has been used in the previous proof, we deduce that 
\[
2^{-nH -1}\sum_{k,l=0}^{\lfloor 2^{n/2} t \rfloor - 1}\big||k-l|^{2H} - |k-l-1|^{2H}\big| \leq  2^{\frac{n}{2}} t^{2H+1}.
\]
By combining this last inequality with (\ref{lemme-3}) we deduce that (\ref{5}) holds true. The proof of (\ref{6}) may be done similarly.
\end{enumerate}

\bigskip

{\bf Acknowledgments}. We are thankful to the referees for their careful reading of the original manuscript and for a number of suggestions. The financial support of the DFG (German Science Foundations) Research Training Group 2131 is gratefully acknowledged.

\end{document}